\documentclass[12pt]{amsart}

\usepackage[margin=2.5cm]{geometry}  

\usepackage{amssymb}
\usepackage{amsmath}
\usepackage{amsfonts}
\usepackage[all]{xy} % For commutative diagrams

\newcommand{\disk}{\ensuremath{\mathbb{D}} } % unit disk
\newcommand{\sphere}{\bar{\Bbb{C}}} %Riemann sphere
\newcommand{\riem}{\Sigma}  %Riemann surface
\newcommand{\teich}{Teichm\"uller }
\renewcommand{\Bbb}[1]{\ensuremath{\mathbb{#1}}}
\newcommand{\st}{\, | \,} % such that

\newcommand{\Oqc}{\mathcal{O}^{\mathrm{qc}}} % Oqc
\newcommand{\qs}{\operatorname{QS}}
\newcommand{\aoneinfinity}{A_1^\infty(\mathbb{D})}
\newcommand{\aonetwo}{A_1^2(\mathbb{D})}

\newcommand{\Oqco}{\Oqc_0}
\newcommand{\ttildep}{\widetilde{T}}
\newcommand{\ttildeop}{\widetilde{T}_0}

\newcommand{\qso}{\operatorname{QS}_0}

\newcommand{\qco}{\operatorname{QC}_0}
\newcommand{\rigo}{\operatorname{Rig}_0}
\newcommand{\rig}{\operatorname{Rig}}
\newcommand{\pmodi}{\operatorname{PModI}}

\newcommand{\db}{\operatorname{DB}}

\newcommand{\di}{\operatorname{DI}}

\newcommand{\pdb}{\operatorname{\Pi}}
\newcommand{\pdbo}{\operatorname{\Pi_0}}

\theoremstyle{plain}
        \newtheorem{theorem}{Theorem}[section]
        \newtheorem{lemma}[theorem]{Lemma}
        \newtheorem{proposition}[theorem]{Proposition}
        \newtheorem{corollary}[theorem]{Corollary}

\theoremstyle{definition}
        \newtheorem{definition}[theorem]{Definition}

\theoremstyle{remark}
    \newtheorem{remark}[theorem]{Remark}

\numberwithin{equation}{section} % Equation labels are 'section'.'eq #'

\title[Refined Teichm\"uller space of bordered surfaces]{A Hilbert manifold structure on the refined Teichm\"uller space of  bordered Riemann surfaces}

\author{David Radnell}
\address{David Radnell \\ Department of Mathematics and Statistics \\
American University of Sharjah \\
PO Box 26666, Sharjah \\ United Arab Emirates} \email{dradnell@aus.edu}

\author{Eric Schippers}
\address{Eric Schippers \\ Department of Mathematics \\
University of Manitoba\\
Winnipeg, Manitoba \\  R3T 2N2 \\ Canada}
\email{eric\_schippers@umanitoba.ca}

\author{Wolfgang Staubach}
\address{Wolfgang Staubach\\ Department of Mathematics\\
Uppsala University\\
Box 480\\ 751 06 Uppsala\\ Sweden}
\email{wulf@math.uu.se}

\thanks{Eric Schippers is
partially supported by the National Sciences and Engineering Research
Council. \\  He
would like to thank Nina Zorboska for several helpful
conversations.}



\begin{document}

\begin{abstract}
 We consider bordered Riemann surfaces which are biholomorphic to compact Riemann surfaces
 of genus $g$ with $n$ regions biholomorphic to the disc removed.
 We define a refined Teichm\"uller space of such Riemann surfaces and demonstrate that in the case that $2g+2-n>0$, this refined Teichm\"uller space is a Hilbert manifold.  The inclusion map from the refined Teichm\"uller space into the usual Teichm\"uller space (which is a Banach manifold) is holomorphic.

 We also show that the rigged moduli space of Riemann surfaces with non-overlapping holomorphic maps, appearing in conformal field theory, is a complex Hilbert manifold.
 This result requires an analytic reformulation of the moduli space, by enlarging the set of non-overlapping mappings to a class of
 maps intermediate between analytically extendible maps and quasiconformally
 extendible maps.  Finally we show that the rigged moduli space is the quotient of the refined Teichm\"uller space
 by a properly discontinuous group of biholomorphisms.
\end{abstract}

\maketitle

\begin{section}{Introduction} \label{Introduction}
In this paper, we construct a refined Teichm\"uller space of bordered Riemann surfaces of genus $g$ with $n$ boundary curves homeomorphic to the circle.  If $2g+2-n>0$ this refined
Teichm\"uller space possesses a Hilbert manifold structure, and furthermore
the inclusion map from this refined Teichm\"uller space into the standard
one is holomorphic.   In brief, the approach can be summarized
as follows: we combine the results of Takhtajan and Teo \cite{Takhtajan_Teo_Memoirs}
and Guo Hui \cite{GuoHui} refining the universal Teichm\"uller space, with the
results of Radnell and Schippers \cite{RS05, RSnonoverlapping, RS_fiber} demonstrating the
relation between
a moduli space in conformal field theory and the Teichm\"uller space of bordered
surfaces.  We also require a result by Nag \cite{NagSchiffer, Nagbook} on the variational method of Gardiner and Schiffer \cite{Gardiner}, together with the theory of marked holomorphic families of Riemann surfaces (see for example \cite{EarleFowler, Nagbook, Hubbard}). The demonstration that the transition
functions of the atlas defining the Hilbert manifold structure are biholomorphisms,
brings us into the realm of Besov spaces and the theory of Carleson measures for analytic Besov spaces.  We also utilize the relationship between the Dirichlet space and the little Bloch space.

Our results are motivated both by
Teichm\"uller theory, where there has been interest in refining Teichm\"uller space
(see below),
and by conformal field theory, where our results are
required to solve certain analytic problems in the  construction of conformal field theory from vertex operator algebras following
Yi-Zhi Huang \cite{Huang}.
First, we give some background for the problem, and then outline our approach.

There have been several refinements of quasiconformal Teichm\"uller space,
obtained by considering natural analytic subclasses either of the quasisymmetries
of the circle or of the quasiconformally extendible univalent functions
in the Bers model of universal Teichm\"uller space.
For example, Astala and Zinsmeister \cite{AstalaZinsmeister} give a model of
the universal Teichm\"uller space based on BMO, and Cui and Zinsmeister \cite{CuiZinsmeister} studied the Teichm\"uller spaces compatible with Fuchsian groups
in this model.  Gardiner and Sullivan \cite{GardinerSullivan} study a refined class
of quasisymmetric mappings (which they call symmetric) and the topology of this
refined class.

A family of refined models of the universal Teichm\"uller space was given by
Guo Hui \cite{GuoHui}, each based on an $L^p$ norm.  These spaces were completely
characterized in three ways: in terms of a space of quadratic differentials,
in terms of univalent functions, and in terms of a space of Beltrami differentials;
all satisfying a weighted $L^p$-type integrability condition.  In this paper, we are
concerned with the $L^2$ case.  Guo Hui attributes the $L^2$ case to a preprint
of Guizhen Cui, which we were unable to locate.
Independently, Takhtajan and Teo \cite{Takhtajan_Teo_Memoirs} defined a Hilbert manifold
structure on the universal Teichm\"uller space and universal Teichm\"uller
curve, equivalent to that of Guo Hui, and
obtained far-reaching results.  These results include (among many others)
obtaining a convergent Weil-Petersson metric and computation of its sectional
curvatures, showing that the Kirillov-Yuri'ev-Nag-Sullivan period matrix is a holomorphic
embedding of the universal Teichm\"uller space, and obtaining equivalent characterizations of elements of their refined
universal Teichm\"uller space
in terms of the generalized Grunsky matrix.

In conformal field theory one considers a moduli space
originating with Friedan and Shenker \cite{FriedanShenker}. We will use two different formulations of this moduli space due to Segal \cite{Segal} and Vafa \cite{Vafa}. Vafa's \textit{puncture model} of the rigged moduli space
consists of equivalence classes of pairs $(\riem,\phi)$,
 where $\riem$ is a compact Riemann surface with $n$
punctures, and $\phi=(\phi_1,\ldots,\phi_n)$ is an $n$-tuple of one-to-one holomorphic maps from the unit disc $\mathbb{D} \subset \mathbb{C}$ into the Riemann surface with non-overlapping images.
Two such pairs $(\riem_1,\phi)$ and $(\riem_2,\psi)$
are equivalent if there is a biholomorphism $\sigma:\riem_1 \rightarrow \riem_2$
such that $\psi_i  = \sigma \circ \phi_i$ for $i=1,\ldots,n$.
The $n$-tuple of maps $(\phi_1,\ldots,\phi_n)$ is called the \text{rigging}, and is
 usually subject to some additional regularity conditions which vary in the
 conformal field theory literature. The choice of these regularity conditions
 relates directly to the analytic structure of this moduli space.
 The regularity also relates directly to the regularity of certain elliptic
 operators, which are necessary for the rigorous definition of conformal field theory in the sense of Segal \cite{Segal}.
 In this paper we show that the rigged moduli space has
 a Hilbert manifold structure, and that this Hilbert manifold structure arises
 naturally from a refined Teichm\"uller space of bordered surfaces, which
 we also show is a Hilbert manifold. These results are further motivated
 by the fact that the aforementioned elliptic operators will have convergent
 determinants on precisely this refined moduli space.  We hope to return
 to this question in a future publication. Moreover, these results will have applications to the construction of higher genus conformal field theory, following a program of Yi-Zhi Huang and others \cite{Huang, HK2010}. Also, it is natural to ask whether there is a convergent natural generalization of the
 Weil-Petersson metric on the refined Teichm\"uller space, as in \cite{Takhtajan_Teo_Memoirs}.  We intend to demonstrate this in a future publication.

These results are made possible by previous work of two of the authors \cite{RSnonoverlapping}, in
which it was shown that if one
chooses the riggings to be extendible to quasiconformal maps of a
neighborhood of the closure of $\mathbb{D}$, then the rigged
moduli space is the same as the Teichm\"uller
space of a bordered Riemann surface (up to a properly discontinuous
group action).  Thus the rigged moduli space
inherits a complex Banach manifold structure from Teichm\"uller space.
This solved certain analytic problems in the
definition of conformal field theory, including holomorphicity of the sewing operation.

 On  the other hand this also provided an alternate description of the Teichm\"uller space
of a bordered surface $\riem$ as a fibre space that is locally modeled on the following rigged \teich space.  In \cite{RS_fiber} (following the first author's thesis \cite{RadThesis}), two of the authors introduced the \textit{rigged \teich space} based on quasiconformally extendible riggings,  which is the analogue of the above rigged moduli space. It was proved that this rigged \teich space is a fibre space: the fibres consist
of non-overlapping maps into a compact Riemann surface with punctures obtained by sewing
copies of the punctured disc onto the boundaries of $\Sigma$.  The base space is the
finite-dimensional Teichm\"uller space of the compact surface with punctures so obtained.

Thus the Teichm\"uller space of bordered surfaces has two independent
complex Banach manifolds structures: the standard one, obtained from the Bers
embedding of spaces of equivalent Beltrami differentials, and one obtained
from the fibre model.  It was shown that the two are equivalent \cite{RSnonoverlapping, RS_fiber}.
Up to normalizations, the fibres look locally like an $n$-fold product of the
universal Teichm\"uller space. We now define a \textit{refined} rigged \teich space and prove that it is a Hilbert manifold by using the results of Guo Hui \cite{GuoHui} and Takhtajan and Teo \cite{Takhtajan_Teo_Memoirs} to define a refined set of fibres that are modeled on Hilbert spaces. Finally, we define a refined \teich space of bordered surfaces and, via the fibre model, show that it is a Hilbert manifold using the refined rigged \teich space. Charts for the refined \teich space will be defined completely explicitly, using Gardiner-Schiffer variation and natural function spaces of non-overlapping maps.

The proof that these charts define a Hilbert manifold structure is somewhat complicated.
We proceed in the following way.  In Section \ref{se:refined_maps}, we define the refined
quasiconformal mappings and function spaces which will appear in the paper.  This section
mostly establishes notation and outlines some previous results, and proves some elementary
facts about the refined mappings.  The difficult work is done
in Sections \ref{se:non-overlapping} and \ref{se:rigged_Teich_space}.  In Section \ref{se:non-overlapping},
we define the set of non-overlapping mappings which serves as a model of the fibres, and show that it
is a complex Hilbert manifold.  In Section \ref{se:rigged_Teich_space},
we show that the refined rigged \teich space is a Hilbert manifold.  We do this
using the results of the previous section, and Gardiner-Schiffer variation. A key part of
the argument relies on the universality properties of the universal Teichm\"uller curve and the theory of marked holomorphic families of Riemann surfaces.
Finally, in Section \ref{se:refined_Teich_space} we show that the refined Teichm\"uller space
of a bordered Riemann surface is a Hilbert manifold, by showing that it covers the refined rigged \teich space and passing the structure upwards.  Furthermore, we show that the Hilbert manifold structure
passes downwards to the two versions of the rigged moduli space of conformal field theory defined by Segal \cite{Segal} and Vafa \cite{Vafa}.
\end{section}
\begin{section}{Refined Quasiconformal maps and quasisymmetries} \label{se:refined_maps}
 In Section \ref{se:collection_refined_results} we collect some known results on the refinement of the set of quasisymmetries and quasiconformal maps, from the work of Takhtajan and Teo
 \cite{Takhtajan_Teo_Memoirs}, Teo \cite{Teo_Velling} and Guo Hui \cite{GuoHui}.
 We also derive two technical lemmas which follow almost directly from
 previous work of two of the authors \cite{RSnonoverlapping}.
 In Section \ref{se:refined_quasisymmetries} we define a refined set of
  quasisymmetries between borders of Riemann surfaces in an obvious way
  and some elementary results are derived.  This is then used to define
 a refined set of quasiconformal maps between Riemann surfaces in Section \ref{se:refined_qc_maps}.
\begin{subsection}{Refined maps on the disc and circle}
\label{se:collection_refined_results}
In this section we collect some necessary results on the refined universal Teichm\"uller space of Takhtajan and Teo \cite{Takhtajan_Teo_Memoirs} and Guo Hui \cite{GuoHui}.
We need a refined class of quasiconformal and quasisymmetric mappings of the
disc and $S^1$.

In \cite{RSnonoverlapping} we defined the set $\Oqc$ of quasiconformally
extendible maps in the following way.
\begin{definition}
 Let $\Oqc$ be the set of maps $f: \mathbb{D} \rightarrow \mathbb{C}$
 such that $f$ is one-to-one, holomorphic, has quasiconformal
 extension to $\mathbb{C}$, and $f(0)=0$.
\end{definition}
A Banach space structure can be introduced on $\Oqc$ as follows.
Let
\begin{equation} \label{eq_aoneinfinity}
 \aoneinfinity = \left\{ \phi \in \mathcal{H}(\mathbb{D}) : \| \phi \|_1^\infty = \sup_{z \in
 \mathbb{D}} (1-|z|^2) |\phi(z)| < \infty \right\}.
\end{equation}
This is a Banach space.
 It follows directly from results of Teo \cite{Teo_Velling} that for
 \[  \mathcal{A}(f)=\frac{f''}{f'}  \]
 the map
\begin{align} \label{eq_chidefinition}
 \chi : \Oqc & \longrightarrow  \aoneinfinity \oplus \mathbb{C} \nonumber \\
 f & \longmapsto  \left( \mathcal{A}(f),f'(0) \right)
\end{align}
takes $\Oqc$ onto an open subset of the Banach space
$\aoneinfinity \oplus \mathbb{C}$ (see \cite{RSnonoverlapping}).  Thus $\Oqc$ inherits a complex
structure from $\aoneinfinity \oplus \mathbb{C}$.

The space $\Oqc$ can be thought of as a two complex dimensional extension of
the universal Teichm\"uller space.  We will construct a Hilbert structure on
a subset of $\Oqc$.  To do this, in place of $A_1^\infty(\mathbb{D})$ we use the Bergman
space
\[  \aonetwo = \left\{ \phi \in \mathcal{H}(\mathbb{D}) : \| \phi\|_2^2 = \iint_{\mathbb{D}}
  |\phi|^2 \, dA < \infty \right\} \]
which is a Hilbert space and a vector subspace of the Banach space
$\aoneinfinity$. Furthermore, the inclusion map from $\aonetwo$
to $\aoneinfinity$ is bounded \cite[Chapter II Lemma 1.3]{Takhtajan_Teo_Memoirs}. Here and in the rest of the paper we shall denote the Bergman space norm $\Vert \cdot\Vert_{A_1^{2}}$ by $\Vert \cdot\Vert.$

We define the class of refined quasiconformally extendible maps as follows.
\begin{definition}  Let
  \[  \Oqco= \left\{f \in \Oqc : \mathcal{A}(f) \in \aonetwo
  \right\}.  \]
\end{definition}
We will embed $\Oqco$ in the Hilbert space direct sum $\mathcal{W} = A_1^2(\mathbb{D})\oplus \mathbb{C}$.
Since $\chi(\Oqc)$ is open, $\chi(\Oqco)= \chi(\Oqc) \cap
\aonetwo$ is also open, and thus $\Oqco$ trivially inherits a Hilbert manifold
structure from $\mathcal{W}$.  We summarize this with the following theorem.
\begin{theorem}  \label{th:Oqco_open_in_Oqc} The inclusion map from
$\aonetwo \rightarrow \aoneinfinity$ is continuous.  Furthermore
$\chi(\Oqco)$ is an open subset of the vector subspace $\mathcal{W} = \aonetwo
\oplus \mathbb{C}$ of $\aoneinfinity \oplus \mathbb{C}$, and the
inclusion map from $\chi(\Oqco)$ to $\chi(\Oqc)$ is holomorphic.
Thus the inclusion map
$\iota: \Oqco \rightarrow \Oqc$ is holomorphic.
\end{theorem}
\begin{remark} Although the inclusion map is continuous, the topology of $\Oqco$ is not
the relative topology inherited from $\Oqc$.  It's enough to show that $A_1^2(\mathbb{D})$
does not have the relative topology from $A_1^\infty(\mathbb{D})$.  To see this observe that
if
\[  f_t = \frac{1}{\sqrt{|\log{(1-t)}|}(1-t^2z^2)}  \]
for $t<1$, then as $t \rightarrow 1$ $\|f_t\| \rightarrow 0$ in $A_1^2(\mathbb{D})$ whereas
$\|f_t\|_{A_1^\infty(\mathbb{D})} \rightarrow \pi/2$.
\end{remark}
\begin{lemma} \label{le:Oqco_composition_preserves}
  Let $f \in \Oqco$.  Let $h$ be a one-to-one holomorphic map
  defined on an open set $W$ containing
  $\overline{f(\mathbb{D})}$.  Then $h \circ f \in \Oqco$.
  Furthermore, there is an open neighborhood $U$ of $f$ in $\Oqco$ and a
  constant $C$ such that
  $\|\mathcal{A}(h \circ g)\| \leq C$ for all $g \in U$.
 \end{lemma}
 \begin{proof} The map
  $h \circ f$ has a quasiconformal extension to
  $\mathbb{C}$ if and only if it has a quasiconformal extension to
  an open neighborhood of $\overline{\mathbb{D}}$ (although
  not necessarily with the same dilatation constant).  Clearly $h
  \circ f$ has a quasiconformal extension to $W$, namely $h$
  composed with the extension of $f$.  Thus $h \circ f$ has an
  extension to the plane, and so $h\circ f \in \Oqc$.

  We need only show that $\mathcal{A}(h \circ f) \in \aonetwo$.
  This follows from Minkowski's inequality:
  \begin{align} \label{eq:Minkowski}
   \left(\iint_{\mathbb{D}} | \mathcal{A}(h \circ f) |^2 dA \right)^{1/2} & \leq 
   \left(\iint_{\mathbb{D}} |\mathcal{A}(h) \circ f \cdot f'|^2 dA \right)^{1/2} +
   \left(\iint_{\mathbb{D}} |\mathcal{A}(f)|^2 dA \right)^{1/2} \\
   & =  \left( \iint_{f(\mathbb{D})}| \mathcal{A}(h)|^2 dA \right)^{1/2} +
   \left(\iint_{\mathbb{D}} |\mathcal{A}(f)|^2 dA \right)^{1/2} \nonumber
  \end{align}
  The first term on the right hand side is finite because $h$ is
  holomorphic and $h' \neq 0$ on an open set containing $\overline{f(\mathbb{D})}$
  so $\mathcal{A}(h)$ is bounded on
  $f(\mathbb{D})$.  The second term is bounded because $f \in
  \Oqco$.  This proves the first claim.

  To prove the second claim, observe that there is a compact set $K$ contained in $W$ which contains $\overline{f(\mathbb{D})}$ in its interior.  By \cite[Corollary 3.5]{RSnonoverlapping}
  there is an open set $\hat{U}$ in $\Oqc$ such that $\overline{g(\mathbb{D})}$ is contained in the
  interior of $K$ for all $g \in \hat{U}$.  Since the inclusion $\iota:\Oqco \rightarrow \Oqc$ is
  continuous, we obtain an open set $\iota^{-1}(\hat{U}) \subset \Oqco$
  with the same property.  Let $U$ be an open ball in $\iota^{-1}(\hat{U})$ containing $f$. There is a constant $C_1$ such that for any  $g\in U$
 \[  \iint_{\mathbb{D}} |\mathcal{A}(g)|^2 \,dA \leq C_1  \]
  and a constant $C_2$ such that
  \[  \iint_{g(\mathbb{D})} |\mathcal{A}(h)|^2 \,dA \leq \iint_K |\mathcal{A}(h)|^2 \,dA \leq C_2 . \]
Applying (\ref{eq:Minkowski}) completes the proof.
 \end{proof}

 We will also need a technical lemma on a certain kind of holomorphicity of left composition in $\Oqco$.
\begin{lemma}
\label{le:CompHoloOqc_manyvar}
 Let $E$ be an open subset of $\mathbb{C}$ containing $0$ and $\Delta$ an open
 subset of $\mathbb{C}$. Let $\mathcal{H}: \Delta \times E \to \mathbb{C}$ be a map which is holomorphic in both variables and let $h_{\epsilon}(z) = \mathcal{H}(\epsilon, z)$. Let $\psi \in \Oqco$ satisfy $\overline{\psi(\mathbb{D})} \subseteq E$.
 Then the map $Q: \Delta \mapsto \Oqc_0$ defined by $Q(\epsilon) = h_\epsilon \circ \psi$ is holomorphic in $\epsilon$.
 \end{lemma}
 \begin{proof}
  We need to show that for fixed $\psi$, $\mathcal{A}(h_\epsilon \circ \psi)$ and $(h_\epsilon \circ \psi)'(0)$ are holomorphic
  in $\epsilon$.  First observe that all the $z$-derivatives of $h_\epsilon$ are holomorphic in $\epsilon$ for fixed $z$.  Thus
  the second claim is immediate.

  To prove holomorphicity of $\epsilon \mapsto \mathcal{A}(h_\epsilon \circ \psi)$,
  it is enough to show weak holomorphicity and local boundedness \cite{GrosseErdmann};
  that is, to show  local boundedness and that for some set of separating continuous
  functionals $\{\alpha\}$ in the dual of the Bergman space, $\alpha \circ \mathcal{A}(h_\epsilon \circ \psi)$ is holomorphic for all $\alpha$.  Let $E_z$ be the
  point evaluation function $E_z \psi = \psi(z)$.  These are continuous on the Bergman space
  and obviously separating on any open set.
  Since
  \[  \mathcal{A}(h_\epsilon \circ \psi)= \mathcal{A}(h_\epsilon) \circ \psi \cdot \psi' + \mathcal{A}(\psi)  \]
  clearly $E_z(\mathcal{A}(h_\epsilon \circ f))$ is holomorphic in $\epsilon$.

  So we only need to prove that $\mathcal{A}(h_\epsilon \circ \psi)$ and $(h_\epsilon \circ \psi)'(0)$ are locally bounded.  The second claim is obvious.  As above, by Minkowski's inequality (\ref{eq:Minkowski}) and a change of variables
  \[   \left(\iint_{\mathbb{D}} | \mathcal{A}(h_\epsilon \circ \psi) |^2 dA \right)^{1/2}  \leq
   \left( \iint_{\psi(\mathbb{D})}| \mathcal{A}(h_\epsilon)|^2 dA \right)^{1/2} +
   \left(\iint_{\mathbb{D}} |\mathcal{A}(\psi)|^2 dA \right)^{1/2}.   \]
  Since $\mathcal{A}(h_\epsilon)$ is jointly holomorphic in $\epsilon$ and $z$
and $\overline{\psi(\mathbb{D})} \subseteq E$ for any fixed $\epsilon_0$, there is a compact set $D$ containing $\epsilon_0$ such that $|\mathcal{A}(h_\epsilon)|$ is bounded on $\psi(\mathbb{D})$ by a constant independent of $\epsilon \in D$.  Since $\mathcal{A}(\psi)$ is in the Bergman
  space this proves the claim.
 \end{proof}

Next, we define a subset $\qso(S^1)$ of the quasisymmetries  in the following way.
 Briefly, a map $h:S^1 \to S^1$ is in $\qso(S^1)$ if the corresponding welding maps are in $\Oqco$.  Let $\mathbb{D}^* = \{z \,: |z|>1 \} \cup \{ \infty\}$.  For $h \in \qs(S^1)$ let $w_{\mu}(h): \disk^* \to \disk^*$  be a quasiconformal extension of $h$ with dilatation $\mu$ (such an extension exists by the Ahlfors-Beurling extension theorem). Furthermore, let $w^{\mu} : \sphere \to \sphere$
 be the quasiconformal map with dilatation $\mu$ on $\disk^*$
 and $0$ on $\disk$, with normalization $w^\mu(0)=0$, ${w^\mu}'(0)=1$ and $w^\mu(\infty)=\infty$ and set
 \[  F(h)=\left. w^\mu \right|_{\mathbb{D}}.  \]  It is a standard fact that $F(h)$ is independent of the choice of extension $w_\mu$.
\begin{definition}
 We define a subset of $\qs(S^1)$ by
$$
\qso(S^1) = \{ h \in \qs(S^1) \,:\,  F(h)  \in \Oqco \}.
$$
\end{definition}
\begin{remark}
 A change in the normalization of ${w^\mu}'(0)$ results in exactly the same set.
\end{remark}

 An alternate characterization of $\Oqco$ follows from a theorem
 proved by Guo Hui \cite{GuoHui}.
 Let
 \[  L^2_{hyp}(\mathbb{D}^*) =\left\{ \mu : \iint_{\mathbb{D}^*}
 (|z|^2-1)^{-2} | \mu(z) |^2 dA < \infty \right\},  \]
 and let
 \[  L^\infty(\mathbb{D}^*)_1 = \left\{ \mu:\mathbb{D}^* \rightarrow \mathbb{C} : \| \mu \|_\infty \leq k \text{ for some } k <1 \right\}   \]
 (that is, the unit ball in $L^\infty(\mathbb{D}^*)$).
 Note that the line element of the hyperbolic metric on $\mathbb{D}$ is $|dz|(1-|z|^2)^{-1}$
 and the line element of the hyperbolic metric on $\mathbb{D}^*$ is $|dz|(|z|^2-1)^{-1}$.
 Thus the above condition says that $\mu$ is $L^2$ with respect to hyperbolic area.
 The following two theorems follow from Theorems 1 and 2 of \cite{GuoHui}.
 \begin{theorem}[Guo Hui] \label{th:Guo_Hui_qs}   Let $f$ be a one-to-one holomorphic function on $\mathbb{D}$
 such that $f(0)=0$. Then $f \in \Oqco$ if and only if there exists a
 quasiconformal extension $\tilde{f}$ of $f$ to $\mathbb{C}$ whose
  dilatation $\mu$ is in $L^2_{hyp}(\mathbb{D}^*) \cap
  L^\infty(\mathbb{D}^*)_1$.
 \end{theorem}
  \begin{theorem}[Guo Hui] \label{th:Guo_Hui_Oqco}
  Let $\phi:S^1 \rightarrow S^1$ be a quasisymmetry.  Then $\phi \in \qso(S^1)$ if and
  only if there is a quasiconformal extension $h:\mathbb{D}^* \rightarrow \mathbb{D}^*$
  of $\phi$ such that the Beltrami differential $\mu(h)$ of $h$ is in $L^2_{hyp}(\mathbb{D}^*)$.
 \end{theorem}

It follows from Theorem 1.12 of Part II and Lemma 3.4 of Part I of \cite{Takhtajan_Teo_Memoirs} that
$\qso(S^1)$ is a group.
\begin{theorem}[Takhtajan-Teo]
\label{th:QSo_group}
 The set $\qso(S^1)$ is closed under composition and inversion.
\end{theorem}

By an analytic map $h:S^1 \rightarrow S^1$ we mean that $h$ is the restriction of an analytic map
of a neighborhood of $S^1$.  Let $\mathbb{A}(r,s)$ denote the annulus $\{ z\,:\, r<|z|<s \}$ and $D(z_0,r)$ denote the disc $\{ z\,:\, |z-z_0|<r\}$.
\begin{proposition} \label{pr:analytic_in_qso}
 If $h:S^1 \rightarrow S^1$ is one-to-one and analytic, then $h$ has a quasiconformal extension to $\mathbb{D}^*$
 which is holomorphic in an annulus $\mathbb{A}(1,R)$ for some $R>1$.  Furthermore $h \in \qso(S^1)$.
\end{proposition}
\begin{proof}
 To prove the first claim, observe
 that $h$ has an analytic extension $\tilde{h}$ to some annulus $\mathbb{A}(r,s)$ for $r<1<s$.  Let $R$ be such that
 $1<R<s$.  Applying the Ahlfors-Beurling extension theorem to the circle $|z|=R$, there exists a quasiconformal map $g:\mathbb{A}(R,\infty) \rightarrow \mathbb{A}(R,\infty)$
 whose boundary values agree with $\tilde{h}$ restricted to $|z|=R$. Let $H$ be the map which is equal to $\tilde{h}$ on $\mathbb{A}(1,R)$ and $g$ on $\mathbb{A}(R,\infty)$. Then $H$
 is  quasiconformal on $\mathbb{D}^*$ since it is quasiconformal on the two pieces and continuous on $\mathbb{D}$ (see \cite[V.3]{Lehto-Virtanen}). Thus, $H$ has the desired properties.

 The second claim follows from Theorem \ref{th:Guo_Hui_qs} since the dilatation of $H$ is zero in $\mathbb{A}(1,R)$.
\end{proof}
\end{subsection}
\begin{subsection}{Refined quasisymmetric mappings between boundaries of Riemann surfaces}
 \label{se:refined_quasisymmetries}
 We first clarify the meaning of ``bordered Riemann surface''.  By a half-disc, we mean a set of
 the form $\{z: |z-z_0| <r \text{ and } \operatorname{Im}(z) \geq 0 \}$ for some $z_0$ on the real axis.
 By a bordered Riemann surface, we mean a Riemann surface with boundary, such that for every point on the
 boundary there is a homeomorphism of a neighborhood of that point onto a half-disc.  It is further assumed
 that for any pair of charts $\rho_1,\rho_2$ whose domains overlap, the map $\rho_2 \circ \rho_1^{-1}$ and its inverse
 is a one-to-one holomorphic map on its domain.  Note that this implies, by the Schwarz reflection principle,
 that $\rho_2 \circ \rho_1^{-1}$ extends to a one-to-one holomorphic map of an open set containing the portion
 of the real axis in the domain of the original map.  Every bordered Riemann surface has a double which is defined in the
 standard way.  See for example \cite{AhlforsSario}.

Following standard terminology (see for example \cite{Nagbook}) we say that a Riemann surface is of \textit{finite topological type} if its fundamental group is finitely generated. A Riemann surface is said to be of finite topological type $(g,n,m)$ if it is biholomorphic to a compact genus $g$ Riemann surface with $n$ points and $m$ parametric disks removed. By a parametric disk we mean a region biholomorphic to the unit disk such that, after its removal, the resultant surface is homeomorphic to a compact surface with a point removed.

 In this paper we will be entirely concerned with Riemann surfaces of type $(g,0,n)$ and $(g,n,0)$ and we will use the following terminology. A
\textit{bordered Riemann surface} of type $(g,n)$ will refer to a bordered Riemann surface of type $(g,0,n)$ and a \textit{punctured Riemann surface} of type $(g,n)$ will refer to a Riemann surface of type $(g,n,0)$. It is furthermore assumed that the boundary curves and punctures are given a numerical ordering. Finally, a \textit{boundary curve} will be understood to mean a connected component of the boundary of a bordered Riemann surface. Note that each boundary curve is homeomorphic to $S^1$.

 \begin{remark} \label{re:extensions}
  Any quasiconformal map between bordered Riemann surfaces has a unique continuous extension taking the boundary curves to the boundary curves.
  To see this let $\riem^B_1$ and $\riem^B_2$ be bordered Riemann surfaces, and let $\riem_1^d$ and $\riem_2^d$
  denote their doubles. By
  reflecting, the quasiconformal map extends to the double: the reflected map is continuous on
  $\riem_1^d$, takes $\riem_1^d$ onto $\riem_2^d$, and is quasiconformal on the double minus the boundary curves.  Since
  each boundary curve of $\riem^B_i$ is an analytic curve in the double, the map is quasiconformal
  on $\riem_1^d$ \cite[V.3]{Lehto-Virtanen} and in particular continuous on each analytic curve.

  Throughout the paper, we will label the original map and its continuous extension with the same letter to avoid complicating
  the notation.
  When referring to a ``bordered Riemann surface'', we will be referring to the interior.  However,
 in the following all maps between bordered Riemann surfaces will be at worst quasiconformal and thus
 by Remark \ref{re:extensions} have unique continuous extensions to the
 boundary. Thus the reader could treat the border as included in the Riemann surface with only
 trivial changes to the statements in the rest of the paper.
 \end{remark}

\begin{definition} \label{de:collar_nbhd}
  Let $\riem^B$ be a bordered Riemann surface and $C$ be one of its boundary components.  A {\it collared neighborhood of $C$}
  is an open set $U$ which is biholomorphic to an annulus, and one of whose boundary curves is $C$.  A {\it collared chart of $C$}
  is a biholomorphism $H:U \rightarrow \mathbb{A} (1,r)$ where $U$ is a collared neighborhood of $C$,
  whose continuous extension to $C$ maps $C$ to $S^1$.
 \end{definition}
 Note that any collared chart must have a continuous one-to-one extension to $C$, which maps $C$ to $S^1$.  (In fact application of the
 Schwarz reflection principle shows that $H$ must have a one-to-one holomorphic extension to an open tubular neighborhood of $C$
 in the double of $\riem$.)
 We may now define the class of refined quasisymmetries between boundary curves of bordered Riemann surfaces.
 \begin{definition} \label{de:refined_qs_surfaces}
  Let $\riem^B_1$ and $\riem^B_2$ be bordered Riemann surfaces, and let $C_1$ and $C_2$ be boundary curves of $\riem^B_1$
  and $\riem^B_2$ respectively.  Let $\qso(C_1,C_2)$ denote the set of orientation-preserving
  homeomorphisms $\phi:C_1 \rightarrow C_2$ such that there are collared charts $H_i$ of $C_i$, $i=1,2$ respectively,
  such that $\left. H_2 \circ \phi \circ H_1^{-1} \right|_{S^1} \in \qso(S^1)$.
 \end{definition}
 \begin{remark}  The notation $\qso(S^1,C_1)$ will always be understood to refer to $S^1$ as the boundary of an annulus
  $\mathbb{A}(1,r)$ for $r>1$.  We will also write $\qso(S^1)=\qso(S^1,S^1)$.
 \end{remark}

 \begin{proposition} \label{pr:rqs_chart_independent}
  If $\phi \in \qso(C_1,C_2)$ then for any pair of collared charts $H_i$ of $C_i$, $i=1,2$ respectively,
  $\left. H_2 \circ \phi \circ H_1^{-1} \right|_{S^1} \in \qso(S^1)$.
 \end{proposition}
 \begin{proof}
  Assume that there are collared charts $H_i'$ of $C_i$ such that $H_2' \circ \phi \circ {H_1'}^{-1} \in \qso(S^1)$.  Let $H_i$
  be any other pair of collared charts.  The composition
  \[  H_2 \circ {H_2'}^{-1} \circ H_2' \circ \phi \circ {H_1'}^{-1} \circ
  H_1' \circ H_1^{-1} = H_2 \circ \phi \circ H_1^{-1}  \]
   is defined on some collared neighborhood of $C_1$.  Since $H_2 \circ {H_2'}^{-1}$ and $H_1' \circ H_1^{-1}$ have analytic extensions to $S^1$, the result follows from Proposition \ref{pr:analytic_in_qso} and Theorem \ref{th:QSo_group}.
 \end{proof}
 \begin{proposition} \label{pr:composition_preserves_qso_surfaces}
  Let $\riem^B_i$ be bordered Riemann surfaces and $C_i$ a boundary curve on each surface for $i=1,2,3$.  If $\phi \in \qso(C_1,C_2)$
  and $\psi \in \qso(C_2,C_3)$ then $\psi \circ \phi \in \qso(C_1,C_3)$.
 \end{proposition}
 \begin{proof}
  Let $H_i$ be collared charts of $C_i$ for $i=1,2,3$.  In that case
  \[  H_3 \circ \psi \circ \phi \circ H_1^{-1}= H_3 \circ \psi \circ H_2^{-1} \circ H_2 \circ \phi \circ H_1^{-1}  \]
  when restricted to $C_1$.  By Proposition \ref{pr:rqs_chart_independent} both $H_3 \circ \psi \circ H_2^{-1}$ and
  $H_2 \circ \phi \circ H_1^{-1}$ are in $\qso(S^1)$, so the composition is in $\qso(S^1)$ by Theorem \ref{th:QSo_group}.
  Thus $\psi \circ \phi \in \qso(C_1,C_3)$ by definition.
 \end{proof}
\end{subsection}
\begin{subsection}{A refined class of quasiconformal mappings between bordered surfaces}
\label{se:refined_qc_maps}
 We can now define a refined class of quasiconformal mappings.
 \begin{definition} \label{de:qco}
  Let $\riem^B_1$ and $\riem^B_2$ be bordered Riemann surfaces of type $(g,n)$, with boundary curves $C_1^i$ and $C_2^j$
  $i=1,\ldots,n$ and $j=1,\ldots,n$ respectively.  The class of maps $\qco(\riem^B_1,\riem^B_2)$ consists of
  those quasiconformal maps from $\riem^B_1$ onto $\riem^B_2$ such that the continuous extension to each boundary curve $C_1^i$, $i=1,\ldots,n$
  is in $\qso(C_1^i,C_2^j)$ for some $j \in \{1,\ldots,n\}$.
 \end{definition}
 Note that the continuous extension to a boundary curve $C_1^i$ {\it must}
 map onto a boundary curve $C_2^j$.

 The following two Propositions follow immediately from Definition \ref{de:qco} and Proposition \ref{pr:composition_preserves_qso_surfaces}.
 \begin{proposition} \label{pr:composition_preserves_qco_surfaces}
  Let $\riem^B_i$ $i=1,2,3$ be bordered Riemann surfaces of type $(g,n)$.  If $f \in \qco(\riem^B_1,\riem^B_2)$ and $g \in \qco(\riem^B_2,\riem^B_3)$
  then $g \circ f \in \qco(\riem^B_1,\riem^B_3)$.
 \end{proposition}
 \begin{proposition} \label{pr:mixed_composition_preserves}
  Let $\riem^B_1$ and $\riem^B_2$ be bordered Riemann surfaces.  Let $C_1$ be a boundary curve of $\riem^B_1$,
  $\phi \in \qso(S^1,C_1)$, $f \in \qco(\riem^B_1,\riem^B_2)$ and $C_2=f(C_1)$ be the boundary curve of $\riem^B_2$ onto
  which $f$ maps $C_1$.
  Then $f \circ \phi \in \qso(S^1,C_2)$.
 \end{proposition}
\end{subsection}
\end{section}
\begin{section}{Non-overlapping mappings} \label{se:non-overlapping}
In this section we show that the class of non-overlapping holomorphic maps into a Riemann
surface, with refined quasiconformal extensions, is a Hilbert manifold.  The
class of non-overlapping mappings is the infinite-dimensional part of both
the moduli space of Friedan and Shenker and the refined Teichm\"uller space.

Let $\riem$ be a punctured Riemann surface of type $(g,n)$.
In Section \ref{se:technical_lemmas}, we define the class of non-overlapping
mappings $\Oqco(\riem)$ and establish a technical theorem which is central to the proof
that it is a Hilbert manifold.  Section \ref{se:complex_structure_non-overlapping}
is devoted to defining a topology and atlas on $\Oqco(\riem)$, and the proof
that this topology is Hausdorff, second countable, and the overlap maps of
the atlas are biholomorphisms.
\begin{subsection}{Definitions and technical results}
\label{se:technical_lemmas}
 We define a class of non-overlapping mappings into a punctured Riemann surface.  Let $\mathbb{D}_0$
 denote the punctured disc $\mathbb{D} \backslash \{0\}$. Let $\riem$ be a compact Riemann surface with punctures
  $p_1,\ldots,p_n$.
 \begin{definition} The class of non-overlapping quasiconformally extendible maps $\Oqc(\riem)$ into
  $\riem$ is the set of $n$-tuples $(\phi_1,\ldots,\phi_n)$ where
  \begin{enumerate}
   \item For all $i \in \{1,\ldots,n\}$, $\phi_i:\mathbb{D}_0 \rightarrow \riem$ is holomorphic, and has a quasiconformal extension to a neighborhood of $\overline{\mathbb{D}}$.
   \item The continuous extension of $\phi_i$ takes $0$ to $p_i$
   \item For any $i \neq j$, $\overline{\phi_i(\mathbb{D})} \cap \overline{\phi_j(\mathbb{D})}$ is empty.
  \end{enumerate}
 \end{definition}
 It was shown in \cite{RSnonoverlapping} that $\Oqc(\riem)$ is a complex Banach manifold.

 As in the previous section, we need to refine the class of non-overlapping
 mappings.
 We first introduce some terminology.
 Denote the compactification of a punctured surface $\riem$ by $\overline{\riem}$.

 \begin{definition}
 \label{de:nchart}
 An $n$-chart on $\riem$ is a collection of open sets $E_1,\ldots,E_n$ contained in the compactification of $\riem$
 such that $E_i \cap E_j$ is
  empty whenever $i \neq j$, together with local parameters $\zeta_i:E_i
  \rightarrow \mathbb{C}$ such that $\zeta_i(p_i)=0$.
 \end{definition}
 In the following, we will refer to the charts $(\zeta_i,E_i)$ as being on
 $\riem$, with the understanding that they are in fact defined on the compactification. Similarly, non-overlapping maps $(f_1,\ldots,f_n)$ will be extended by the removable singularities theorem to the compactification, without further comment.

 \begin{definition}
  Let $\Oqco(\riem)$ be the set of $n$-tuples of maps
  $(f_1,\ldots,f_n) \in \Oqc(\riem)$ such that for any choice of
  $n$-chart $\zeta_i:E_i \rightarrow \mathbb{C}$, $i=1,\ldots, n$
  satisfying $\overline{f_i(\mathbb{D})} \subset E_i$ for all
  $i=1,\ldots,n$, it holds that $\zeta_i \circ f_i \in \Oqco$.
 \end{definition}

  The space $\Oqco(\riem)$ is well-defined.
 To see this let $(\zeta_i,E_i)$ and $(\eta_i,F_i)$,
 $i=1,\ldots,n$,
 be $n$-charts satisfying $\overline{f_i(\mathbb{D})} \subset E_i
 \cap F_i$ and assume that $\zeta_i \circ f_i \in \Oqco$.  Since
 $\eta_i \circ \zeta_i^{-1}$ is holomorphic on an open set
 containing $\overline{\zeta_i \circ f_i(\mathbb{D})}$, it follows
 from Lemma \ref{le:Oqco_composition_preserves} that $\eta_i \circ
 f_i = \eta_i \circ \zeta_i^{-1} \circ \zeta_i \circ f_i \in \Oqco$.

 In order to construct a Hilbert manifold structure on $\Oqco(\riem)$ we will need some technical
 theorems.
 \begin{theorem} \label{th:into_U_is_open} Let $E$ be an open neighborhood of $0$ in
 $\mathbb{C}$.  Then the set
 \[  \left\{ f \in \Oqc : \overline{f(\mathbb{D})} \subset E
 \right\}  \]
 is open in $\Oqc$ and the set
 \[  \left\{ f \in \Oqco : \overline{f(\mathbb{D})} \subset E
 \right\}  \]
 is open in $\Oqco$.
 \end{theorem}
 \begin{proof}
  Let $f_0 \in \Oqc$ satisfy $\overline{f_0(\mathbb{D})} \subset E$.
  By \cite[Corollary 3.5]{RSnonoverlapping}, there exists an
  open subset $W$ of $\Oqc$ such that $\overline{f(\mathbb{D})}
  \subset E$ for all $f \in W$.  Since $f_0$ was arbitrary, this proves the first claim.

  Now let $f_0 \in \Oqco$ satisfy $\overline{f_0(\mathbb{D})} \subset
  E$.  As above, there exists an
  open subset $W$ of $\Oqc$ such that $\overline{f(\mathbb{D})}
  \subset E$ for all $f \in W$.  But by Theorem \ref{th:Oqco_open_in_Oqc}
  $W \cap \Oqco=\iota^{-1}(W)$ is open in $\Oqco$.  Thus $\overline{f(\mathbb{D})} \subset E$ for all $f$ in the open set $W \cap \Oqco$ containing $f_0$.  This proves the second claim.
 \end{proof}

 Composition on the left by $h$ is holomorphic operation in both $\Oqc$ and $\Oqco$.  This was proven in \cite{RSnonoverlapping}
 in the case of $\Oqc$. The corresponding theorem in the refined case is
 considerably more delicate, and is one of the key theorems necessary
 to demonstrate the existence of a Hilbert
 manifold structure on $\Oqco(\riem^P)$.  Before we state and prove it we need to investigate some purely analytic issues in the underlying function theory, which  will be utilized later.

We start first with the following lemma.

\begin{lemma} \label{le:prewulfslemma}
  Let $f_t(z)$ be a holomorphic curve in $\Oqco$ for $t \in N$ where $N \subset \mathbb{C}$ is an open set containing $0$.
  Then there is a domain $N' \subseteq N$ containing $0$ and a $K$ which is independent of $t\in N'$ such that
  \begin{equation} \label{eq:ftprime_estimate}
   \iint_{\mathbb{D}} |f_t'(z)|^p (1-|z|^2)^{\alpha} dA \leq K,
  \end{equation}
for all $p>0$ and $\alpha>-1.$ The constant $K$ will depend on $p$ and $\alpha$.
 \end{lemma}

\begin{proof}
 To establish the estimate \eqref{eq:ftprime_estimate} we observe that since $\mathcal{A}(f_t) \in A_1^2(\mathbb{D})$, $\log{f_t'}$ is in the little Bloch space; that is
   \[  \lim_{|z| \rightarrow 1^-} (1-|z|^2) |g_t'(z)|=0,  \]
  see \cite[Corollary 1.4, Chapter 2]{Takhtajan_Teo_Memoirs}.  By \cite[Theorem 1 (1)]{GGPPR}, the integral in (\ref{eq:ftprime_estimate}) is finite for each $t$.  However, we need a uniform estimate in $t$.  Although this does not follow from the theorem as stated in \cite[Theorem 1 (1)]{GGPPR}, the proof of that theorem can be modified to get the uniform estimate. We proceed by providing the details of this argument.
 The claim of \cite[Theorem 1 (1)]{GGPPR} is that
  \begin{equation} \label{eq:GGPPR_claim}
   g=\log{f'} \in \mathcal{B}_0 \Longrightarrow \iint_{\mathbb{D}} |f'|^p(1-|z|^2)^\alpha \,dA <\infty
  \end{equation}
 for all $p>0$ and $\alpha >-1$ where $\mathcal{B}_0$ is the little Bloch space.

 Let $h_s(z)=g(sz)$.  This function is continuous on $\overline{\mathbb{D}}$ for $0<s<1$.
 Hence  for each fixed $s$
 the integral in question converges by an elementary estimate.  Therefore (\ref{eq:GGPPR_claim})
  will follow
 if we can show that the integral is uniformly bounded for $s$ in some interval $[s_0,1)$.

 We have that $h_s \in \mathcal{B}_0$, that is,
 \[  \lim_{|z| \rightarrow 1-} (1-|z|^2)|h_s'(z)|=0  \]
 for all $0 < s \leq 1$.
  Since $h_1(z)=g(z)$ is in the little Bloch space, and $S^1$ is compact, given any $\epsilon >0$ there is an $R>0$ such that $(1-|z|^2)|h_1'(z)| <\epsilon$ for all $|z|>R$.
 Fix any $0<s_0<1$ and let $r=R/s_0$.  Therefore, if $|z|>r$ and $s_0<s \leq 1$ then
 $|sz| > s_0 r = R$ and so for all $|z|>r$ and $s_0 < s \leq 1$ we have $(1-|z|^2)|h_s'(z)|=(1-|z|^2)s |h_1'(sz)| <s \epsilon \leq \epsilon.$

 Thus for any $\epsilon>0$ there are fixed $0<r<1$
 and $0<s_0 <1$
 such that
\begin{equation}\label{estimate for h_s}
(1-|z|^2)|h_s'(z)| < \epsilon
\end{equation}
 for all $(s,z) \in [s_0,1] \times \overline{\mathbb{D}} \backslash
 D_r$ where $D_r =\{z \,:\,|z|<r \}$.
 Now set
\begin{align*}
 I & =  \iint_{\mathbb{D}} |e^{h_s (z)}|^p (1-|z|^2)^\alpha \,dA,  \\
 I_1 & = \iint_{ D_r} |e^{h_s (z)}|^p (1-|z|^2)^\alpha \,dA,\\
 I_2 & = \iint_{\mathbb{D} \setminus D_r} |e^{h_s (z)}|^p (1-|z|^2)^\alpha \,dA.
 \end{align*}

Our goal is to show that there is a constant $C$ which is independent of $s\in [s_0 ,1)$ such that $I$ is bounded by $C.$ It is obvious that this will follow by establishing the aforementioned type of bounds for $I_1$ and $I_2$.
The estimate for $I_1$ follows  from
\begin{align}  \iint_{ D_r} |e^{h_s (z)}|^p (1-|z|^2)^\alpha \,dA & \leq \frac{(1-r^2)^{\min(\alpha,0)}}{s^2} \iint_{ D_{rs}} |e^{h_1 (z)}|^p \,dA \\ & \leq  \frac{(1-r^2)^{\min(\alpha,0)}}{s_{0}^2} \iint_{ D_r} |e^{h_1 (z)}|^p \,dA   \nonumber \\
& \leq C. \nonumber
\end{align}

Now we turn to the estimate for $I_2$.  It follows from a theorem of Hardy and Littlewood (see for example \cite[Theorem 6]{Flett} for a proof in the most general case) that there is a $C$ depending only on $p$ and $\alpha$, such that

  \begin{equation}\label{equivalent norms}
    \iint_{\mathbb{D}} |F(z)|^p(1-|z|^2)^\alpha dA \leq C \left( \iint_{\mathbb{D}} |F'(z)|^p(1-|z|^2)^{p+\alpha}dA + |F(0)|^{p}\right)
  \end{equation}
  for $p>0$ and $\alpha > -1$, whenever at least one of the integrals converges
  (in fact the two norms represented by each side are equivalent). Now for $s \in [s_0,1)$ we may apply
 \eqref{equivalent norms} and \eqref {estimate for h_s} to $e^{h_s (z)}$ which yield
 \begin{align*}
   I_2  & \leq  \iint_{\mathbb{D}} |e^{h_s (z)}|^p (1-|z|^2)^\alpha \,dA \\
    & \leq  C \left( \iint_{\mathbb{D}} |e^{h_s(z)}|^p|h_s'(z)|^p(1-|z|^2)^{p+\alpha}dA
      + |e^{h_s(0)}|^{p}\right) \\
    & \leq C \iint_{\mathbb{D}\setminus D_r} |e^{h_s(z)}|^p|h_s'(z)|^p(1-|z|^2)^{p+\alpha}\,dA
      + C  \iint_{D_r} |e^{h_s(z)}|^p|h_s'(z)|^p(1-|z|^2)^{p+\alpha}\,dA  \\
      & \qquad + C|e^{h_s(0)}|^{p} \\
    &\leq C  \epsilon I_2
      + C  \iint_{D_r} |e^{h_s(z)}|^p|h_s'(z)|^p(1-|z|^2)^{p+\alpha}\,dA +
   C|e^{h_s(0)}|^{p} \\
   & \leq \frac{1 }{2} I_2 +C  \iint_{D_r} |e^{h_s(z)}|^p|h_s'(z)|^p(1-|z|^2)^{p+\alpha}\,dA+  C |e^{h_s(0)}|^{p} ,
 \end{align*}
by choosing $\epsilon\leq \frac{1}{2C}$.
 Summarizing, we have
 \begin{equation} \label{eq:I1estimate}
  I_2 \leq 2C \left( \iint_{D_r} |e^{h_s(z)}|^p|h_s'(z)|^p(1-|z|^2)^{p+\alpha}dA
      + |e^{h_s(0)}|^{p}\right),
 \end{equation}
 where $r$ and $C$ are independent of $s$.  Since $h_s$  and $h_s'$ are continuous on $\overline{D_r}$ for $s \in [s_0,1 ) $ the integral on the right hand side is bounded by a constant which is independent of $s \in [s_0,1).$ Therefore the estimates for $I_1$ and $I_2$ yield the desired uniform estimate for $I$. Since the estimate on $I$ is uniform it extends to $s=1$.

Setting $g_t=\log f_t'$, an argument identical to the above (substituting $h_s$ with $g_t$) gives the desired uniform bound (\ref{eq:ftprime_estimate}) in $t$, provided
  that the function $(1-|z|^2) |g_t'(z)|$ is jointly continuous in $(t,z)$.
  Thus it remains to demonstrate the joint continuity. To this end fix $z_0 \in \overline{\mathbb{D}}$, $t_0 \in N$ and $\epsilon>0$.
 There is a $\delta$ such that for
 any $z \in B(z_0,\delta) \cap \overline{\mathbb{D}}$ where $B(z_0,r)$ is the ball
 of radius $\delta$ centered on $z_0$,
 \[  \|(1-|z|^2)g_{t_0}'(z) - (1-|z|^2)g_{t_0}'(z_0) \|_\infty <\frac{\epsilon}{2}.  \]

 Since $f_t$ is a holomorphic curve, there is an interval $(t_0-\delta_1,t_0+\delta_1)$ such that
 \[  \|\mathcal{A}(f_t)-\mathcal{A}(f_{t_0})\| < \epsilon/2.  \]
 By \cite[Lemma 1.3, Chapter II]{Takhtajan_Teo_Memoirs} for $g=\log{f'}$
 \[  \| (1-|z|^2)g'(z) \|_\infty \leq \frac{1}{\sqrt{\pi}} \|\mathcal{A}(f)\|  \]
 (note that in their notation the left hand side is $\|g'(z)\|_\infty$).
 So for all $z \in \mathbb{D}$ and $t \in (t_0 - \delta_1,t_0+\delta_1)$,
 \begin{equation} \label{eq:littleBloch_temp}
  \|(1-|z|^2)g_{t}'(z) - (1-|z|^2)g_{t_0}'(z) \|_\infty < \frac{\epsilon}{2}.
 \end{equation}
 Combining this with the fact that $(1-|z|^2)g_t'(z) \rightarrow 0$ as $|z| \rightarrow 1$ shows that equation (\ref{eq:littleBloch_temp}) holds on $\overline{\mathbb{D}}$.
 Thus, by the triangle inequality
 \[ \|(1-|z|^2)g_{t}'(z) - (1-|z|^2)g_{t_0}'(z_0) \|_\infty  <\epsilon \]
 on $(t_0-\delta_1,t_0+\delta_1) \times (D(z_0,r) \cap \overline{\mathbb{D}})$. This proves joint continuity and thus
 completes the proof.
\end{proof}

Before we state our next lemma we would needs some tools from the theory of Besov spaces which we recall bellow.

\begin{definition}
 For $p\in (1,\infty),$ one defines the {\it{Besov space}} $B^{p}$ as the space of holomorphic functions $f$ on $\mathbb{D}$ for which
\begin{equation*}
    \Vert f\Vert_{B^{p}} = |f(0)| + \left\{\iint_{\mathbb{D}} |f'(z)|^{p}\,(1-|z|^2)^{p-2}\, dA\right\}^{\frac{1}{p}}<\infty.
  \end{equation*}
\end{definition}
From this definition it follows at once that $B^{2}$ is the usual Dirichlet space. \noindent One also defines for $z\in \mathbb{D},$ the set $S(z)$ by
\begin{equation}\label{defn of S}
S(z) = \left\{ \zeta\in\mathbb{D} \,:\, 1-|\zeta|\leq 1-|z|, \, \left|\frac{\arg(z\,\overline{\zeta})}{2\pi}\right|\leq \frac{1-|z|}{2}\right\},
\end{equation}
which is obviously a subset of the annulus $|z|\leq |\zeta|<1.$\\
\noindent In our study we shall use the following result, concerning Carleson measures for Besov spaces, due to N. Arcozzi, R. Rochberg and E. Saywer \cite{ARS}.

\begin{theorem}\label{ARS theorem}
Given real numbers $p$ and $q$ with $1<p< q<\infty$ and a positive Borel measure $\mu$ on $\mathbb{D},$ the following two statements are equivalent:

\begin{enumerate}

\item There is a constant $C(\mu)>0$ such that
\begin{equation*}
\Vert f\Vert_{L^{q}(\mu)} \leq C(\mu) \Vert f\Vert_{B^{p}}.
\end{equation*}
\item For $S(z)$ defined above, one has
\begin{equation*}
 \mu(S(z)) ^{\frac{1}{q}}\leq C \left\{\log \frac{1+|z|}{1-|z|}\right\}^{-\frac{1}{p'}},
\end{equation*}
where $p'$ is the H\"older dual of $p.$
\end{enumerate}
\end{theorem}
Using Lemma \ref{le:prewulfslemma} and Theorem \ref{ARS theorem} we can prove the following result:
 \begin{lemma} \label{le:wulfslemma}
  Let $f_t(z)$ be a holomorphic curve in $\Oqco$ for $t \in N$ where $N \subset \mathbb{C}$ is an open set containing $0$.
  For any holomorphic function $\psi:\mathbb{D} \rightarrow \mathbb{C}$ such that $\iint_{\mathbb{D}} |\psi'|^2 <\infty$ and $\psi(0)=0$, and any $\beta>1$, there is a constant $C$
  and an open set $N' \subseteq N$ containing $0$ such that for all $t \in N'$
  \[  \iint_{\mathbb{D}} |f_t'|^2 |\psi|^\beta dA \leq C. \]
 \end{lemma}

\begin{proof}

The Cauchy-Schwarz inequality and Lemma \ref{le:prewulfslemma} with $p=4$ and $\alpha=-\frac{1}{2}$ yield
\begin{align*}\label{rune estim 1}
 \iint_{\mathbb{D}} |f_t' (z)|^2 |\psi(z)|^\beta dA & \leq
 \left\{\iint_{\mathbb{D}} |f_t'(z)|^{4}\,(1-|z|^2)^{\frac{-1}{2}}\, dA \right\}^{\frac{1}{2}} \times \left\{\iint_{\mathbb{D}} |\psi(z)|^{2\beta}\,(1-|z|^2)^{\frac{1}{2}}\,  dA\right\}^{\frac{1}{2}} \\
 & \leq  \sqrt{K}\left\{\iint_{\mathbb{D}} |\psi(z)|^{2\beta}\,(1-|z|^2)^{\frac{1}{2}}\, dA\right\}^{\frac{1}{2}}.
\end{align*}

Therefore, since $\psi$ is in the Dirichlet space, to prove that $\iint_{\mathbb{D}} |f_t' (z)|^2 |\psi(z)|^\beta dA\leq C$, it would be enough to show that
\begin{equation}\label{sufficient condition}
\left\{\iint_{\mathbb{D}} |\psi(z)|^{2 \beta}\,(1-|z|^2)^{\frac{1}{2}}\, dA\right\}^{\frac{1}{2\beta}} \leq C' \left\{\iint_{\mathbb{D}} |\psi'(z)|^{2}\, dA\right\}^{\frac{1}{2}}.
\end{equation}

Now, since $\psi(0)=0,$ Theorem \ref{ARS theorem} with $q=2 \beta,$ $p=2$ and $d\mu=(1-|\zeta|^2)^{\frac{1}{2}}\, dA,$ yields that  \eqref{sufficient condition} holds if and only if for all $z\in \mathbb{D}$

\begin{equation}\label{perhaps the main estimate}
\left\{\iint_{S(z)} (1-|\zeta|^2)^{\frac{1}{2}}\, dA\right\}^{\frac{1}{2\beta}} \leq C'  \left\{\log \frac{1+|z|}{1-|z|}\right\}^{-\frac{1}{2}}.
\end{equation}

Moreover

 \begin{equation*}
 \iint_{S(z)} (1-|\zeta|^2)^{\frac{1}{2}}\, dA \leq \iint_{|z|\leq |\zeta|<1} (1-|\zeta|^2)^{\frac{1}{2}}\, dA =4\pi \frac{(1-|z|^2)^{\frac{3}{2}}}{3}.
 \end{equation*}
Therefore an elementary calculation yields that \eqref{perhaps the main estimate} follows from an estimate of the form

 \begin{equation}\label{final estim for L2 boundedness}
   (1-|z|^2)^{\frac{3}{2 \beta}} \log \frac{1+|z|}{1-|z|} \leq C,
 \end{equation}
 for all $|z|<1.$ Now if we set $f(r) = (1-r^2)^{{\frac{3}{2\beta}}}\log \frac{1+r}{1-r}$ then for all $\varepsilon>0,$ $f(r)$ is continuous on the compact interval $[0,1-\varepsilon].$ Indeed the continuity of $f(r)$ is obvious on $[0,1-\varepsilon)$ and moreover

\begin{equation*}
  \lim_{r\to 1^{-}} (1-r^2)^{{\frac{3}{2\beta}}}\log \frac{1+r}{1-r} =0.
\end{equation*}
From this, \eqref{final estim for L2 boundedness} follows and the proof of the lemma is now complete.
\end{proof}

 Now we will state and prove the holomorphicity of the operation of left composition in $\Oqc_0$ which will play a crucial role in the establishment of the existence of the Hilbert manifold structure on $\Oqc_0 (\Sigma^P)$.

\begin{theorem} \label{th:left_comp_holo}
  Let $K \subset \mathbb{C}$ be a compact set which is the closure of an open
  neighborhood $K_{int}$ of $0$ and let $A$ be an open set in $\mathbb{C}$
  containing $K$. If $U$ is the open set
  \[  U = \{ g \in \Oqco \,:\, \overline{g(\mathbb{D})} \subset
  K_{int} \}, \]
  and $h:A \rightarrow \mathbb{C}$ is a one-to-one holomorphic map such that
  $h(0)=0,$ then the map $f \mapsto h \circ f$ from $U$ to $\Oqco$ is holomorphic.
 \end{theorem}
 \begin{remark}  The fact that $U$ is open follows from Theorem
 \ref{th:into_U_is_open}.
 \end{remark}

 \begin{proof}
 It was shown in \cite[Lemma 3.10]{RSnonoverlapping} that composition on the left is holomorphic in the above sense on $\Oqc$.  However, this does not immediately lead to the desired result, since the norm has changed.  Nevertheless some of the computations in \cite[Lemma 3.10]{RSnonoverlapping} can be used here.

 As in \cite[Lemma 3.10]{RSnonoverlapping}, by Hartogs' theorem \cite{Mujica} it suffices to show that the maps
  $(\mathcal{A}(f),f'(0)) \mapsto \mathcal{A}(h \circ f)$ and $f'(0) \mapsto h'(0)f'(0)$ are separately holomorphic.  The second map is clearly holomorphic.  By a theorem in \cite[p 198]{Chae}, it suffices to show that $(\mathcal{A}(f),f'(0)) \mapsto \mathcal{A}(h \circ f)$ is G\^ateaux holomorphic and locally bounded.
  It is locally bounded by Lemma \ref{le:Oqco_composition_preserves}.

  To show that this map is G\^ateaux holomorphic, consider the curve $(\mathcal{A}(f_0) + t \phi,q(t))$ where $\phi
  \in A^2_1(\mathbb{D})$ and $q$ is holomorphic in $t$ with $q(0)=f_0'(0)$.  It can be easily computed that
  $(\mathcal{A}(f_t),f_t'(0))=(\mathcal{A}(f_0) + t \phi,q(t))$ if and only if
  $f_t$ is the curve
  \[  f_t(z)=\frac{q(t)}{f_0'(0)} \int_0^z f_0'(u) \exp{\left( t\int
       _0^u \phi(w)dw \right)} du.  \]
  Note that $f_t(z)$ is holomorphic in $t$ for fixed $z$.
  Since $\chi(\Oqc)$ is open and $\iota:\Oqc_0 \rightarrow \Oqc$ is continuous, there is an open neighborhood $N$ of $0$ in $\mathbb{C}$ such that $f_t \in \Oqc_0$ for all $t \in N$.  The neighborhood $N$ can also be chosen small enough that $\overline{f_t(\mathbb{D})} \subset K_{int}$ for all $t \in N$,
  since we assumed that $t \mapsto f_t$ is a holomorphic curve and the set of $f \in \Oqco$ mapping into $K_{int}$ is open by Theorem \ref{th:into_U_is_open}.

  Defining $\alpha(t)=\mathcal{A}(h) \circ f_t \cdot f_t'$
  and denoting $t$-differentiation with a dot we then have that
  \[  \lim_{t \rightarrow 0} \frac{1}{t} \left( \mathcal{A}(h \circ
     f_t) - \mathcal{A}(h \circ f_0) \right) = \dot{\alpha}(t) + \phi.
  \]
  So it is enough to show that
  \begin{equation} \label{eq:necessary_claim}
    \left\|\frac{1}{t} \left( \mathcal{A}(h \circ f_t) - \mathcal{A}(h
    \circ f_0)
   \right) -  (\dot{\alpha}(t) + \phi)  \right\| = \left\|
   \frac{1}{t}\left( \alpha(t) - \alpha(0) - t \dot{\alpha}(0) \right)
   \right\| \rightarrow 0
  \end{equation}
  as $t \rightarrow 0$. For any fixed $z$ (recall that $\alpha(t)$ is also a function of $z$) we have
  \[  \alpha(t)-\alpha(0) - t \dot{\alpha}(0) = \int_0^t \ddot{\alpha}(s) (t-s) ds.  \]

  We claim that there is a constant $C_0$
  such that $\| \ddot{\alpha}\|<C_0$ for all $t$ in some neighborhood of $0$.
  Assuming for the moment that this is true, for $|s|<|t|<C$ we set $t=e^{i\theta}u$
  and $s=e^{i\theta}v$, and integrating along a ray, we have
  \begin{align*}
   \| \alpha(t) - \alpha(0) - t \dot{\alpha}(0) \|^2 & =  \left\| \int_0^t \ddot{\alpha}(s) (t-s) ds \right\|^2 \\
   & = \iint_{\mathbb{D}} \left| \int_0^t \ddot{\alpha}(s) (t-s) \, ds\right|^2 dA  \\
   & \leq
   \iint_{\mathbb{D}} \left(  \int_0^{u} | \ddot{\alpha}(e^{i\theta}v)| (u-v) dv \right)^2 dA \\
   & \leq  \iint_{\mathbb{D}} \int_0^u u | \ddot{\alpha}(e^{i\theta}v)|^2 (u-v)^2 dv\, dA \\
 & \leq  C \iint_{\mathbb{D}} \int_0^u | \ddot{\alpha}(e^{i\theta}v)|^2 (u-v)^2 dv\, dA
  \end{align*}
  where we have used Jensen's inequality and the assumption that $u<C$. Therefore Fubini's theorem and the assumption that $v<u<|t|$ yield
  \begin{align*}
   \| \alpha(t) - \alpha(0) - t \dot{\alpha}(0) \|^2 &\leq
     4C  |t|^2\int_0^{|t|} \left(\iint_{\mathbb{D}} | \ddot{\alpha}(s)|^2 dA\right)   d|s| \\
   & \leq  C_1 |t|^3.
  \end{align*}

  Fubini's theorem can be applied since the second to last
  integral converges by the final inequality.  This would prove (\ref{eq:necessary_claim}).
  Thus the proof reduces to establishing a bound on $\| \ddot{\alpha} \|$ which is uniform in $t$ in
  some neighborhood of $0$.

  By \cite[equation 3.2]{RSnonoverlapping},
  \begin{align} \label{eq:alphadoubledot}
   \ddot{\alpha}(t) & =   \mathcal{A}(h)'' \circ f_t \cdot f_t'  \cdot
   {\dot{f}_t}^2
   + \mathcal{A}(h)'
   \circ f_t \cdot f_t'  \cdot   \ddot{f}_t  \\ \nonumber 
   &  \quad + 2 \mathcal{A}(h)'
   \circ f_t \cdot \dot{f}_t  \cdot
    {\dot{f}_t}'  +  \mathcal{A}(h) \circ f_t \cdot {\ddot{f}_t'} \\ \nonumber
    &= I + II + III + IV
  \end{align}
  where
  \[  \mathcal{A}(h)' = \frac{h'''}{h'} - \frac{{h''}^2}{{h'}^2}  \]
  and
  \[  \mathcal{A}(h)'' = \frac{h''''}{h'} - 3 \frac{h'''h''}{{h'}^2} - \frac{{h''}^3}{{h'}^3}. \]

  We will uniformly bound all the terms on the right side of (\ref{eq:alphadoubledot}) in the $A_1^2(\mathbb{D})$ norm.
  For all $t \in N$ we have $\overline{f_t(\mathbb{D})} \subset K$ and $h$ is holomorphic on an open set containing the compact set $K$, and $h' \neq 0$ since $h$ is one-to-one on $A$.  Thus there is a uniform bound for $\mathcal{A}(h)$, $\mathcal{A}(h)'$ and $\mathcal{A}(h)''$ on $f_t(\mathbb{D})$.  So by a change of variables, there is an $M$ such that
  \begin{equation} \label{eq:Abound}
   \| \mathcal{A}(h) \circ f_t \cdot f_t' \| = \left( \iint_{f_t(\mathbb{D})} |\mathcal{A}(h)|^2 dA\right)^{1/2} \leq M.
  \end{equation}
  Similarly there are $M'$ and $M''$ such that
  \begin{equation} \label{eq:Aprimebound}
    \| \mathcal{A}(h)' \circ f_t \cdot f_t' \| \leq M' \quad \text{and} \quad 
     \| \mathcal{A}(h)'' \circ f_t \cdot f_t' \| \leq M''.
  \end{equation}

  Since $\overline{f_t(\mathbb{D})}$ is contained in the compact set $K$, $|f_t(z)|$  is bounded
  by a constant $C$ which is independent of $t$.  By applying Cauchy estimates in the variable $t$ on a curve $|t|=r_2$, we see that for $0<r_1<r_2$ and $|t| \leq r_1$,
  \[  |\dot{f}_t(z)| \leq \frac{r_2}{(r_1-r_2)^2} \sup_{|s|=r_2}|f_s(z)| \]
  and thus we can
  find a constant $C'$ such that $|\dot{f}_t(z)| \leq C'$ for $|t| \leq r_1$.  Similarly, there is a $C''$ such that
  $|\ddot{f}_t(z)| \leq C''$ for all $z \in \mathbb{D}$ and $|t| \leq r_1$. Combining with (\ref{eq:Aprimebound}), we have that $\|I\|$ and $\|II\|$ are uniformly bounded on $|t|\leq r_1$.

  Next, observe that $\| \mathcal{A}(h)' \circ f_t\|_\infty \leq D$ and $\|\mathcal{A}(h) \circ f_t\|_\infty \leq D'$ for some
  constants $D$ and $D'$ which are independent of $t$, since $f_t(\mathbb{D})$ is contained inside a compact set in the
  interior of the domain of $h$, and $h$ is holomorphic and one-to-one.  Therefore, to get a uniform bound on $\|\ddot{\alpha}\|$ we only need to show that
  $\|\dot{f}_t'\|$ and $\|\ddot{f}_t'\|$ are bounded by some constant which is independent of $t$ on a neighborhood of $0$.

  A simple computation yields
  \[  \dot{f}_t'(z) = \frac{\dot{q}(t)}{q(t)}f_t'(z) + \left( \int_0^z \phi(w) dw \right) f_t'(z).  \]
  Since $q(t)$ is holomorphic and non-zero, $\dot{q}/q$ is uniformly bounded on a neighborhood of $0$.  Furthermore,
  \[  \iint_{\mathbb{D}} |f_t'|^2 dA = \operatorname{Area}(f_t(\mathbb{D}))  \]
  which is uniformly bounded since $f_t(\mathbb{D})$ is contained in a fixed compact set.
  Since $\psi(z)=\int_0^z \phi(w)dw$ is in the Dirichlet space, we can apply Lemma \ref{le:wulfslemma} with $\beta=2$,
  which proves that $\|\dot{f}_t'\|$ is uniformly bounded for $t$ in some neighborhood of $0$.  We further
  compute that
  \[  \ddot{f}_t'(z)= \frac{\ddot{q}(t)}{q(t)} f_t'(z) + 2 \frac{\dot{q}(t)}{q(t)} \left(\int_0^z \phi(w) dw \right) f_t'(z)+
    \left( \int_0^z \phi(w) dw \right)^2 f_t'(z),  \]
  so the same reasoning (this time using Lemma \ref{le:wulfslemma} with $\beta=2$ and $\beta=4$) yields a uniform bound for $\|\ddot{f}_t'\|$. This completes the proof.
 \end{proof}

\end{subsection}
\begin{subsection}{Complex Hilbert manifold structure on $\Oqco(\riem)$}
\label{se:complex_structure_non-overlapping}
The idea behind the complex Hilbert space structure is as follows. Any element $(f_1,\ldots,f_n)$ of $\Oqco(\riem)$ maps $n$ closed discs onto closed sets
 containing the punctures.  We choose charts $\zeta_i$, $i=1,\ldots,n$,
 which map non-overlapping open neighborhoods of
 the closed discs into $\mathbb{C}$.  The maps $\zeta_i \circ f_i$ are in $\Oqco$, which is an open subset of a Hilbert space.   By Theorem \ref{th:into_U_is_open} the components $g_i$ of
 an element $g$ nearby to $f$
 will also have images in the domains of the charts $\zeta_i$.  Thus we can model $\Oqco(\riem)$
 locally by $\Oqco \times \cdots \times \Oqco$.  Theorem \ref{th:left_comp_holo} will ensure that the transition functions of the charts are biholomorphisms.

 We now turn to the proofs, beginning with the topology on $\Oqco(\riem)$.
 Before defining a topological basis we need some notation.
 \begin{definition} \label{de:compatible_open_set}
 For any $n$-chart $(\zeta,E)=(\zeta_1,E_1,\ldots,\zeta_n,E_n)$ (see Definition \ref{de:nchart}), we say that an $n$-tuple
 $U=(U_1,\ldots,U_n) \subset \Oqco \times \cdots \times \Oqco$, with $U_i$ open in $\Oqco$, is compatible with $(\zeta,E)$ if
 $\overline{f(\mathbb{D})} \subset \zeta_i(E_i)$ for all $f \in U_i$.
 \end{definition}
 For any $n$-chart $(\zeta,E)$ and compatible open subset $U$ of ${\Oqco} \times \cdots \times \Oqco$ let
 \begin{align} \label{eq:Oqcoriem_base_definition}
   V_{\zeta,E,U} & =  \{ g \in \Oqco(\riem) \,:\, \zeta_i \circ g_i \in U_i, \ \ i=1,\ldots,n \} \\
    & =  \{ (\zeta^{-1}_1 \circ h_1,\ldots,\zeta^{-1}_n \circ h_n) \,:\, h_i \in U_i, \ \
    i=1,\ldots,n \}.   \nonumber
 \end{align}
 \begin{definition}[base a for topology on $\Oqco(\riem)$]   \label{de:Oqco_base} Let
 \[ \mathcal{V}=\{ V_{\zeta,E,U} \,:\, (\zeta,E)   \text{ an $n$-chart, }
    U  \text{ compatible with }   (\zeta,E) \}. \]
 \end{definition}
 \begin{theorem} \label{th:Oqco_base}
 The set $\mathcal{V}$ is the base for a topology on $\Oqco(\riem)$.  This topology is
  Hausdorff and second countable.
 \end{theorem}
 \begin{proof}
  We first establish that $\mathcal{V}$ is a base.  For any element $f$ of $\Oqco(\riem^P)$, since
  $\overline{f_i(\mathbb{D})}$ is compact for all $i$, there is an $n$-chart $(\zeta,E)$ such
  that $\overline{f_i(\mathbb{D})} \subset E_i$ for each $i$.  By Theorem \ref{th:into_U_is_open}
  there is a $U=(U_1,\ldots,U_n)$ compatible with $(\zeta,E)$.  Thus $\mathcal{V}$ covers
  $\Oqco(\riem^P)$.

  Now let $V_{\zeta,E,U}$ and $V_{\zeta',E',U'}$ be two elements of $\mathcal{V}$ containing
  a point $f \in \Oqco(\riem^P)$.   Define $E''$ by $E_i''=E_i \cap E_i'$.
  For each $i$ choose a compact set $\kappa_i$ such that $\overline{f_i(\mathbb{D})}
  \subseteq \kappa_i \subseteq E''_i$.  Let $K_i=\zeta_i(\kappa_i)$, $K_i'=\zeta'_i(\kappa_i)$,
  \[  W_i = \{ \phi \in \Oqco : \overline{\phi(\mathbb{D})} \subseteq K_i^{int} \} \]
  and
  \[  W'_i = \{ \phi \in \Oqco : \overline{\phi(\mathbb{D})} \subseteq {K'_i}^{int} \} \]
  where $K_i^{int}$ and ${K'_i}^{int}$ are the interiors of $K_i$ and $K_i'$ respectively.
  By Theorem \ref{th:into_U_is_open} $W_i$ and $W_i'$ are open, and by Theorem \ref{th:left_comp_holo} the map $\phi \mapsto \zeta'_i \circ \zeta_i^{-1} \circ \phi$
  is a biholomorphism from $W_i$ onto $W'_i$.
  So the set
  \[  U_i'' = U_i \cap \left( \zeta_i \circ {\zeta_i'}^{-1} \left( W_i' \cap U_i' \right) \right) \subseteq U_i \cap W_i  \]
  is an open subset of $\Oqco$ (by ${\zeta_i'}^{-1}( W_i' \cap U_i' )$ we mean the set of ${\zeta_i'}^{-1} \circ \phi$ for $\phi \in W_i' \cap U_i'$).
  Setting $\zeta_i''=\left. \zeta \right|_{E_i''}$
  we have that $f \in V_{\zeta'',E'',U''} \subseteq V_{\zeta,E,U} \cap V_{\zeta',E',U'}$
  by construction.
  Thus $\mathcal{V}$ is a base.

  To show that the topology generated by $\mathcal{V}$ is Hausdorff, let $f,g \in \Oqco(\riem^P)$.
  Choose open, simply connected sets $E_i$ and $F_i$, $i=1,\ldots,n$ such that
  $\overline{f_i(\mathbb{D})} \subset E_i$ and $\overline{g_i(\mathbb{D})} \subset F_i$ and
  $E_i \cap E_j = F_i \cap F_j = \emptyset$ whenever $i \neq j$.  For each $i$ let $\zeta_i:E_i \cup F_i \rightarrow  \mathbb{C}$ be a biholomorphism taking $p_i$ to $0$.  Thus $\left. \zeta_i \right|_{E_i}$
  defines an $n$-chart $(\zeta,E)$, and similarly for $\left. \zeta_i \right|_{F_i}$.
  (The collection $\left. \zeta_i \right|_{E_i \cup F_i}$ does not necessarily form an $n$-chart, but
  this is inconsequential).

  Since $\Oqco$ is a Hilbert space, it is Hausdorff, so for all $i$ there are open sets $U_i$ and $W_i$ such
  that $\zeta_i \circ f_i \in U_i$, $\zeta_i \circ g_i \in W_i$, and $U_i \cap W_i = \emptyset$.
  By Theorem \ref{th:into_U_is_open}, by shrinking $U_i$ and $W_i$ if necessary, we can assume that
  $\overline{h_i(\mathbb{D})} \subset \zeta_i(E_i)$ for all $h_i \in U_i$ and
  $\overline{h_i(\mathbb{D})} \subset \zeta_i(F_i)$ for all $h_i \in W_i$.  That is, $U$ is compatible
  with $(\zeta,E)$ and $W$ is compatible with $(\zeta,F)$.  Furthermore $f \in V_{\zeta,E,U}$,
  $g \in V_{\zeta,F,W}$ and $V_{\zeta,E,U} \cap V_{\zeta,F,W} = \emptyset$ by construction.
  Thus $\Oqco(\riem)$ is Hausdorff with the topology defined by $\mathcal{V}$.

  To see that $\Oqco(\riem)$ is second countable, we proceed as follows.
  First observe that $\riem$ is second countable by Rado's Theorem (see for example \cite{Hubbard}).
  Thus it has a countable basis $\mathfrak{B}$
  of open sets.  Let $\mathfrak{B}^n = \{ (B_1,\ldots,B_n) \}$ where each $B_i$ (1) is
  a finite union of elements of $\mathfrak{B}$ and (2) contains $p_i$.  Clearly $\mathfrak{B}^n$
  is countable.  Consider the set of $n$-tuples $C=(C_1,\ldots,C_n)$ such that (1) $(C_1,\ldots,C_n) \in
  \mathfrak{B}^n$ and (2) $C_i \cap C_j$ is
  empty whenever $i \neq j$.  Since this is a subset of $\mathfrak{B}^n$, it is countable.
  Furthermore, for each $(C_1,\ldots,C_n)$, we can
  fix a chart $\zeta_i:C_i \rightarrow \mathbb{C}$.  Let
  $\mathfrak{C}$ be the collection of $n$-charts $\{ (\zeta_1,C_1,\ldots,\zeta_n,C_n)\}$
  where $\zeta_i$ and $C_i$ are as above.

  Next, since $\Oqco$ is a Hilbert space (and hence a separable metric space), it has a countable basis of open sets $\mathfrak{O}$.  We define a countable basis for
  the topology of $\Oqco(\riem)$ as follows:
  \[  \mathcal{V}' =\{ V_{(\zeta,C,W)} \,:\, (\zeta,C) \in \mathfrak{C}, \ W \
     \mathrm{compatible} \ \mathrm{with} \ (\zeta,C), \ W_i \in \mathfrak{O}, \ i=1,\ldots,n \}.  \]
  Each $V' \in \mathcal{V}'$ is open by Theorem \ref{th:into_U_is_open}.  Furthermore $\mathcal{V}'$ is
  countable since $\mathfrak{C}$ and $\mathfrak{O}$ are countable.  We need to show that $\mathcal{V}'$ is a base for the topology of
  $\Oqco(\riem)$.
  Clearly $\mathcal{V}' \subset \mathcal{V}$.  Thus it is enough to show that for every
  $f=(f_1,\ldots,f_n) \in \Oqco(\riem)$ and $V \in \mathcal{V}$ containing $f$, there is a $V' \in \mathcal{V}'$
  such that $f \in V' \subset V$.

  Let $V_{\zeta,E,U} \in \mathcal{V}$ contain $f$.  We claim that
  there is an $n$-chart $(\eta,C) \in \mathfrak{C}$
  such that $\overline{f_i(\mathbb{D})} \subset C_i \subset E_i$ for all $i$.
  To see this, fix $i$ and observe that since $\mathfrak{B}$ is a base for $\riem$, for each point $x \in \overline{f_i(\mathbb{D})}$ there is an open set $B_{i,x} \in \mathfrak{B}$ such that
  $x \in B_{i,x} \subset E_i$.  The set $\{B_{i,x}\}_{x \in \overline{f_i(\mathbb{D})}}$ is a cover of
  $\overline{f_i(\mathbb{D})}$; since it is compact there is a finite subcover say $\{B_{i,\alpha}\}$.
  Set $C_i=\cup_\alpha B_{i,\alpha}$ and perform this procedure for each $i=1,\ldots,n$.  By construction
  the $C_i$ are non-overlapping and $C=(C_1,\ldots,C_n) \in \mathfrak{B}^n$.  It follows that
  $(\eta,C)=(\eta_1,C_1,\ldots,\eta_n,C_n) \in \mathfrak{C}$ where $\eta_i$ are the charts
  corresponding to $C_i$.  This proves the claim.

  Since $\mathfrak{O}$
  is a basis of $\Oqco$, by Theorems \ref{th:into_U_is_open} and \ref{th:left_comp_holo} (using an argument similar to the one earlier in the proof),  for each $i$ there is a $W_i \in \mathfrak{O}$
  satisfying $\eta_i \circ f_i \in W_i \subset \eta_i \circ \zeta_i^{-1}(U_i)$.
  If $g \in V'_{\eta,C,W}$ then $g_i=\eta_i^{-1} \circ h_i$ for some $h_i \in W_i$ for all $i=1,\ldots, n$
  by (\ref{eq:Oqcoriem_base_definition}).  But $h_i \in \eta_i \circ \zeta_i^{-1}(U_i)$,  so
  $g_i \in \zeta^{-1}_i (U_i)$ and hence $g \in V_{\zeta,E,U}$ by (\ref{eq:Oqcoriem_base_definition}).
  Thus $V'_{\eta,C,W} \subset V_{\zeta,E,U}$ which completes the proof.
 \end{proof}
 \begin{remark}
  In particular, $\Oqco(\riem)$ is separable since it is second countable and Hausdorff.
 \end{remark}
 We make one final simple but useful observation regarding the
 base $\mathcal{V}$.

 For a Riemann surface $\riem$ denote by
  $\mathcal{V}(\riem)$ the base for $\Oqco(\riem)$ given in Definition
  \ref{de:Oqco_base}.  For a biholomorphism $\rho:\riem \rightarrow \riem_1$ of Riemann surfaces $\riem$ and $\riem_1$, and
  for any $V \in \mathcal{V}(\riem)$, let
  \[  \rho(V) = \{ \rho \circ \phi \,:\, \phi \in V \} \]
  and
  \[  \rho(\mathcal{V}(\riem)) = \{ \rho(V)\,:\, V \in \mathcal{V} \}.  \]
 \begin{theorem}  \label{th:biholo_preserve_Oqco}
  If $\rho:\riem \rightarrow \riem_1$ is a biholomorphism between
  punctured Riemann surfaces $\riem$ and $\riem_1$ then $\rho(\mathcal{V}(\riem)) = \mathcal{V}(\riem_1)$.
 \end{theorem}
 \begin{proof}  It is an immediate consequence of Definition \ref{de:Oqco_base}
  and Theorem \ref{th:left_comp_holo} that $\rho(\mathcal{V}(\riem)) \subseteq
  \mathcal{V}(\riem_1)$.  Similarly $\rho^{-1}(\mathcal{V}(\riem_1)) \subseteq
  \mathcal{V}(\riem)$.  Since  $\rho (\rho^{-1}(\mathcal{V}(\riem_1)))
  = \mathcal{V}(\riem_1)$ and $\rho^{-1} ( \rho(\mathcal{V}(\riem))) =
  \mathcal{V}(\riem)$ the result follows.
 \end{proof}

 \begin{definition}[standard charts on $\Oqco(\riem)$] \label{de:standard_charts Oqco}
  Let $(\zeta,E)$ be an $n$-chart on $\riem$ and let $\kappa_i \subset E_i$
  be compact sets containing $p_i$.  Let $K_i=\zeta_i(\kappa_i)$.  Let
  $U_i=\{ \psi \in \Oqco \,:\, \overline{\psi(\mathbb{D})} \subset \mathrm{interior}(K_i) \}$.
  Each $U_i$ is open by Theorem \ref{th:into_U_is_open} and $U = (U_1,\ldots,U_n)$ is
  compatible with $(\zeta,E)$ so we have $V_{\zeta,E,U} \in \mathcal{V}$.  A standard
  chart on $\Oqco(\riem)$ is a map
  \begin{align*}
   T: V_{\zeta,E,U} & \longrightarrow  \Oqco \times \cdots \times \Oqco \\
   (f_1,\ldots,f_n) & \longmapsto  (\zeta_1 \circ f_1, \ldots, \zeta_n \circ f_n).
  \end{align*}
 \end{definition}
 \begin{remark}  To obtain a chart into a Hilbert space, one simply composes with $\chi$ as defined by (\ref{eq_chidefinition}). Abusing notation somewhat and defining $\chi^n$ by
 \begin{align*}
   \chi^n \circ T : V_{\zeta,E,U} &\longrightarrow \bigoplus^n \aonetwo \oplus
   \mathbb{C} \\
   (f_1,\ldots, f_n) & \longmapsto  (\chi \circ \zeta_1 \circ f_1,\ldots,
 \chi \circ \zeta_n \circ g_n )
 \end{align*}
 we obtain a chart into $\bigoplus^n \aonetwo \oplus
   \mathbb{C}$.  Since $\chi(\Oqco)$ is an open subset of
   $\aonetwo \oplus \mathbb{C}$ by Theorem \ref{th:Oqco_open_in_Oqc}, and $\chi$ defines the complex
   structure $\Oqco$, we may treat $T$ as a chart with the
   understanding that the true charts are obtained by composing with
   $\chi^n$.
 \end{remark}
 \begin{theorem} \label{th:complex_structure_Oqco}
  Let $\riem$ be a punctured Riemann surface of type $(g,n)$.
  With the atlas consisting of the standard charts of Definition \ref{de:standard_charts Oqco},
  $\Oqco(\riem)$ is a complex Hilbert manifold, locally
  biholomorphic to $\Oqco \times \cdots \times \Oqco$.
 \end{theorem}
 \begin{proof}
  We have already shown that $\Oqco(\riem)$ is Hausdorff and separable (in fact second countable).
  So we need only show that the charts above form an atlas of homeomorphisms with biholomorphic
  transition functions.

  Let $V=V_{\zeta,E,U}$ and $V'=V_{\zeta',E',U'}$ where $U$ and $U'$ are
  determined by compact sets $\kappa_i$ and $\kappa_i'$ respectively, as in Definition
  \ref{de:standard_charts Oqco}.
   With the topology from the basis $\mathcal{V}$ of Definition \ref{de:Oqco_base} the charts are automatically
  homeomorphisms.
  It suffices to show that for two standard charts
  $T:V \rightarrow \Oqco \times \cdots \Oqco$ and $T':V' \rightarrow
  \Oqco \times \cdots \times \Oqco$
  the overlap maps $T \circ T'^{-1}$  and $T' \circ T^{-1}$ are
  holomorphic.

  Assume that $V \cap V'$ is non-empty.  For $(\psi_1,\ldots,\psi_n) \in T'(V \cap V')$
  \[  T \circ {T'}^{-1} (\psi_1,\ldots,\psi_n) =
       (\zeta_1 \circ {\zeta'_1}^{-1} \circ \psi_1,\ldots,\zeta_n \circ {\zeta'_n}^{-1} \circ \psi_n). \]

  The maps $\psi_i \mapsto \zeta_i \circ {\zeta_i'}^{-1} \circ \psi_i$ are
  holomorphic maps of $\zeta_i'(V_i \cap V_i')$ by Theorem
  \ref{th:left_comp_holo} with $A=\zeta_i'(E_i \cap E_i')$, $U=\zeta_i'(\zeta_i^{-1}(U_i) \cap
  {\zeta_i'}^{-1}(U_i'))=\zeta_i' \circ \zeta_i^{-1}(U_i) \cap U_i'$, $K=\zeta_i' \circ \zeta_i^{-1}(K_i)
  \cap K_i'$ and
  $h=\zeta_i \circ {\zeta_i'}^{-1}$.
  Similarly $T' \circ T^{-1}$ is holomorphic.
 \end{proof}
 \begin{remark}[chart simplification] \label{re:chart_simplification}
  Now that this theorem is proven, we can simplify the definition of the charts.
  For an $n$-chart $(\zeta,E)$, if we let $U_i=\{ f \in \Oqco \, :\, \overline{f(\mathbb{D})} \subset
  \zeta_i(E_i) \}$, then the charts $T$ are defined on $V_{\zeta,E,U}$.  It is easy to show
  that $T$ is a biholomorphism on $V_{\zeta,E,U}$, since any $f \in V_{\zeta,E,U}$ is contained
  in some $V_{\zeta,E,W} \subset V_{\zeta,E,U}$ which satisfies Definition \ref{de:standard_charts Oqco},
  and thus $T$ is a biholomorphism on $V_{\zeta,E,W}$ by Theorem \ref{th:complex_structure_Oqco}.
 \end{remark}
 \begin{remark}[standard charts on $\Oqc(\riem)$]
  \label{re:standard_charts_Oqc} A standard chart on $\Oqc(\riem)$ is
  defined in the same way as Definition \ref{de:standard_charts
  Oqco} and its preamble, by replacing $\Oqco$ with $\Oqc$
  everywhere.  Furthermore with this atlas $\Oqc(\riem)$
  is a complex Banach manifold \cite{RSnonoverlapping}.
 \end{remark}

 Finally, we show that the inclusion map $I:\Oqco \rightarrow \Oqc$
 is holomorphic.
 \begin{theorem} \label{th:Oqco_holo_in_Oqc} The complex manifold
  $\Oqco(\riem)$ is holomorphically contained in $\Oqc(\riem)$ in the
  sense that the inclusion map $I:\Oqco(\riem) \rightarrow \Oqc(\riem)$
  is holomorphic.
 \end{theorem}
 \begin{proof}  This follows directly from the construction of the
  charts on $\Oqc(\riem)$.  Let $T:V \rightarrow \Oqc
  \times \cdots \times \Oqc$ be a standard chart on $\Oqc(\riem)$
  as specified in Remark \ref{re:standard_charts_Oqc}.  Let $U=T(V)$
  and $U_0=U \cap \Oqco \times \cdots \times \Oqco$.  Let $V_0=T^{-1}(U_0)$.
  The map $\left. T \right|_{V_0}$ is a chart on $V_0 \subseteq \Oqco(\riem)$,
  so it is holomorphic in the refined setting.
  Since the inclusion map $\iota:U_0 \rightarrow U$ is holomorphic by Theorem
  \ref{th:Oqco_open_in_Oqc}, the inclusion map $I=T^{-1} \circ \iota
  \circ \left(\left. T \right|_{V_0}\right)$ is holomorphic on $V_0$.  Since $\Oqco(\riem)$ is covered by
  charts of this form, $I$ is holomorphic.
 \end{proof}
\end{subsection}
\end{section}
\begin{section}{The rigged Teichm\"uller space is a Hilbert manifold} \label{se:rigged_Teich_space}

In \cite{RS05}, two of the authors proved that the Teichm\"uller space of a bordered surface is
 (up to a quotient by a discrete group) the same as a certain rigged Teichm\"uller space whose corresponding rigged moduli space appears
 naturally in two-dimensional conformal field theory \cite{FriedanShenker, Huang, Segal}.  We will use this fact to define a Hilbert
 manifold structure on the refined Teichm\"uller space in Section \ref{se:refined_Teich_space}.

 First we must define an atlas on rigged Teichm\"uller space, and this is the main task of the current section. We will achieve this by using universality of the universal \teich curve together with a  variational
 technique called \textit{Schiffer variation} as adapted to the quasiconformal Teichm\"uller setting by Gardiner \cite{Gardiner} and Nag \cite{NagSchiffer, Nagbook}. This overall approach was first developed in the thesis of the
 first author \cite{RadThesis} for the case of analytic riggings.

\begin{subsection}{Definition of rigged Teichm\"uller space}

We first recall the definition of the usual \teich space. The reader is referred to Section \ref{se:refined_quasisymmetries} for terminology regarding Riemann surfaces.
\begin{definition}
\label{de:Teichspace}
Fix a Riemann surface $X$ (of any topological type).  Let
$$
T(X) = \{ (X,f,X_1) \} / \sim
$$
   where
   \begin{enumerate}
    \item $X_1$ is a Riemann surface of the same topological type as $X$.
    \item $f: X \rightarrow X_1$ is a quasiconformal homeomorphism (the \textit{marking map}).
   \item the equivalence relation $(\sim)$ is defined by $(X,f_1,X_1) \sim (X,f_2,X_2)$ if and only if there exists a biholomorphism
   $\sigma:X_1 \rightarrow X_2$ such that $f^{-1}_2
   \circ \sigma \circ f_1$ is homotopic to the identity rel boundary.
  \end{enumerate}
  The term \textit{rel boundary} means that the homotopy is the identity on the boundary throughout the homotopy.
\end{definition}
It is a standard fact of \teich theory (see for example \cite{Nagbook}) that if $X$ is a punctured surface of type $(g,n)$ then $T(X)$ is a complex manifold of dimension $3g-3+n$, and if $X$ is a bordered surface of type $(g,n)$ then $T(X)$ is an infinite-dimensional complex Banach manifold.

 Using the set $\Oqco(\riem)$ we now define the \textit{(refined) rigged Teichm\"uller space}, denoted by $\ttildeop(\riem)$.
 \begin{definition} \label{de:rigged_Teich0}
   Fix a punctured Riemann surface of type $(g,n)$.  Let
   \[  \ttildeop(\riem) = \{ (\riem,f,\riem_1,\phi) \}
   / \sim \]
   where
   \begin{enumerate}
    \item $\riem_1$ is a punctured Riemann surface of type $(g,n)$
    \item $f: \riem \rightarrow \riem_1$ is a quasiconformal homeomorphism
    \item $\phi \in \Oqco(\riem_1)$.
   \item
   Two quadruples  are said
   to be equivalent, denoted by  $(\riem,f_1,\riem_1,\phi_1) \sim (\riem,f_2,\riem_2,\phi_2)$, if and only if there exists a biholomorphism
   $\sigma:\riem_1 \rightarrow \riem_2$ such that $f^{-1}_2
   \circ \sigma \circ f_1$ is homotopic to the identity rel boundary and $\phi_2
   = \sigma \circ \phi_1$.
  \end{enumerate}
  The equivalence class of $(\riem,f_1,\riem_1,\phi_1)$ will be denoted $[\riem,f_1,\riem_1,\phi_1]$
 \end{definition}
 Condition (2) can be stated in two alternate ways.  One is to require that $f$ maps the
 compactification of $\riem$ into the compactification of $\riem_1$, and takes the
 punctures of $\riem$ to the punctures of $\riem_1$ (now thought of as marked points).  The other is to say simply that $f$ is a quasiconformal map between $\riem$
 and $\riem_1$.  Since $f$ is quasiconformal its extension to the compactification
 will take punctures to punctures. Thus condition (2) does  not explicitly mention the punctures.

 In \cite{RS05}, two of the authors defined a rigged Teichm\"uller space $\ttildep(\riem)$ obtained by replacing
 $\Oqco(\riem_1)$ with $\Oqc(\riem_1)$ in the above definition. It was demonstrated in \cite{RS05} that $\ttildep(\riem)$ has a complex Banach manifold structure, which comes from the fact that it is a quotient of the Teichm\"uller space of a bordered surface by a properly discontinuous, fixed-point free group of biholomorphisms.  In \cite{RS_fiber} we demonstrated that it is fibred over $T(\riem)$, where the fiber over a point $[\riem,f_1,\riem_1]$ is biholomorphic to $\Oqc(\riem_1)$.  Furthermore, the complex structure of $\Oqc(\riem_1)$ is compatible with the complex structure that the fibres inherit from $\ttildep(\riem)$.

 This notion of a \textit{rigged \teich space} was first defined,  in the case of analytic riggings, by one of the authors in \cite{RadThesis}, and it was used to obtain a complex Banach manifold structure on the analytically rigged moduli space. However, in this case the connection to the complex structure of the infinite-dimensional \teich space of bordered surfaces can not be made.

 From now on, any punctured Riemann surface is assumed to satisfy $2g+2-n>0$.
 We would now like to demonstrate that $\ttildeop(\riem)$ has a natural complex Hilbert manifold structure which arises from $\Oqco(\riem)$, and that this also passes to the rigged Riemann moduli space. In Section \ref{se:refined_Teich_space}, we will use it to construct a complex Hilbert manifold structure on a refined Teichm\"uller space of a bordered surface.
 To accomplish these tasks we use a natural coordinate system developed in \cite{RadThesis, RS_fiber}, which is based on Gardiner-Schiffer variation and the complex structure on $\Oqc(\riem)$.  We will refine these coordinates to $\ttildeop(\riem)$.

We end this section with a basic result concerning the above definition.
Since $\riem$ satisfies $2g+2-n>0$ we have the following well known theorem \cite{Nagbook}.
\begin{theorem}
\label{th:homotopic_zero}
If $\sigma : \riem \to \riem$ is a biholomorphism that is homotopic to the identity then $\sigma$ is the identity.
\end{theorem}

\begin{corollary}
\label{co:rigged_equiv}
If $[\riem, f_1, \riem_1, \phi_1] = [\riem, f_1, \riem_1, \phi_2] \in \ttildeop(\riem)$ then $\phi_1 = \phi_2$.
\end{corollary}
\end{subsection}
\begin{subsection}{Marked families}
In this section we collect some standard definitions and facts about marked holomorphic families of Riemann surfaces and the universality of the \teich curve. These will play a key role in the construction of an atlas on rigged Teichm\"uller space. A full treatment appears in \cite{EarleFowler}, and also in the books \cite{Hubbard, Nagbook}.

\begin{definition} \label{de:holo_family}
A holomorphic family of complex manifolds is a pair of connected complex manifolds $(E,B)$ together with a surjective holomorphic map  $\pi : E \to B$ such that (1) $\pi$ is topologically a locally trivial fibre bundle, and (2) $\pi$ is a split submersion (that is, the derivative is a surjective map whose kernel is a direct summand).
\end{definition}

\begin{definition}  A \textit{morphism of holomorphic families} from $(E',B')$ and $(E,B)$ is a pair of holomorphic maps $(\alpha, \beta)$ with $\alpha : B' \to B$ and $\beta: E' \to E$ such that
$$
\xymatrix{
E' \ar[r]^{\beta} \ar[d]_{\pi'} & E \ar[d]^{\pi} \\
B' \ar[r]^{\alpha} & B
}
$$
commutes, and for each fixed $t\in B'$, the restriction of $\beta$ to the
fibre $\pi'^{-1}(t)$   is a biholomorphism onto $\pi^{-1}(\alpha(t))$.
\end{definition}

Throughout, $(E,B)$ will be a holomorphic family of Riemann surfaces; that is, each fibre $\pi^{-1}(t)$ is a Riemann surface.  Moreover, since our trivialization will always be global we specialize the standard definitions (see \cite{EarleFowler}) to this case in what follows.

Let $\riem$ be a punctured Riemann surface of type $(g,n)$.  This fixed surface $\riem$ will serve as a model of the fibre.

\begin{definition} \label{de:strong_triv}\text{}
 \begin{enumerate}
\item  A  \textit{global trivialization} of $(E,B)$ is a homeomorphism $\theta: B \times \riem \to E$ such that $\pi(\theta(t,x)) = t$ for all $(t,x) \in B\times \riem$.
\item A global trivialization $\theta$ is a \textit{strong trivialization} if for fixed $x \in \riem$, $t\mapsto \theta(t,x)$ is holomorphic, and for each $t\in B$, $x \mapsto \theta(t,x)$ is a quasiconformal map from $\riem$ onto $\pi^{-1}(t)$.
\item $\theta : B\times \riem \to E$ and $\theta' : B\times \riem \to E$ are \textit{compatible} if and only if $\theta'(t,x) = \theta(t,\phi(t,x))$ where for each fixed $t$, $\phi(t,x) : \riem \to \riem$ is a quasiconformal homeomorphism that is homotopic to the identity rel boundary.
\item A \textit{marking} $\mathcal{M}$ for $\pi : E \to B$ is an equivalence class of compatible strong trivializations.
\item A \textit{marked family of Riemann surfaces} is a holomorphic family of Riemann surfaces with a specified marking.
\end{enumerate}
\end{definition}

\begin{remark} Let $\theta$ and $\theta'$ be compatible strong trivializations. For each fixed $t\in B$, $[\riem, \theta(t,\cdot), \pi^{-1}(t)] = [\riem, \theta'(t,\cdot), \pi^{-1}(t)]$ in $T(\riem)$ (see Definition \ref{de:Teichspace}). So a marking specifies a \teich equivalence class for each $t$.
\end{remark}

We now define the equivalence of marked families.

\begin{definition}
A \textit{morphism of marked families} from $\pi':E' \to B'$ to $\pi: E \to B$ is a pair of holomorphic maps $(\alpha,\beta)$ with $\beta : E' \to E$ and $\alpha: B' \to B$ such that
\begin{enumerate}
\item $(\alpha,\beta)$ is a morphism of holomorphic families, and
\item the markings $B' \times \riem \to E$ given by  $\beta(\theta'(t,x))$ and $\theta(\alpha(t), x)$ are compatible.
\end{enumerate}
\end{definition}
The second condition says that $(\alpha,\beta)$ preserves the marking.

\begin{remark}[relation to Teichm\"uller equivalence]
\label{re:marking_preserving}
Define $E = \{(s,Y_s)\}_{s\in B}$ and $E' = \{(t,X_t)\}_{t\in B'}$ to be marked families of Riemann surfaces with markings $\theta(s, x) = (s, g_s(x))$ and $\theta'(t, x) = (t, f_t(x))$ respectively. Say $(\alpha,\beta)$ is a morphism of marked families, and define $\sigma_t$ by
$\beta(t, y) = (\alpha(t), \sigma_t(y))$. Then $\beta(\theta'(t,x)) = (\alpha(t), \sigma_t(f_t(x)))$ and $\theta(\alpha(t), x) = (\alpha(t), g_{\alpha(t)}(x))$. The condition that $(\alpha, \beta)$ is a morphism of marked families is simply that $\sigma_t \circ f_t$ is homotopic rel boundary to $g_{\alpha(t)}$.
That is, when $s = \alpha(t)$,  $[\riem, f_t, X_t] = [\riem, g_s, Y_s]$  via the biholomorphism $\sigma_t : X_t \to Y_s$.
\end{remark}

The universal \teich curve, denoted by $\pi_T : \mathcal{T}(\riem) \to T(\riem)$,
is a marked holomorphic family of Riemann surfaces with fibre model $\riem$. The following universal property
of $\mathcal{T}(\riem)$ (see \cite{EarleFowler, Hubbard, Nagbook}) is all that we need for our purposes.
\begin{theorem}[Universality of the \teich curve]
\label{th:UniversalProperty}
Let $\pi : E \to B$ be a marked holomorphic family of Riemann surfaces with fibre model $\riem$ of type $(g,n)$ with $2g-2+n>0$, and trivialization $\theta$.
Then there exists a unique map $(\alpha, \beta)$ of marked families from $\pi : E \to B$ to $\pi_T : \mathcal{T}(\riem) \to T(\riem)$. Moreover, the canonical ``classifying" map $\alpha: B \to T(\riem)$ is given by $\alpha(t) = [\riem, \theta(t,\cdot), \pi^{-1}(t)]$.
\end{theorem}

\end{subsection}
\begin{subsection}{Schiffer variation}
The use of Schiffer variation to construct holomorphic coordinates for Teichm\"uller space by using quasiconformal deformations is due to Gardiner \cite{Gardiner} and Nag \cite{NagSchiffer, Nagbook}. We review the construction in some detail, as it will be used in a crucial way.

Let $B_R = \{z \in \mathbb{C} : |z|<R\}$, and for $r<R$ let $\mathbb{A}(r,R)= \{z \in \mathbb{C} : r<|z|<R \}$ as before.
Choose  $r$ and $R$ such that $0<r<1<R$.
Let $\riem$ be a (possibly punctured) Riemann surface and $\xi: U \to \mathbb{C}$ be  local holomorphic coordinate on an open connected set $U \subset \riem$ such that $B_R \subset \text{Image}(\xi)$. Let $D = \xi^{-1}(\mathbb{D})$, which we call a parametric disk.

Define $v^{\epsilon} : \mathbb{A}(r,R ) \longrightarrow \Bbb{C}$ by $ v^{\epsilon}(z) = z + \epsilon/z$ where $\epsilon \in \mathbb{C}$.
For $|\epsilon|$ sufficiently small $v$ is a biholomorphism onto its image. Let $D^{\epsilon}$ be the interior of the analytic Jordan curve $v^{\epsilon}(\partial \mathbb{D})$. We regard $\overline{D^{\epsilon}}$ as a bordered Riemann surface (with the standard complex structure
inherited from $\Bbb{C}$) with analytic boundary parametrization
given by $v^{\epsilon} : S^1 \to  \partial D^{\epsilon}$. We also have the Riemann surface $\riem \setminus D$ with the boundary analytically parametrized by $\xi^{-1} |_{S^1} : S^1 \to \partial (\riem \setminus D)$.

We now sew $\overline{D^{\epsilon}}$ and $\riem \setminus D$ along their boundaries by identifying $x \in \partial (\riem \setminus D)$ with $x' \in \partial \overline{D^{\epsilon}}$ if and only if
$x' = (v^{\epsilon} \circ \xi)(x)$. Let
$$
\riem^{\epsilon} = (\riem \setminus D) \sqcup \overline{D^{\epsilon}} \ / \ \text{boundary
identification}
$$
and we say this Riemann surface is obtained from $\riem$ by
\textit{Schiffer variation} of complex structure on $D$. Let
$$\iota^{\epsilon} : \riem \setminus D  \longrightarrow \riem^{\epsilon} \quad  \text{and} \quad \iota_D^{\epsilon} : D^{\epsilon} \longrightarrow \riem^{\epsilon}$$
be the holomorphic inclusion maps. With a slight abuse of notation we could use the identity map in place of $\iota^{\epsilon}$, however the extra notation will make
the following exposition clearer.

Define $w^{\epsilon}: \overline{\mathbb{D}} \to \overline{D^{\epsilon}}$ by
$
w^{\epsilon}(z) = z + \epsilon \overline{z} $. Note that $w^{\epsilon}$ is a homeomorphism, and on the boundary $v^{\epsilon}=w^{\epsilon}$.
Define the quasiconformal homeomorphism $\nu^{\epsilon} : \riem \to \riem^{\epsilon}$ by
$$
\nu^{\epsilon}(x) =
\begin{cases}
\iota^{\epsilon}(x), &  x \in \riem \setminus D \\
(\iota_D^{\epsilon} \circ w \circ \xi)(x), & x \in D.
\end{cases}
$$
So we now have a point
$[\riem,\nu^{\epsilon},\riem^{\epsilon}] \in T(\riem)$ obtained by
Schiffer variation of the base point $[\riem, \text{id},\riem]$.

To get coordinates on $T(\riem)$ we perform Schiffer variation on $d$ disks
where $d = 3g-3+n$ is the complex dimension of $T(\riem)$. Let $(D_1, \ldots, D_d)$ be $d$ disjoint parametric disks on
$\riem$, where $D_i = (\xi_i)^{-1}(\mathbb{D})$ for suitably chosen local
coordinates $\xi_i$. Let $D = D_1 \cup \cdots \cup D_d$ and
let $\epsilon = (\epsilon_1,\ldots, \epsilon_d) \in \Bbb{C}^d$.
Schiffer variation can be performed on the $d$ disks to get a new
surface which we again denote by $\riem^{\epsilon}$.
The map $\nu^{\epsilon}$ becomes
\begin{equation}
\label{eq:multi_schiffer}
\nu^{\epsilon}(x) =
\begin{cases}
\iota^{\epsilon}(x), &  x \in \riem \setminus D \\
(\iota_D^{\epsilon} \circ w^{\epsilon_i} \circ \xi_i)(x), &  x \in D_i \ , i=1,\ldots,d.
\end{cases}
\end{equation}

The following theorem is the main result on Schiffer variation  \cite{Gardiner, Nagbook}.
Let $\Omega \subset \Bbb{C}^d$ be an open neighborhood of $0$ such that Schiffer variation is defined for $\epsilon \in \Omega $.
Define
\begin{align}
\label{Schiffer_map}
\mathcal{S} : \Omega & \longrightarrow T(\riem) \\
\epsilon & \longmapsto [\riem, \nu^{\epsilon}, \riem^{\epsilon}]. \nonumber
\end{align}

\begin{theorem}
\label{th:Schiffer}
Given any $d$ disjoint parametric disks on $\riem$, it is possible to choose the local parameters $\xi_i$ such that on some open neighborhood $\Omega$ of $0\in \mathbb{C}^n$,  $\mathcal{S}: \Omega \to \mathcal{S}(\Omega)$ is a biholomorphism. That is, the parameters $(\epsilon_1, \ldots, \epsilon_d)$ provide local holomorphic coordinates for \teich space in a neighborhood of $[\riem, \text{id}, \riem]$
\end{theorem}
It is important to note that we are free to choose the domains $D_i$ on which the Schiffer variation is performed.

By a standard change of base point argument we can use Schiffer variation to produce a neighborhood of any point
$[\riem, f, \riem_1] \in T(\riem)$. Performing Schiffer variation on $\riem_1$ gives a neighborhood $\mathcal{S}(\Omega)$ of $[\riem_1, \text{id}, \riem_1] \in T(\riem_1)$. Consider the change of base point biholomorphism (see \cite[Sections 2.3.1
and 3.2.5]{Nagbook}) $f^* : T(\riem_1) \to T(\riem)$ given by $f^*([\riem_1, g, \riem_2] = [\riem, g \circ f, \riem_2]$. Then $f^* \circ \mathcal{S}$ is a biholomorphism onto its image $f^*(\mathcal{S}(\Omega)) = \{[\riem, \nu^{\epsilon} \circ f, \riem_1^{\epsilon}] \}$ which is a neighborhood of $[\riem, f, \riem_1] \in T(\riem)$.

Thus, denoting $f^* \circ \mathcal{S}$ itself by $\mathcal{S}$, the
Schiffer variation
\begin{align}
\label{Schiffer_map2}
\mathcal{S} : \Omega & \longrightarrow T(\riem) \\
\epsilon & \longmapsto [\riem, \nu^{\epsilon} \circ f, \riem^{\epsilon}]. \nonumber
\end{align}
produces a neighborhood of $[\riem, f, \riem_1] \in T(\riem)$.

\end{subsection}

\begin{subsection}{Marked Schiffer family}
\label{se:MarkedSchifferFamily}
Fix a point $[\riem, f, \riem_1] \in T(\riem)$.  We will show that
Schiffer variation on $\riem_1$ produces a marked holomorphic family of Riemann surfaces with fiber $\riem_1^{\epsilon}$ over the point $\epsilon$ and marking $\nu^{\epsilon} \circ f$.
Since this construction does not appear in the literature, we present it here in some detail as it is an essential ingredient in our later proofs. An efficient way to describe the family is to do the sewing for all $\epsilon$ simultaneously.

For $i = 1,\ldots, d$, let $\Omega_i$ be connected open neighborhoods of $0 \in \mathbb{C}$ such that  $\Omega = \Omega_1 \times \cdots \times \Omega_d$ is an open subset of $\mathbb{C}^d$  for which Schiffer variation is defined and Theorem \ref{th:Schiffer} implies that $\mathcal{S}:\Omega \to \mathcal{S}(\Omega) \subset T(\riem)$ is a biholomorphism.

Define, for each $i = 1,\ldots, d$,
\begin{align*}
w_i : \Omega_i \times \mathbb{D} & \longrightarrow \mathbb{C} \times \mathbb{C} \\
(\epsilon_i, z) & \longmapsto (\epsilon_i, w^{\epsilon_i}(z)),
\end{align*}
\begin{align*}
v_i : \Omega_i \times A_r^1 & \longrightarrow \mathbb{C} \times \mathbb{C} \\
(\epsilon_i, z) & \longmapsto (\epsilon_i, v^{\epsilon_i}(z)),
\end{align*}
and let
$$
Y_i = w_i(\Omega_i \times \mathbb{D}).
$$
Since $w_i$ is a homeomorphism, $Y_i$ is open and so inherits a complex manifold structure from $\mathbb{C} \times \mathbb{C}$. Note that for fixed $\epsilon_i$, $\{z \st (\epsilon_i , z) \in Y_i \} = D^{\epsilon_i}$.

With $r<1$ as in the construction of Schiffer variation, let $D^r_i = \xi_i^{-1}(B(0,r))$ and $D^r = D^r_1 \cup \cdots \cup D^r_d$.
Let
$$
X = \Omega \times (\riem_1 \setminus \overline{D^r})
$$
and endow it with the product complex manifold structure. Define the map
\begin{align*}
\rho_i : \Omega \times (D_i \setminus \overline {D^r_i}) & \longrightarrow v(\Omega \times A_r^1) \\
(\epsilon_i, x) & \longmapsto (\epsilon, v^{\epsilon_i}(\xi_i(x)).
\end{align*}
From the definition of $v^{\epsilon_i}$ it follows directly that $\rho_i$ is a biholomorphism from an open subset of $X$ to an open subset of $Y_i$.

Using the standard gluing procedure for complex manifolds (see for example \cite[page 170]{FG_GTM213}) we can make the following definition.
\begin{definition}
\label{de:schiffer_family}
Let $S(\Omega, D)$ be the complex manifold obtained by gluing $X$ to  $Y_1,\dots, Y_d$ using the biholomorphisms $\rho_1,\ldots,\rho_d$.
\end{definition}
The inclusions $\iota_X: X \hookrightarrow S(\Omega, D)$ and $\iota_{Y_i} : Y_i \hookrightarrow S(\Omega, D)$ are holomorphic. Moreover, since $r$ just determines the size of the overlap, $S(\Omega, D)$ is independent of $r$.

Equivalently, we can think of gluing $\Omega \times (\riem_1 \setminus D)$ and $w(\Omega \times \overline{\mathbb{D}})$ using the $\rho_i$ restricted to $\Omega \times \partial D_i$ to identify the boundary components.  For each fixed $\epsilon$ this gluing is precisely that used to define $\riem_1^{\epsilon}$.
So we see that
$$
S(\Omega, D) = \{ (\epsilon, x) : \epsilon \in \Omega, x \in \riem_1^{\epsilon} \}.
$$

Define the projection map
\begin{align*}
\pi_S: S(\Omega,D) & \to \Omega \\
(\epsilon, x) & \mapsto \epsilon
\end{align*}
and the trivialization
\begin{align} \label{eq:trivialization_definition}
\theta : \Omega \times \riem &\to S(\Omega,D)\\
(\epsilon,x) & \mapsto (\epsilon, (\nu^{\epsilon} \circ f)(x)). \nonumber
\end{align}

It is immediate that $\pi_S$ is onto, holomorphic and defines a topologically trivial bundle.

\begin{definition}
We call $\pi_S:S(\Omega,D) \rightarrow \Omega$ with trivialization $\theta$
a marked Schiffer family.
\end{definition}
We will have use for explicit charts on $S(\Omega, D)$, but only on the part that is disjoint from the Schiffer variation.
Let $(U,\zeta)$ be a chart on $\riem_1 \setminus \overline{D}$. Recall that $\iota^{\epsilon} : \riem_1 \setminus \overline{D} \to \riem_1^{\epsilon}$ is the holomorphic inclusion map. Let
\begin{equation}
\label{family_open_set}
\tilde{U} = \{ (\epsilon, x) \st \epsilon \in \Omega, x \in \iota^{\epsilon}(U) \} \subset S(\Omega, D)
\end{equation}
and define
\begin{align}
\label{family_chart_map}
\tilde{\zeta}: \tilde{U} & \longrightarrow \mathbb{C} \times \mathbb{C} \\
(\epsilon, x) &\longmapsto \left( \epsilon, (\zeta \circ (\iota^{\epsilon})^{-1})(x)\right). \nonumber
\end{align}
Then $(\tilde{\zeta}, \tilde{U})$ is a holomorphic chart on $S(\Omega, D)$.

Note that with a slight of abuse of notation we could simply write $\tilde{U} = \Omega \times U$ and define $\tilde{\zeta}$ by $(\epsilon, x) \mapsto ( \epsilon, \zeta(x))$, but we will refrain from doing so.

\begin{theorem}
\label{th:SchifferMarkedFamily}
A marked Schiffer family is a marked holomorphic family of Riemann surfaces.
\end{theorem}

\begin{proof} We must check the conditions in Definitions \ref{de:holo_family} and \ref{de:strong_triv}.

Because $\nu^{\epsilon}$ is a quasiconformal homeomorphism, $\theta(\epsilon, z)$ is a homeomorphism, and for fixed $\epsilon$, $\theta(\epsilon, z)$ is quasiconformal.
Next, we show  that for fixed $x$, $\theta(\epsilon, x)$ is holomorphic in $\epsilon$.
\begin{enumerate}
\item If $x \in \riem \setminus f^{-1}(D_i)$ then $\theta(\epsilon, x) \in \iota_{X}(X)$. Let $\zeta$ and $\zeta'$ be a local coordinates in  neighborhoods of $x$ and $f(x)$ respectively, and let $z = \zeta(x)$. Use these to form the product charts on $\Omega \times \riem$ and $X$. From the definition of $\nu^{\epsilon}$ (see (\ref{eq:multi_schiffer})) it follows directly that in terms of local coordinates $\theta(\epsilon, x)$ is the map $(\epsilon, z) \mapsto (\epsilon,  (\zeta' \circ f \circ \zeta^{-1}(z))$. Since the second entry is independent of $\epsilon$ the map is clearly holomorphic in $\epsilon$.
\item If $x \in f^{-1}(D_i)$ then $\theta(\epsilon, x) \in \iota_{Y_i}(Y_i)$. Let $\eta$ be a coordinate map on $f^{-1}(D_i)$ and let $z = \eta(x)$. Use $(\epsilon, t) \mapsto (\epsilon, \zeta(t))$ as the product chart on $\Omega \times f^{-1}(D_i)$. Let $y = \xi_i \circ f \circ \eta^{-1}(z)$ which is independent of $\epsilon$. Then in terms of local coordinates, $\theta$ becomes $(\epsilon, z) \mapsto (\epsilon, w^{\epsilon_i}(y))$.
Since $w^{\epsilon_i}(y) = y + \epsilon_i \bar{y}$, it is certainly holomorphic in $\epsilon$ for fixed $y$.
\end{enumerate}
Conditions (1) and (2) of Definition \ref{de:strong_triv} are thus satisfied. It remains to prove condition (2) of Definition \ref{de:holo_family}.

Because $\theta(\epsilon,x)$ is holomorphic in $\epsilon$, $S(\Omega, D)$ has a holomorphic section though every point.  This implies that $\pi_S: S(\Omega, D) \to \Omega$ is a holomorphic split submersion (see for example \cite[section 1.6.2]{Nagbook}, and also \cite[Section 6.2]{Hubbard} for an alternate definition of marked families).

So $\theta(\epsilon, z)$ is a strong trivialization and hence $S(\Omega, D)$ is a marked family of Riemann surfaces.
\end{proof}

We will need the following lemma regarding maps between marked Schiffer families.
We consider two Schiffer families, whose corresponding neighborhoods in Teichm\"uller space intersect on an open se,t and the morphism between these families.

For $i=1,2$, let $\pi_1: S_i(\Omega_i, D_i) \rightarrow
 \Omega_i$ be marked Schiffer
 families based at $[\riem,f_i,\riem_i]$. Let $\mathcal{S}_i: \Omega_i \rightarrow T(\Sigma)$ be the corresponding variation maps defined by (\ref{Schiffer_map2}), and assume that $\mathcal{S}_1(\Omega_1) \cap \mathcal{S}_2(\Omega_2)$ is non-empty. Let $N$ be any connected component of $\mathcal{S}_1(\Omega_1) \cap \mathcal{S}_2(\Omega_2)$, and let
 $\Omega_i'= \mathcal{S}_i^{-1}(N)$.

  Consider the marked Schiffer families $S_i(\Omega'_i, D_i) = \pi_i^{-1}(\Omega'_i)$
with trivializations
$\theta_i : \Omega'_i \times \riem \to S_i(\Omega'_i, D_i)$ defined by $\theta_i(\epsilon, x) = (\epsilon, (\nu^{\epsilon}_i \circ f_i)(x))$.
For ease of notation we write $S'_i = S_i(\Omega'_i,  D_i)$.

 Recall that throughout we are assuming that $\riem$ is of type $(g,n)$ with $2g-2+n>0$.

\begin{lemma} \label{le:marked_families_lemma}
There is a unique invertible morphism of marked families $(\alpha, \beta)$ from $\pi_1 : S'_1 \to \Omega'_1$ to $\pi_2 : S'_2 \to \Omega'_2$. In particular, the following hold:
 \begin{enumerate}
 \item
 There is a unique map
 $\alpha:\Omega_1' \rightarrow \Omega_2'$
 such that $[\riem,\nu_1^\epsilon \circ f_1, \riem_1^\epsilon]=[\riem,\nu_2^{\alpha(\epsilon)} \circ f_2,\riem_2^{\alpha(\epsilon)}]$, and $\alpha$ is a biholomorphism.
 \item For each $\epsilon \in \Omega_1'$, there is a unique biholomorphism
 $\sigma_\epsilon:\riem_1^\epsilon \rightarrow \riem_2^{\alpha(\epsilon)}$
 realizing the equivalence in (1).
 \item The function $\beta(\epsilon,z)=( \alpha(\epsilon), \sigma_\epsilon(z))$ is a
 biholomorphism on $\pi_1^{-1}(\Omega_1') \subset S_1(\Omega_1, D_1)$.
 \end{enumerate}
\end{lemma}
\begin{proof}
By Theorem \ref{th:UniversalProperty} there are unique mappings of marked families $(\alpha_i, \beta_i)$ from $\pi_i : S'_i \to \Omega'_i$ to $\pi_T : \mathcal{T}(\riem) \to T(\riem)$. By Theorem \ref{th:Schiffer} and the fact that $\alpha_i = \mathcal{S}_i$ from equation (\ref{Schiffer_map2}) we see that $\alpha_i$ is injective. Since $\beta_i$ is injective fibrewise and $\alpha_i \circ \pi_i  = \pi_T \circ \beta_i$ it follows that $\beta_i$ is injective. So $\alpha_i$ and $\beta_i$ are biholomorphisms onto their images since they are holomorphic and injective functions on finite-dimensional complex spaces.

Let  $\alpha = \alpha_2^{-1} \circ \alpha_1$ and $\beta = \beta_2^{-1} \circ \beta_1$; these are biholomorphisms from $\Omega'_1 \to \Omega'_2$ and $S'_1 \to S'_2$ respectively. Then $(\alpha, \beta)$ is the unique map of marked families from $\pi_1 : S'_1 \to \Omega'_1$ to $\pi_2 : S'_2 \to \Omega'_2$, and has inverse $(\alpha^{-1}, \beta^{-1})$.

The proof of (1) is completed by noting that the equation
$$
[\riem,\nu_1^\epsilon \circ f_1, \riem_1^\epsilon]=[\riem,\nu_2^{\alpha(\epsilon)} \circ f_2,\riem_2^{\alpha(\epsilon)}]
$$
is precisely  $\alpha_1(\epsilon) = \alpha_2(\alpha(\epsilon))$, which is true by the definition of $\alpha$.

Because $\beta$ restricted to the fibres is a biholomorphism and $\alpha_1 \circ \pi_1 = \pi_2 \circ \beta$ we can write (as in Remark \ref{re:marking_preserving}) $\beta$ in the form
$$
\beta(\epsilon, x) = (\alpha(\epsilon), \sigma_{\epsilon}(x))
$$
where
$
\sigma_{\epsilon}: \riem_1^{\epsilon} \to \riem_2^{\alpha(\epsilon)}
$
is a biholomorphism.

Since $2g-2+n>0$, the uniqueness in (2) follows directly from Theorem \ref{th:homotopic_zero}.

We have already proved that $\beta : S'_1 \to S'_2$ is a biholomorphism and so (3) is proved.
\end{proof}
\begin{remark}  Part (3) of the above lemma is the reason for introducing
the theory of marked families.  Without this theory, it is impossible to prove (or even formulate the notion of) holomorphicity in $\epsilon$ of the map $\sigma_\epsilon$ realizing the Teichm\"uller equivalence.  The holomorphicity
in $\epsilon$ is necessary for the proof that the transition functions on the rigged \teich space are
biholomorphisms (Theorem \ref{th:tilde_top_base} below).
\end{remark}
\end{subsection}

\begin{subsection}{Topology and atlas for the rigged Teichm\"uller space}
We will now give the rigged Teichm\"uller space a Hilbert manifold structure.

We begin by defining a base for the topology.
Let $\riem$ be a punctured Riemann surface of type $(g,n)$.
We fix a point $[\riem, f, \riem_1] \in T(\riem)$.
Let $(\zeta,E)$ be an $n$-chart on $\riem_1$, let $U \subset \Oqco \times \cdots \times \Oqco$ be compatible
with $(\zeta,E)$, and let $V=V_{\zeta,E,U}$ (defined in equation (\ref{eq:Oqcoriem_base_definition})).
\begin{definition} \label{de:Schiffer_variation_compatible}
 We say that a marked Schiffer family $S(\Omega,D)$ is compatible with an $n$-chart $(\zeta,E)$ if the closure of
 each disc $D_i$ is disjoint from the closure of $E_j$ for all $i$ and $j$.
\end{definition}

For any punctured Riemann surface $\riem'$ denote by $\mathcal{V}(\riem')$
  the basis of $\Oqco(\riem')$ as in Definition \ref{de:Oqco_base}.
\begin{lemma} \label{le:nu_preserves_base}
  Let $S(\Omega,D)$ be a marked Schiffer family based at
  $[\riem,f,\riem_1]$ and let $V \in \mathcal{V}(\riem_1)$.
  If $S(\Omega,D)$ is compatible with $V$ then $\nu^\epsilon(V)=\{ \nu^\epsilon
  \circ \phi \,:\, \phi \in V \}$ is an element of $\mathcal{V}(\riem_1^\epsilon)$.
\end{lemma}
\begin{proof}
 Writing $V$ in terms of its corresponding $n$-chart $(\zeta,E)$ and
 $W \subset \Oqco \times \cdots \times \Oqco$,
 this is an immediate consequence of
 the fact that $\nu^\epsilon$ is holomorphic on
 the sets $E_i$.
\end{proof}

Define the set
\[  F(V,S,\Delta)= \{ [\riem,\nu^\epsilon \circ f,\riem_1^\epsilon,\phi]\,:\,
    \epsilon \in \Delta, \phi \in \nu^\epsilon(V)\}  \]
where $V \in \mathcal{V}$, $S=S(\Omega,D)$ is a Schiffer variation compatible with
$V$, and $\Delta$ is a connected open subset of $\Omega$.   The base $\mathcal{F}$ consists of
such sets.
\begin{definition} \label{de:base_rigged_Teich}
 The base for the topology of $\ttildeop(\riem)$ is
 \[  \mathcal{F} = \{ F(V,S,\Delta) \,:\, S(\Omega,D)  \text{ compatible with }  V,\  \Delta \subseteq \Omega \text{ open and connected} \}.  \]
\end{definition}
It is an immediate consequence of the definition that the restriction of
any $F \in \mathcal{F}$ to a fibre is open in
in the following sense.
 \begin{lemma} \label{le:restriction_open}
  Let $\riem$ and $\mathcal{F}$ be as above.  For any $F \in \mathcal{F}$
  and representative $(\riem,f_1,\riem_1)$ of any point $[\riem,f_1,\riem_1] \in T(\riem)$
  \[  \{ \phi \in \Oqco(\riem_1) \,:\, [\riem,f_1,\riem_1,\phi] \in F \}\]
  is an open subset of $\Oqco(\riem_1)$.
 \end{lemma}
 \begin{proof}  This follows immediately from Lemma \ref{le:nu_preserves_base}.
 \end{proof}

It is necessary to show that $\mathcal{F}$ is indeed a base.  This will be accomplished
in several steps, together with the proof that the overlap maps of the charts are biholomorphisms.  The charts are given in the following definition.
\begin{definition}
\label{de:riggedcharts}
For each open set $F(V,S,\Delta) \subset \ttildeop(\riem)$ we define the chart
\begin{align} \label{eq:Gdefinition}
  G: \Delta \times U & \longrightarrow  F(V,S,\Delta) \\ \nonumber
  (\epsilon, \psi) & \longmapsto  [\riem,\nu^\epsilon \circ f_1, \riem_1^\epsilon, \nu^{\epsilon} \circ \zeta^{-1} \circ \psi].
 \end{align}
 where $U \subset (\Oqco)^n$ is related to $V$ as in Definition \ref{de:Oqco_base}
 and $S=S(\Omega,D)$ is compatible with $V$.
 \end{definition}

 \begin{lemma}
 The map $G$ is a bijection.
 \end{lemma}
 \begin{proof}
 If $G(\epsilon_1, \psi_1) = G(\epsilon_1, \psi_1)$ then $\epsilon_1 = \epsilon_2$ by Theorem \ref{th:Schiffer}.
 Because $2g-2+n>0$, Corollary \ref{co:rigged_equiv} implies $\psi_1 = \psi_2$. This proves injectivity. Surjectivity of $G$ follows from the definition of $F(V,S,\Delta)$.
\end{proof}

 It was shown in \cite{RS_fiber}, that if in the above map $\Oqco$ and $\Oqco(\riem)$ are replaced
 by $\Oqc$ and $\Oqc(\riem)$, and the corresponding changes are made to the sets $U_i$ and $V_i$,
 then these coordinates can be used to form an atlas on $\widetilde{T}^P(\riem)$.  We need to show the same result
 in the refined setting.
 \begin{remark} \label{re:index_suppression} Between here and the end of
 the proof of Lemma \ref{le:DavidsLemma}, we will suppress the subscripts
 on $n$-charts $(\zeta_i, E_i)$ and elements of $\Oqco(\riem_1)$  to avoid clutter.
 The subscripts which remain will distinguish $n$-charts on different Riemann surfaces.

 When clarification is necessary we will use the notation, for example $(\zeta_{i,j}$, $E_{i,j})$, where
 the first index labels the Riemann surface and the second labels the puncture.
 \end{remark}

We proceed as follows.  We first prove two lemmas, whose purpose is to
show that in a neighborhood of any point, the transition functions are defined
and holomorphic on some open set.  Once this is established, we show that $\mathcal{F}$
is a base, the topology is Hausdorff and separable, and the charts form a holomorphic
atlas.

Some notation is necessary regarding the transition functions.
Fix two points $[\riem, f_1, \riem_1]$
and $[\riem, f_2, \riem_2]$ in $T(\riem)$. Let $G_1$ and $G_2$ be two corresponding parametrizations as in (\ref{eq:Gdefinition}) above, defined on $\Delta_1 \times U_1$ and $\Delta_2 \times U_2$ respectively and using the two Schiffer families $S_1(\Delta_1, D_1)$ and $S_2(\Delta_2, D_2)$. We assume that the intersection $G_1(\Delta_1 \times U_1) \cap G_1(\Delta_2 \times U_2)$  is non-empty. From the definitions of $\ttildeop(\riem)$ and $\mathcal{S}$ it follows that $\mathcal{S}(\Delta_1) \cap \mathcal{S}(\Delta_2)$ is also non-empty.
We follow the notation and setup of Lemma \ref{le:marked_families_lemma} and the paragraph immediately preceding it, with $\Delta'_i = \mathcal{S}_i^{-1}(N)$ replacing $\Omega'_i$, where $N$ is any connected component of $\mathcal{S}(\Delta_1) \cap \mathcal{S}(\Delta_2)$.

Recall that in $\widetilde{T}_0^P(\riem)$, $[\riem, g_1, \riem_1, \phi_1] = [\riem, g_2, \riem_2, \phi_2]$ if and only if $[\riem, g_1, \riem_1] = [\riem, g_2, \riem_2]$ via the biholomorphism $\sigma : \riem_1 \to \riem_2$ and $ \sigma \circ \phi_1 = \phi_2$. Lemma \ref{le:marked_families_lemma} now implies that
$G_1(\epsilon, \psi) = G_2(\epsilon', \psi')$ if and only if $\epsilon' = \alpha(\epsilon)$ and
$$ \nu_2^{\alpha(\epsilon)} \circ \zeta_2^{-1} \circ \psi' = \sigma_{\epsilon} \circ  \nu_1^\epsilon \circ \zeta_1 ^{-1} \circ \psi.$$

Let
\begin{equation}
\label{eq:littleh_definition}
\mathcal{H}(\epsilon, z) = \mathcal{H}_{\epsilon}(z) = \left(\zeta_2 \circ (\nu^{\alpha(\epsilon)}_2)^{-1} \circ \sigma_{\epsilon} \circ  \nu^\epsilon_1 \circ \zeta_1^{-1} \right) (z)
\end{equation}
which is a function of two complex variables.
We also define
$$ \mathcal{G}(\epsilon,z) =  (\alpha(\epsilon),\mathcal{H}(\epsilon,z)).
$$
Note that this is shorthand for a collection of maps $\mathcal{H}^j(\epsilon,z)$
and $\mathcal{G}^j(\epsilon,z)$, $j=1,\ldots,n$, where $j$ indexes
 the punctures (cf. Remark \ref{re:index_suppression}).
Define further
\begin{align} \label{eq:bigH_definition}
 H:\Omega'_1 \times (\Oqco)^n & \longrightarrow ({\Oqco})^n \\
 (\epsilon, \psi) & \longmapsto \mathcal{H}_\epsilon\circ \psi. \nonumber
\end{align}
The overlap maps can then be written
\begin{equation} \label{eq:G_overlap_definition}
 (G_2^{-1} \circ G_1)(\epsilon, \psi) = (\alpha(\epsilon), \mathcal{H}_{\epsilon} \circ \psi )=(\alpha(\epsilon),H(\epsilon,\psi)).
\end{equation}

\begin{lemma} \label{le:DavidsLemma} Let $[\riem,f_1,\riem_1]$ and $[\riem,f_2,\riem_2] \in \widetilde{T}_0^P(\riem)$ for a punctured Riemann surface $\riem$.  For $i=1,2$ let $\mathcal{V}_i$ be the base for the topology on $\Oqco(\riem_i)$ as in Definition \ref{de:Oqco_base}.  Again
 for $i=1,2$ let $(\zeta_i,E_i)$ be $n$-charts on $\riem_i$, let $V_i \in \mathcal{V}_i$ be compatible with the $n$-charts $(\zeta_i,E_i)$, and let $S_i(\Omega_i,D_i)$ be Schiffer variations compatible with $V_i$.
 Finally, for open connected sets $\Delta_i \subseteq \Omega_i$ consider the sets $F(V_i,S_1,\Delta_i)$ which we assume have non-empty intersection.

Choose any $e_1 \in \Delta_1$ and $\phi_1 \in V_1$ such that $[\riem,\nu_1^{e_1} \circ f_1,\riem_1^{e_1},\nu_1^{e_1} \circ \phi_1] \in F(V_1,S_1,\Delta_1) \cap F(V_2,S_2,\Delta_2)$. Then, there exists a $\Delta \subset
 \mathcal{S}_1^{-1}( \mathcal{S}_1(\Delta_1) \cap \mathcal{S}_2(\Delta_2))$ containing $e_1$, and an
 open set $E'_1 \subseteq \zeta_1(E_1)$
 containing $\overline{\zeta_1 \circ \phi_1(\mathbb{D})}$, such that $\mathcal{H}$
 is holomorphic in $\epsilon$ and $z$ on $\Delta \times E'_1$
 and $\mathcal{G}(\epsilon,z)=(\alpha(\epsilon),\mathcal{H}(\epsilon,z))$
 is a biholomorphism onto $\mathcal{G}(\Delta \times E'_1 )$.
\end{lemma}
\begin{proof} Let $N$ be the connected component of $\mathcal{S}_1(\Delta_1) \cap \mathcal{S}_2(\Delta_2)$ that contains $\mathcal{S}_1(e_1)$.
For $i = 1,2$, let  $ \Delta'_i= \mathcal{S}_i^{-1}(N)$,

$$
E_i^{\epsilon_i} = \nu_i^{\epsilon_i}(E_i)
$$
and
$$
A_i^{\epsilon_i} = (\nu_i^{\epsilon_i} \circ \phi_{1}) (\mathbb{D}).
$$
Note that $\overline{A_i^{\epsilon_i}} \subset E_i^{\epsilon_i}$. By construction $\Delta_1'$ contains $e_1$.

Let
$$
\widetilde{E}_i = \{(\epsilon_i, z) : \epsilon_i \in \Delta'_i, z \in  E_i^{\epsilon_i} \}
$$
and
$$
\widetilde{A}_i = \{(\epsilon_i, z) : \epsilon_i \in \Delta'_i, z \in  A_i^{\epsilon_i} \}.
$$
Both of these sets are open by definition of $S(\Omega_i, D_i)$.

Now by Lemma \ref{le:marked_families_lemma} there is a biholomorphism $\beta: S(\Delta'_1, D_1) \to S(\Delta'_2, D_2)$ and  moreover, $\beta(\widetilde{A}_1) = \widetilde{A}_2$. The last assertion follows
from the definition of equivalence in the rigged \teich space.

Let
$$
\widetilde{C} = \beta^{-1}(\widetilde{E}_2) \cap \widetilde{E}_1
$$
and note that $\overline{\widetilde{A}}_1 \subset \widetilde{C}$.

Since $\widetilde{C}$ is open, so is
$$
J = \tilde{\zeta}_1(\widetilde{C}) \subset \Delta'_1 \times \zeta_1(E_1),
$$
where $\tilde{\zeta}_1$ is defined in (\ref{family_chart_map}).
Let $J^{\epsilon} = \{ z : (\epsilon, z) \in J \}$. Then
$$\overline{\psi_1(\disk)} \subset J^{\epsilon} \subset \zeta_1(E_1)$$
for all $\epsilon$, where $\psi_1 = \zeta_1 \circ \phi_1$. By the definition of $\widetilde{C}$, $\mathcal{H}$ is defined on $J^{\epsilon}$.

We claim that there are connected open sets $\Delta$ and $E'$ such that the
closure of $\Delta \times E'$ is contained in $J$, $e_1 \in \Delta$
and $\overline{\psi_1(\mathbb{D})} \subset E'$.
Since $J$ is open and
$\{e_1\} \times \overline{\psi_1(\mathbb{D})}$ is compact the existence of such sets $\Delta$ and $E'$ follow from a standard topological argument.

Since $\mathcal{H}$, and therefore $\mathcal{G}$ are defined on $J$ they are defined on $\Delta \times E'$.
We will prove that $\mathcal{G}$ is biholomorphic by showing that it is equal to $\beta$ expressed in terms of local coordinates.
Using the coordinates defined in (\ref{family_chart_map}), noting that on $E'$, $\nu^{\epsilon} = \iota^{\epsilon}$, and applying Lemma \ref{le:marked_families_lemma}, we have for $(\epsilon, z) \in \Delta \times E'$ that
\begin{align*}
(\tilde{\zeta}_2 \circ \beta \circ \tilde{\zeta}_1^{-1})(\epsilon, z) & = \left(\alpha(\epsilon), (\zeta_2 \circ (\nu_2^{\alpha(\epsilon)})^{-1} \circ \sigma_{\epsilon} \circ \nu_1^{\epsilon} \circ \zeta_1^{-1})(z) \right) \\
& = (\alpha(\epsilon), \mathcal{H}(\epsilon, z)) \\
& = \mathcal{G}(\epsilon, z).
\end{align*}
Since $\beta$ is a biholomorphism we see that on the domain $\Delta \times E'$, $\mathcal{G}$ is a biholomorphism and $\mathcal{H}$ is holomorphic.
\end{proof}
 \begin{theorem}  \label{th:the_lemma_theorem}
  With notation as in Lemma \ref{le:DavidsLemma}, assume that
   $p=[\riem,\nu_1^{e_1} \circ f_1,\riem_1^{e_1},\nu_1^{e_1} \circ \phi_1]$ is
   an arbitrary point in $F(V_1,S_1,\Delta_1) \cap F(V_2,S_2,\Delta_2)$.
   There exists a $V_1' \in \mathcal{V}_1$ and a $\Delta_1'$ such that
   \begin{enumerate}
    \item $p \in F(V_1' ,S_1, \Delta'_1) \subseteq F(V_1,S_1,\Delta_1) \cap F(V_2,S_2,\Delta_2)$
    \item For all $\psi \in U_1'$ (where $U_1'$ is associated to $V_1'$ as in
    Definition \ref{de:Oqco_base}), $\overline{\psi(\mathbb{D})}$ is contained in an
    open set $E'$ satisfying the consequences of Lemma \ref{le:DavidsLemma}
    \item $G_2^{-1} \circ G_1$ is holomorphic on $\Delta_1' \times U_1'$.
   \end{enumerate}
 \end{theorem}
 \begin{proof}
  By Lemma \ref{le:DavidsLemma}, there is an open set $\Delta_1'' \times E_1'$
  such that $\overline{\zeta_1 \circ \phi_1(\mathbb{D})} \subset E_1' $, $e_1 \in \Delta_1''$,  $\mathcal{H}$ is holomorphic on $\Delta_1'' \times E_1'$ and $\mathcal{G}$ is biholomorphic
  on $\Delta_1'' \times E_1'$.  This immediately implies that there is an open set
  $\Delta_2' \times E_2' \subset \mathcal{G}(\Delta_1'' \times E_1')$
  such that $\alpha(e_1) \in \Delta_2'$ and for $\psi_2=H(e_1,\zeta_1 \circ \phi_1)$, $\overline{\psi_2(\mathbb{D})} \subseteq E_2'$.
  Now let
  $
  W_2 =\{\psi \in \Oqco \,:\, \overline{\psi(\mathbb{D})} \subseteq E_2' \}.
  $
  By Theorem \ref{th:into_U_is_open} and Remark \ref{re:chart_simplification}, $W_2 \cap U_2$ is open in $\Oqco$.  Note that
  $H(e_1,\zeta_1 \circ \phi_1) \in W_2 \cap U_2$.

  Choose a compact set $K \subset E_1'$ which contains $\overline{\zeta_1 \circ \phi_1(\mathbb{D})}$ in its interior $K_{int}$.  If we let $W_1 = \{ \psi \in \Oqco \,:\,
  \overline{\psi(\mathbb{D})} \subseteq K_{int} \},$ then $W_1$ is open by Theorem \ref{th:into_U_is_open}.  We claim that $H$ is holomorphic
  on $\Delta_1' \times W_1$.  By Hartogs' theorem (see \cite{Mujica} for
  a version in a suitably general setting), it is enough to check
 holomorphicity separately in $\epsilon$ and $\psi$.  By Lemma \ref{le:CompHoloOqc_manyvar}, $H$ is holomorphic in  $\epsilon$ for fixed $\psi$.  On the other hand, by Theorem \ref{th:left_comp_holo}, $H$ is holomorphic in $\psi$ for fixed $\epsilon$ by our
careful choice of $W_1$.

  In particular, $H$ is continuous and therefore $H^{-1}(W_2 \cap U_2) \cap (\Delta_1'' \times (W_1 \cap U_1))$ is open
  and contains $(e_1, \zeta_1 \circ \phi_1)$, hence we may choose an open subset $\Delta_1'
  \times U_1'$ containing $(e_1, \zeta_1 \circ \phi_1)$. Let $V_1'$
  be the element of $\mathcal{V}_1$ associated to $U_1'$.  Clearly $U_1' \subseteq  U_1$, and $H(\Delta_1' \times U_1') \subseteq U_2$ by construction; thus $F(V_1' ,S_1, \Delta'_1 ) \subseteq F(V_1,S_1,\Delta_1) \cap F(V_2,S_2,\Delta_2)$ so
  the first condition is satisfied.  By construction, (2) is  also satisfied.  Since $U_1' \subseteq W_1$, $H$ is holomorphic on $\Delta_1' \times U_1'$ and the fact that $\alpha$
 is holomorphic on $\Delta'$ yields that $G_2^{-1} \circ G_1$ is holomorphic on $\Delta_1' \times U_1'$. This concludes the proof.
 \end{proof}

 \begin{theorem} \label{th:tilde_top_base} The set
  $\mathcal{F}$ is a base for a Hausdorff, separable topology on $\ttildeop(\riem)$.  Furthermore, with the atlas of charts given
  by (\ref{eq:Gdefinition}), $\ttildeop(\riem)$ is a Hilbert manifold.
 \end{theorem}
 \begin{proof}
   It follows directly from part (1) of Theorem \ref{th:the_lemma_theorem} that
  $\mathcal{F}$ is a base for a topology on $\tilde{T}_0(\riem)$.  From part (3),
  we have that the inverses of the maps (\ref{eq:Gdefinition}) form an atlas
  with holomorphic transition functions.  Thus it remains only to show that this
  topology is Hausdorff and separable.  We first show that it is Hausdorff.

  For $i=1,2$, let $p_i=[\riem,\nu_i^{e_i} \circ f_i,\riem_i^{e_i},\nu_i^{e_i} \circ \phi_i]$
  be distinct points, in sets $F(V_i,S_i,\Delta_i)$.  If $F(V_i,S_i,\Delta_i)$ are disjoint, we are done.
  If not, by Lemma \ref{le:marked_families_lemma} setting $\Delta_i'$ to be the connected component of  $\mathcal{S}_i^{-1} \left( \mathcal{S}_1(\Delta_1)
  \cap \mathcal{S}_2(\Delta_2) \right)$ containing $e_i$,  there is a biholomorphism $\alpha:\Delta_1' \rightarrow \Delta_2'$
 such that $[\riem,\nu_2^{\alpha(\epsilon)} \circ f_2, \riem_2^{\alpha(\epsilon)}] = [\riem,\nu_1^\epsilon \circ f_1, \riem_1^\epsilon]$
  for all $\epsilon \in \Delta_1'$.

  There are two cases to consider.  If $[\riem,\nu_2^{\alpha(e_1)} \circ f_2, \riem_2^{\alpha(e_1)}] \neq [\riem,\nu_2^{e_2} \circ f_2, \riem_2^{e_2}]$, then one can find $\Omega_1 \subset \Delta_1$ and $\Omega_2 \subset \Delta_2$ such that
  $\mathcal{S}_1(\Delta_1)$ and $\mathcal{S}_2(\Delta_2)$ are disjoint and $F(V_i,S_i,\Omega_i)$ still contains
  $[\riem,\nu_1^{e_i} \circ f_1,\riem_1^{e_i},\nu_1^{e_i} \circ \phi_i]$ for $i=1,2$.  But then
  $F(V_i,S_i,\Omega_i)$ are disjoint, which takes care of the first case.

  If on the other hand $[\riem,\nu_2^{\alpha(e_1)} \circ f_2, \riem_2^{\alpha(e_1)}] = [\riem,\nu_2^{e_2} \circ f_2, \riem_2^{e_2}]$,
  then by Theorem \ref{th:the_lemma_theorem} there are sets $F(V'_1,S_1,\Omega_1)$ and $F(V'_2,S_1,\Omega_2)$
  in $F(V_1,S_1,\Delta_1) \cap F(V_2,S_2,\Delta_2)$
  containing $p_1$ and
  $p_2$ respectively.  Thus we may write
  \[  p_1 = [\riem,\nu_1^{e_1} \circ f_1,\riem_1^{e_1},\nu_1^{e_1} \circ \psi_1] \quad \text{and} \quad  p_2=[\riem,\nu_1^{e_1} \circ f_1,\riem_1^{e_1},\nu_1^{e_1} \circ \psi_2].  \]

  For $i=1,2$, let $U_i'$ be the subsets of $(\Oqco)^n$ associated
  with $V_i'$ as in Definition \ref{de:Oqco_base}.  Since $\Oqco$ is an open subset of a
  Hilbert space, it is Hausdorff, so there are open sets $W_i$ in $U_i'$
  containing $p_i$
  for $i=1,2$ and such that $W_1 \cap W_2$ is empty.  In that case if $V''_i$ are the
  elements of $\mathcal{V}$ associated to $W_i$, then $V''_1 \cap V''_2$ is empty.
  This in turn implies that $F(V''_1,S_1,\Omega_1) \cap F(V''_2,S_1,\Omega_2)$ is empty which
  proves the claim in the second case.

  We now prove that $\ttildeop(\riem)$ is separable. Since $T(\riem)$ is a finite
  dimensional complex manifold it is, in particular, separable.  Choose a countable
  dense subset $\mathfrak{A}$ of $T(\riem)$.  For each $p=[\riem,f_1,\riem_1] \in \mathfrak{A}$, choose a specific representative $(\riem,f_1,\riem_1)$.  The space $\Oqco(\riem_2)$ is second countable and, in particular, it has a countable dense subset
  $\mathfrak{B}_p(\riem_1)$.  Now if $(\riem,f_2,\riem_2)$ is any other representative,
  there exists a unique biholomorphism $\sigma:\riem_1 \rightarrow \riem_2$ (if $\sigma_1$ is another such biholomorphism,
   since by hypothesis $\sigma_1^{-1} \circ \sigma$ is homotopic to the identity and $2g-2+n>0$, it follows from Theorem \ref{th:homotopic_zero} that $\sigma_1^{-1} \circ \sigma$ is the identity).
  We set
  \[  \mathfrak{B}_p(\riem_2)= \left\{ (\sigma \circ \phi_1,\ldots,\sigma \circ \phi_n) \,:\, (\phi_1,\ldots,\phi_n) \in \mathfrak{B}_p(\riem_1)  \right\}.   \]
  This is easily seen to be itself a countable dense set in $\Oqco(\riem_2)$ and it is not hard to see that
  \[  \Upsilon =
      \{ [\riem,f_1,\riem_1,\psi_1] \,:\, [\riem,f_1,\riem_1] \in \mathfrak{A}, \ \psi_1 \in
      \mathfrak{B}_p(\riem_1) \}  \]
  is well-defined.  We will show that it is dense.
  Note that for any fixed $[\riem,f_1,\riem_1]$, the set of
  $[\riem,f_1,\riem_1,\psi_1] \in \Upsilon$ is entirely determined by any particular representative $(\riem,f_1,\riem_1)$, and so this is a countable set.

  Let $F(V,S, \Delta) \in \mathcal{F}$.  Since $\mathfrak{A}$ is dense, there is some $[\riem,f_2,\riem_2] \in \mathfrak{A} \cap S(\Delta)$.  For a specific
  representative $(\riem,f_2,\riem_2)$ there is a $\psi_2 \in \Oqco(\riem_2)$
  such that $[\riem,f_2,\riem_2,\psi_2] \in F(V,S,\Delta)$.  By Lemma \ref{le:restriction_open} the set of points in $F$ over $[\riem,f_2,\riem_2]$
  is open. Thus since $\mathfrak{B}_p(\riem_2)$ is dense in $\Oqco(\riem_2)$ there is a $\psi_3 \in \mathfrak{B}_p(\riem_2)$ such that $[\riem,f_2,\riem_2,\psi_3] \in F$.
  By definition $[\riem,f_2,\riem_2,\psi_3] \in \Upsilon$, which completes the proof.
 \end{proof}
 \begin{remark} \label{re:tildeTP_second_countable} It can be shown that $\ttildeop(\riem)$ is second countable.  The proof involves somewhat tedious notational difficulties, so we only give a sketch of the proof.  No results in
 this paper depend on second countability of $\ttildeop(\riem)$.

  Fix a countable basis $\mathfrak{O}$ for
  $\Oqco$.
  For any $[\riem,f_1,\riem_1] \in
  \mathfrak{A}$, choose a representative $(\riem,f_1,\riem_1)$, and fix
  the following objects.  Let $\mathfrak{C}(\riem_1)$ be a countable collection
  of $n$-charts on $\riem_1$ constructed as in the proof of Theorem \ref{th:Oqco_base}. Let $\mathcal{V}_c(\riem_1)$ be the countable dense subset of $\mathcal{V}(\riem_1)$ corresponding to $\mathfrak{O}$ and $\mathfrak{C}(\riem_1)$ as in the proof
  of Theorem \ref{th:Oqco_base}.  Finally, fix a countable base $\mathfrak{B}(\riem_1)$
  of open sets in $\riem_1$.

  Now if $(\riem,f_2,\riem_2)$ is any other representative, there is a unique biholomorphism $\sigma:\riem_1 \rightarrow \riem_2$ as in the proof
  of Theorem \ref{th:tilde_top_base}.
   Transfer each of the preceding objects to $\riem_2$ by composition with
   $\sigma$ in the appropriate way; for example, $\mathfrak{C}(\riem_2)$ is the set of
   $n$-charts $(\zeta_1 \circ \sigma^{-1},\sigma(E_1), \ldots, \zeta_n \circ \sigma^{-1},\sigma(E_n))$ and so on.  Finally fix a countable base $\mathfrak{D}$
   of $\mathbb{C}^n$ (for example, the set of discs of rational radius centered at rational points).

   We now define the subset $\mathcal{F}_c$ of $\mathcal{F}$ to
   be the set of $F(V,S,\Delta) \in \mathcal{F}$ such that
   \begin{enumerate}
    \item the variation $S(\Omega)$ is based at a point $[\riem,f_1,\riem_1] \in
     \mathfrak{A}$
    \item $S(\Omega)$ is compatible with some fixed $n$-chart in $\mathfrak{C}(\riem_1)$
    \item $\Omega$ and $\Delta$ are both in $\mathfrak{D} \times \cdots \times \mathfrak{D}$
    \item $V \in \mathcal{V}(\riem_1)$.
   \end{enumerate}
   The set $\mathcal{F}_c$ is countable by construction, and does not depend on the choice
   of representative.  It can be shown with some work that $\mathcal{F}_c$ is a base compatible with $\mathcal{F}$.
 \end{remark}

\end{subsection}

\begin{subsection}{Compatibility with the non-refined rigged \teich space}

In \cite{RS05} the following rigged \teich space was defined.
\begin{definition}
Let $\ttildep(\riem)$ be defined by replacing $\Oqco(\riem_1)$ with $\Oqc(\riem_1)$ in Definition \ref{de:rigged_Teich0}.
\end{definition}
It was shown in \cite{RSnonoverlapping} that $\ttildep(\riem)$ is a complex Banach manifold with charts as in Definition \ref{de:riggedcharts} with $U \subset (\Oqc)^n$, and $\Oqc$ replacing $\Oqco$ in all the preceding definitions and constructions. Furthermore, the complex structure on $\Oqc$ is given by the embedding $\chi$ defined by (\ref{eq_chidefinition}). We use the same notation for the charts and constructions on $\ttildep(\riem)$ as for $\ttildeop(\riem)$ without further comment.

 The complex structures on $\ttildeop(\riem)$ and
 $\ttildep(\riem)$  are compatible in the following sense.
 \begin{theorem} \label{th:top in tp}
  The inclusion map $I_T: \ttildeop(\riem) \rightarrow \ttildep(\riem)$ is
  holomorphic.
 \end{theorem}
 \begin{proof}
  Choose any point $[\riem,f,\riem_*,\phi] \in
  \ttildeop(\riem)$. There is a parametrization
  $G: \Omega \times U \rightarrow \ttildep(\riem)$ onto a
  neighborhood of this point (see Definition \ref{de:riggedcharts}).
  We choose $U$ small enough that $\nu^\epsilon$ is
  holomorphic on $\overline{\phi(\mathbb{D})}$ for all $\phi \in U$.

  Let $W = \chi^n(U)$ where $\chi^n:\Oqc \times \cdots \times \Oqc
  \rightarrow \bigoplus^n(\aoneinfinity \oplus \mathbb{C})$ is defined by
  $$
  \chi^n(\phi_1,\ldots,\phi_n) =
  (\chi(\phi_1),\ldots,\chi(\phi_n)).
  $$
  Define $F:\Omega \times W \rightarrow \ttildep(\riem)$ by
  $$
  F = G \circ
  (\text{id}, (\chi^n)^{-1} )
  $$
  where $\text{id}$ is the identity map on
  $\Omega$. These are coordinates on $\ttildep(\riem)$.

  Let $W_0 = W \cap \Oqco = \iota^{-1}(W)$ (recall that $\iota$ is the
  inclusion map of $\Oqco$ in $\Oqc$).  The set $W_0$ is open by Theorem
  \ref{th:Oqco_open_in_Oqc}.  We further have that $F(\Omega
  \times W_0) = \ttildeop \cap W$.  To see this note that $F(\Omega
  \times W_0) = G (\Omega \times (\chi^n)^{-1}(W_0))$.  By
  definition $\nu^\epsilon \circ \phi \in \Oqco(\riem)$ if and
  only if for a parameter $\eta:A \rightarrow \mathbb{C}$ defined on
  an open neighborhood $A$ of
  $\overline{\nu^\epsilon(\phi(\mathbb{D}))}$ it holds that $\eta
  \circ \nu^\epsilon \circ \phi \in \Oqco$. This holds if and only
  if $\phi \in \Oqco$ since $\nu^\epsilon$ is holomorphic on a
  neighborhood of $\overline{\phi(\mathbb{D})}$.

  It follows from Theorem \ref{th:into_U_is_open} that $F^{-1} \circ
  I_T \circ F$ is holomorphic.  Since $F$ are local coordinates,
  $I_T$ is holomorphic on the image of $F$.  Since coordinates of
  the form $F$ cover $\ttildep(\riem)$, this proves the theorem.
 \end{proof}
 Note that this does not imply that $\ttildeop(\riem)$ is a
 complex submanifold of $\ttildep(\riem)$.
\end{subsection}
\end{section}
\begin{section}{A refined Teichm\"uller space of bordered surfaces} \label{se:refined_Teich_space}
 We are at last in a position to define the refined Teichm\"uller space of a bordered surface and demonstrate that it
 has a natural complex Hilbert manifold structure.  In Section \ref{se:definition_refined} we define the refined Teichm\"uller space $T_0(\riem^B)$ of a bordered
 surface $\riem^B$, and define some ``modular groups'' which act on it.  In Section \ref{se:sewing_on_caps} we show how to obtain a punctured surface by sewing ``caps'' onto the bordered surface using the riggings.  It is also demonstrated that sewing on
 caps takes the refined Teichm\"uller space into the refined rigged Teichm\"uller space
 $\ttildeop(\riem)$.  In Section \ref{se:complex_structure_refined_Teich} we prove that the refined Teichm\"uller space of bordered surfaces is a Hilbert manifold.  We do this by showing that the refined rigged Teichm\"uller space $\ttildeop(\riem)$ is
 a quotient of $T_0(\riem^B)$ by a properly discontinuous, fixed point free
 group of local homeomorphisms, and passing the charts on $\ttildeop(\riem)$ upwards.  Finally, in Section \ref{se:Friedan_Shenker_refined} we show that the
 rigged moduli space of Friedan and Shenker is a Hilbert manifold.  This follows
 from the fact that the rigged moduli space is a quotient of $T_0(\riem^B)$ by a
 properly discontinuous fixed-point free group of biholomorphisms.

\begin{subsection}{Definition of the refined Teichm\"uller space and modular groups} \label{se:definition_refined}

The reader is referred to Section \ref{se:refined_quasisymmetries} for some of the notation and definitions used below.

We now define the refined \teich space of a bordered Riemann surface which is obtained by replacing the quasiconformal marking maps in the usual \teich space (see Definition \ref{de:Teichspace}) with refined quasiconformal maps.
 \begin{definition}  \label{de:refined_Teich_space}
  Fix a bordered Riemann surface $\riem^B$ of type $(g,n)$.  Let
  \[  T_0(\riem^B) = \{ (\riem^B,f,\riem^B_1) \} / \sim \]
  where $\riem^B_1$ is a bordered Riemann surface of the same type, $f \in \qco(\riem^B,\riem^B_1)$, and two triples
  $(\riem^B,f_i,\riem^B_i)$, $i=1,2$ are equivalent if there is a biholomorphism $\sigma:\riem^B_1 \rightarrow \riem^B_2$ such that
  $f_2^{-1} \circ \sigma \circ f_1$ is homotopic to the identity rel boundary.

The space $T_0(\riem^B)$ is called the \textit{refined \teich space} and its elements are denoted by equivalence classes of the form $[\riem^B, f_1, \riem^B_1]$.
 \end{definition}

 An important ingredient in the construction of the complex Hilbert manifold structure is a kind of modular group (or mapping class group). To distinguish between the different possible boundary condition we use some slightly non-standard notation following \cite{RS05}; we recall the definitions here.

 Let $\riem^B$ be a bordered Riemann surface and $\operatorname{QCI}(\riem^B)$ denote the set of quasiconformal maps from $\riem^B$
 onto $\riem^B$ which are the identity on the boundary.  This is a group which acts on the marking maps  by right composition. Let $\operatorname{QCI}_n(\riem^B)$ denote the subset of $\operatorname{QCI}(\riem^B)$ which are homotopic to
 the identity rel boundary (the subscript $n$ stands for ``null-homotopic'').
 \begin{definition} \label{de:pmodi}
  Let $\pmodi(\riem^B) = \operatorname{QCI}(\riem^B)/\sim$ where two elements $f$ and $g$ of $\operatorname{QCI}(\riem^B)$ are equivalent
  ($f \sim g$) if and only if $f \circ g^{-1} \in \operatorname{QCI}_n(\riem) $.
 \end{definition}
 The ``P'' stands for ``pure'', which means that the mappings preserve the ordering of the boundary components, and ``I''
 stands for ``identity''.

 There is a natural action of $\pmodi(\riem^B)$ on $T(\riem^B)$ by right composition, namely
 \begin{equation}
 \label{eq:MCG_action}
  [\rho] [\riem^B,f,\riem^B_1] = [\riem^B,f \circ \rho,\riem^B_1].
 \end{equation}
 This is independent of the choice of representative $\rho \in \operatorname{QCI}(\riem^B)$ of $[\rho] \in \pmodi(\riem^B).$
It is a standard fact that $\pmodi(\riem^B)$ is finitely generated by Dehn twists. Using these twists we can define two natural subgroups of $\pmodi(\riem^B)$  (see \cite{RS05} for details).
 \begin{definition} \label{de:dbdi}
  Let $\riem^B$ be a bordered Riemann surface.
  Let $\db(\riem^B)$ be the subgroup of $\pmodi(\riem^B)$ generated by Dehn twists around simple closed curves $\riem$ which are homotopic to a boundary curve.
  Let $\di(\riem^B)$ be the subgroup of $\pmodi(\riem^B)$ generated by Dehn twists around simple closed curves in $\riem^B$ which are neither homotopic to
  a boundary curve nor null-homotopic.
 \end{definition}
 Here ``B'' stands for ``boundary'' and ``I'' stands for ``internal''.

The next Lemma implies that we can consider $\pmodi(\riem^B)$ and $\db(\riem^B)$ as acting on $T_0(\riem^B)$.
\begin{lemma}
\label{le:MCG_reduced_action}
Every element of $\operatorname{QCI}(\riem^B)$ is in $\qco(\riem^B,\riem^B)$.  Thus, the group action of $\pmodi(\riem^B)$ on $T(\riem^B)$ preserves
 $T_0(\riem^B)$.
\end{lemma}
\begin{proof}
The first statement follows from Definition \ref{de:qco}, and Definition \ref{de:refined_qs_surfaces} with $H_1 = H_2$. The second statement follows from Proposition \ref{pr:composition_preserves_qco_surfaces}.
\end{proof}

\end{subsection}
\begin{subsection}{Sewing on caps} \label{se:sewing_on_caps}
 Given a bordered Riemann surface $\riem^B$ together with quasisymmetric parametrizations of its boundaries by the circle, one can sew on copies
 of the punctured disc to obtain a punctured Riemann surface $\riem$.  The collection of parametrizations extend to an element of
 $\Oqc(\riem)$.  In \cite{RS05}, two of the authors showed that this operation can be used to exhibit a natural correspondence between the rigged Teichm\"uller
 space $\ttildep(\riem)$ and the Teichm\"uller space $T(\riem^B)$, and showed in \cite{RS_fiber} that this results in a natural
 fibre structure on $T(\riem^B)$. We will be using this fibre structure as the principle framework for constructing the Hilbert manifold structure
 on $T_0(\riem^B)$.  It is thus necessary to describe sewing on caps here, in the setting of refined quasisymmetries.

 \begin{definition} \label{de:riggings}
  Let $\riem^B$ be a bordered Riemann surface with boundary curves $C_i$, $i=1,\ldots,n$.
  The {\it riggings} of $\riem^B$ is the collection $\rig(\riem^B)$ of $n$-tuples $\psi=(\psi_1,\ldots,\psi_n)$ such
  that $\psi_i \in \qs(S^1,C_i)$.  The refined riggings is the collection $\rigo(\riem^B)$ of $n$-tuples $\psi=(\psi_1,\ldots,\psi_n)$ such
  that $\psi_i \in \qso(S^1,C_i)$
 \end{definition}

 Let $\riem^B$ be a fixed bordered Riemann surface of type $(g,n)$ say, and $\psi \in \rig(\riem^B)$.  Let $\mathbb{D}_0$
 denote the punctured unit disc $\mathbb{D} \backslash \{0\}$. We obtain a new topological space
 \begin{equation} \label{sew}
 \riem = \overline{\riem^B} \sqcup \overline{\mathbb{D}}_0 \sqcup \cdots \sqcup \overline{\mathbb{D}}_0/\sim .
 \end{equation}
 Here we treat the $n$ copies of $\mathbb{D}_0$ as distinct
 and ordered, and two points $p$ and $q$ are equivalent($p \sim q$) if $p$ is in the boundary of the $i$th disc, $q$ is in the $i$th boundary $C_i$, and
 $q=\psi_i(p)$.  By \cite[Theorems 3.2, 3.3]{RS05} this topological space has a unique complex structure which is compatible with the complex structures
 on $\riem^B$ and each copy of $\mathbb{D}_0$.  We will call the image of a boundary curve in $\riem$ under inclusion
 (which is also the image of $\partial \mathbb{D}$ under inclusion) a \textit{seam}.  We will call the copy of each
 disc in $\riem$ a \textit{cap}. Finally, we will denote equation (\ref{sew}) by
 $$
 \riem=\riem^B \#_\psi \mathbb{D}_0^n
 $$
 to emphasize the underlying element of $\rig(\riem^B)$ used to sew.

For each $i=1,\ldots,n$ the map $\psi_i$ can be extended to a map $\tilde{\psi}_i : \overline{\mathbb{D}}_0 \to \riem$ defined by
\begin{equation} \label{eq:rigging_ext}
\tilde{\psi}_i(z) = \begin{cases}
\psi(z), & \text{for } z \in \partial \mathbb{D} \\
z, & \text{for } z \in \mathbb{D} .
\end{cases}
\end{equation}
Note that $\tilde{\psi}_i$ is well defined and continuous because the map $\psi_i$ is used to identify $\partial{\mathbb{D}}$ with $C_i$. Moreover, $\tilde{\psi}$ is holomorphic on $\mathbb{D}_0$. It is important to keep in mind that if the seam in $\riem$ is viewed as $\partial \mathbb{D}$ then in fact $\tilde{\psi}_i$ is also the identity on $\partial \mathbb{D}$.

 \begin{remark} \label{re:complex_structure_sewn_surface}
  The complex structure on the sewn surface is easily described in terms of conformal welding.  Choose a seam $C_i$ and let $H$ be a
  collared chart (see Definition \ref{de:collar_nbhd}) with respect to $C_i$ with domain $A$ say.  We have that $H \circ \psi_i$ is in $\qs(S^1)$.   Let $F:\mathbb{D} \rightarrow \mathbb{C}$
  and $G:\mathbb{D}^* \rightarrow \overline{\mathbb{C}}$ be the unique holomorphic welding maps such that $G^{-1} \circ F = H \circ \psi_i$
  when restricted to $S^1$, $F(0)=0$, $G(\infty)=\infty$ and $G'(\infty)=1$.  Note that $F$ and $G$ have quasiconformal extensions to
  $\mathbb{C}$ and $\overline{\mathbb{C}}$ respectively.

  Let $\zeta_i$ be the continuous map on $A \cup \overline{\tilde{\psi}_i(\mathbb{D})}$ defined by
  \begin{equation} \label{eq:chart_containing_cap}
    \zeta_i =
   \begin{cases}
     F \circ \tilde{\psi}_i^{-1} & \text{on } \tilde{\psi}(\mathbb{D}) \\
      G \circ H & \text{on } A.
   \end{cases}
  \end{equation}
  It is easily checked that there is such a continuous extension.  Since $\zeta_i$ is $0$-quasiconformal on $\tilde{\psi}_i(\mathbb{D})$
  and $A$, by removability of quasicircles \cite[V.3]{Lehto-Virtanen} $\zeta_i$ is $0$-quasiconformal   (that is, holomorphic and one-to-one), on $A \cup \overline{\tilde{\psi}_i(\mathbb{D})}$.  Thus $\zeta$ is a local coordinate on $\riem$ containing the closure of the cap.
 \end{remark}

 The crucial fact about the extension $\tilde{\psi}=(\tilde{\psi}_1,\ldots,
 \tilde{\psi}_n)$ is that it is in $\Oqco(\riem)$.   In fact we have the following proposition.
 \begin{proposition} \label{pr:two_kinds_riggings}  Let $\riem^B$ be a bordered Riemann surface, and let $\psi=(\psi_1,\ldots,\psi_n)$ be in $\qs(S^1,\riem^B)$.  Let $\riem = \riem^B\#_{\psi} \mathbb{D}_0^n$ and $\tilde{\psi}=(\tilde{\psi}_1,\ldots,\tilde{\psi}_n)$ be the
 $n$-tuple of holomorphic extensions to $\mathbb{D}_0$.  Then $\psi \in \rigo(\riem^B)$ if and only if $\tilde{\psi} \in \Oqco(\riem)$.
 \end{proposition}
 \begin{proof}
  Let $H$ be a collared chart with respect to the $i$th boundary curve $C_i$, and let $F$, $G$ and $\zeta_i$ be as in Remark \ref{re:complex_structure_sewn_surface}.  By definition $\psi_i \in \qso(S^1,C_i)$ if and only if $H \circ \psi_i \in \qso(S^1)$
  which holds if and only if the welding map $F$ is in $\Oqco$.  Since $F=\zeta \circ \tilde{\psi}_i$ this proves the claim.
 \end{proof}
 The following Proposition is a consequence of Proposition \ref{pr:mixed_composition_preserves} and Theorem \ref{th:QSo_group}.
 \begin{proposition} \label{pr:mixed_composition_riggings}
  Let $\riem^B_1$ and $\riem^B_2$ be bordered Riemann surfaces, and let $\tau \in \rigo(\riem^B_1)$. Then $f \in \qco(\riem^B_1,\riem^B_2)$ if and only if $f \circ \tau \in \rigo(\riem^B_2)$.
 \end{proposition}

 We now have enough tools to describe the relation between $T_0(\riem^B)$ and $\ttildeop(\riem)$.
 \begin{definition} \label{de:pdb}
  Let
 $\riem^B$ be a bordered Riemann surface, let $\tau \in \rigo(\riem)$ be a fixed rigging, and  let $\riem = \riem^B \#_{\tau} \mathbb{D}_0^n$.  We define
 \begin{align*}
  \pdb: T(\riem^B) & \longrightarrow  \ttildep(\riem) \\
  {[{\riem^B},f,{\riem_1^B}]} & \longmapsto  [\riem,\tilde{f},\riem_1,\tilde{f} \circ \tilde{\tau}].
 \end{align*}
 where $\tilde{\tau}$ is the extension defined by (\ref{eq:rigging_ext}),
 \begin{equation}
 \label{eq:marking_ext}
   \tilde{f}(z) = \begin{cases} f(z),  & z \in \overline{\riem^B} \\ z, & z \in \  \text{cap}, \end{cases}
 \end{equation}
 and $\riem_1 = \riem^B_1 \#_{f \circ \tau} \mathbb{D}_0^n$ is the Riemann surface obtained by sewing caps onto $\riem^B_1$ using $f \circ \tau$.
 \end{definition}
 The map $\tilde{f}$ is quasiconformal, since it is quasiconformal
 on $\riem^B$ and the cap, and is continuous on the seam \cite[V.3]{Lehto-Virtanen}.
 \begin{remark} \label{re:tilde_commutes}
  If $\widetilde{f \circ \tau}$ denotes the holomorphic extension of $f \circ \tau$ as in equation (\ref{eq:rigging_ext}), then $\widetilde{f \circ \tau}= \tilde{f} \circ \tilde{\tau}$.
 \end{remark}

 It was shown in \cite{RS05} that $\pdb$ is invariant under the action of $\db$, and in fact
 $$
 \pdb([\riem^B,f,\riem^B_1])=\pdb([\riem^B,f_2,\riem^B_2]) \Longleftrightarrow
 [\riem^B,f_2,\riem^B_2]=[\rho][\riem^B,f_1,\riem^B_1]
 $$
 for some $[\rho] \in \db$.
 (The reader is warned that the direction of the riggings in \cite{RS05} is opposite to the convention used here).
 Thus $\ttildep(\riem)= T(\riem^B)/\db$ as sets.  Furthermore, the group action by $\db$ is properly discontinuous
 and fixed point free, and the map $\pdb$ is holomorphic with local holomorphic inverses.  Thus $\ttildep(\riem)$ inherits
 a complex structure from $T(\riem^B)$.

 On the other hand, in the refined setting, instead of having a complex
 structure on Teichm\"uller space in the first place, we are trying to construct one.  In the next section, we will reverse the argument above and lift the complex Hilbert manifold structure on $\ttildeop(\riem)$ to
 $T_0(\riem^B)$.  To this end we need the following facts.
 \begin{proposition} \label{pr:pdb_preserves_refined} Let $p = [\riem^B, f, \riem^B_1] \in T(\riem^B)$. Then
  $p \in T_0(\riem^B)$ if and only if $\pdb(p) \in \ttildeop(\riem)$.
 \end{proposition}
 \begin{proof}  Since $\tau \in \rigo(\riem^B)$,  $f \in \qco(\riem^B,\riem^B_1)$ if and only if  $f \circ \tau \in \rigo(\riem_1^B)$  by Proposition \ref{pr:mixed_composition_riggings}. And this holds if and only if $\widetilde{f\circ \tau} \in \Oqco(\riem_1)$ by Proposition \ref{pr:two_kinds_riggings}.  By Remark \ref{re:tilde_commutes},
  $\tilde{f} \circ \tilde{\tau} \in \Oqco(\riem_1)$ which proves the claim.
 \end{proof}
 We now define the map $\Pi_0$ by
 $$\Pi_0 = \Pi|_{T_0(\riem^B)},
 $$
 and as a result of this proposition we have
 \begin{equation}
 \label{eq:Pio}
 \Pi_0  : T_0(\riem^B) \longrightarrow \ttildeop(\riem).
 \end{equation}

 \begin{proposition} \label{pr:db_fpf_and_kernel_is_db}
The action of $\db$ is fixed point free, and for $[\riem^B,f_i,\riem_i^B] \in T_0(\riem^B)$, $i=1,2$, $\pdb_0([\riem^B,f_1,\riem_1^B])=\pdb_0([\riem^B,f_2,\riem_2^B])$ if and only
  if there is a $[\rho] \in \db$ such that $[\rho][\riem^B,f_1,\riem_1^B]=[\riem^B,f_2,\riem_2^B]$.  The map $\Pi : T_0(\riem^B) \to \ttildeop(\riem)$ is onto and thus, as sets, $T_0(\riem^B)/\db$ and $\ttildeop(\riem)$ are in one-to-one correspondence.
 \end{proposition}
 \begin{proof} These claims are all true in the non-refined setting \cite[Lemma 5.1, Theorem 5.6]{RS05}.
  Thus by Proposition \ref{pr:pdb_preserves_refined} they are true in the refined setting.
 \end{proof}
\end{subsection}
\begin{subsection}{Complex Hilbert manifold structure on refined Teichm\"uller
space}  \label{se:complex_structure_refined_Teich}
 Next we describe how to construct the complex structure on $T_0(\riem^B)$.  Let $\riem^B$ be a bordered
 Riemann surface, and let $\tau \in \rigo(\riem^B)$.  Let $\riem$ be the Riemann surface obtained by sewing on caps
 via $\tau$ as in the previous section.

 We define a base $\mathcal{B}$ for a topology on $T_0(\riem^B)$ as follows.
 Recall that $\mathcal{F}$ is the base for $\ttildeop(\riem)$ (Definition
  \ref{de:base_rigged_Teich}).
 \begin{definition} \label{de:base_refined_bordered}
 A set $B \in \mathcal{B}$ if and only if
 \begin{enumerate}
 \item $\pdb_0(B) \in \mathcal{F}$
 \item $\pdb_0$ is one-to-one
 on $B$.
 \end{enumerate}
 \end{definition}

 \begin{theorem} \label{th:refined_Teich_base} The set $\mathcal{B}$ is a base.  With the topology corresponding to $\mathcal{B}$, $\ttildeop(\riem)$
  has the quotient topology with respect to $\pdbo$ and $\db$ is properly discontinuous.
 \end{theorem}
 \begin{proof}
  Let $x \in T_0(\riem^B)$.  We show that there is a $B \in \mathcal{B}$ containing $x$. There is a neighborhood $U$ of $x$ in $T(\riem^B)$ on which $\Pi$ is one-to-one \cite{RS05}.
  Let $U'=\Pi(U)$; this is open in $\widetilde{T}(\riem)$ \cite{RS05}.  By Theorem \ref{th:top in tp}, the set $U' \cap \ttildeop(\riem)$
  is open in $\ttildeop(\riem)$. Thus there is an element $F \subset U' \cap \ttildeop(\riem)$ of the base $\mathcal{F}$ which
  contains $\Pi(x)$.  Since $\left. \Pi \right|_{U}$ is invertible, we can set $B=\left(\left. \Pi \right|_{U}\right)^{-1}(F)$,
  and $B$ is in $\mathcal{B}$ and contains $x$.

  Next, fix $q \in T_0(\riem^B)$ and let $B_1, B_2 \in \mathcal{B}$ contain $q$. We show that the intersection contains an element of $\mathcal{B}$.  Let
  $U \subset \pdbo(B_1) \cap \pdbo(B_2)$ be a set in $\mathcal{F}$ containing $\pdb_0(q)$.  Set $B_3 = \left(\left. \pdbo \right|_{B_1} \right)^{-1}(U)
  \subset B_1 \cap B_2$.  We then have that $\pdbo$ is one-to-one on $B_3$ (since $B_3 \subset B_1$) and $\pdbo(B_3)=U$.  So $B_3 \in \mathcal{B}$.  Thus $\mathcal{B}$ is a base.

  Now we show that $\ttildeop(\riem)$ has the quotient topology with respect to $\pdbo$.  Let $U$ be open in $\ttildeop(\riem)$
  and let $x \in \pdbo^{-1}(U)$.  There is a $B_x \in \mathcal{B}$ containing $x$ such that $\pdbo$ is one-to-one on $B_x$, and $\pdbo(B_x)$ is
  open and in $\mathcal{F}$.  Since $\pdbo(B_x) \cap U$ is open and non-empty (it contains $\pdbo(x)$), there is a $F_x \in \mathcal{F}$ such that $\pdbo(x) \in F_x$
  and $F_x \subset \pdbo(B_x) \cap U$.  By definition $\tilde{B}_x = \left( \left. \pdbo \right|_{B_x} \right)^{-1}(F_x) \in \mathcal{B}$.
  By construction $x \in \tilde{B}_x$ and $\tilde{B}_x$ is open and contained in $U$.  Since $x$ was arbitrary, $\pdbo^{-1}(U)$ is open.

  Let $U \in \ttildeop(\riem)$ be such that $\pdbo^{-1}(U)$ is open.  Let $x \in U$ and $y \in \pdbo^{-1}(U)$ be such that $\pdbo(y)=x$.
  There is a $B_y \in \mathcal{B}$ such that $y \in B_y \subset \pdbo^{-1}(U)$.  So $\pdbo(B_y) \subset U$ and $x \in \pdbo(B_y)$.
  Since $B_y$ is in $\mathcal{B}$, $\pdbo(B_y) \in \mathcal{F}$, so $\pdbo(B_y)$ is open.  Since $x$ was arbitrary, $U$ is open.  This
  completes the proof
  that $\ttildeop(\riem)$ has the quotient topology.

  Finally, we show that $\db$ acts properly discontinuously on $T_0(\riem^B)$.  Let $x \in T_0(\riem^B)$.
  By \cite[Lemma 5.2]{RS05}, $\db$ acts properly discontinuously
  on $T(\riem^B)$ in its topology.  Thus there is an open set $U \subset T(\riem^B)$ containing $x$ such that
  $g(U) \cap U$ is empty for all $g \in \db$, and on which $\pdb$ is one-to-one.  Furthermore, $\pdb(U)$ is open in
  $\ttildep(\riem)$ since $\pdb$ is a local homeomorphism \cite{RS05}.  By Theorem \ref{th:top in tp},
  $\pdb(U) \cap \ttildeop(\riem)$ is open in $\ttildeop(\riem)$, so there exists an $F \in \mathcal{F}$ such that
  $F \subset \pdb(U) \cap \ttildeop(\riem)$ and $\pdb(x) \in F$ (note that $\pdb(x) \in \ttildeop(\riem)$ by
  Proposition \ref{pr:pdb_preserves_refined}).  So $W=\left( \left. \pdb \right|_{U} \right)^{-1}(F)$ is in $\mathcal{B}$
  by definition, and contains $x$.  In particular $W$ is open, and since $W \subset U$ by construction,
  $g(W) \cap W$ is empty for all $g \in \db$.  This completes the proof.
 \end{proof}
 \begin{corollary} \label{co:Teichzero_Haus_2nd}
  With the topology defined by $\mathcal{B}$, $T_0(\riem^B)$ is
  Hausdorff and separable.
 \end{corollary}
 \begin{proof}  Let $x, y \in T_0(\riem^B)$, $x \neq y$.  If $\pdbo(x) \neq \pdbo(y)$, then since $\ttildeop(\riem)$ is Hausdorff by Theorem \ref{th:tilde_top_base}, there are disjoint open sets $F_x, F_y \in \mathcal{F}$ such that $\pdbo(x) \in F_x$ and
 $\pdbo(y) \in F_y$.
 Since $\mathcal{B}$ is a base there are sets $B_x, B_y \in \mathcal{B}$ such that $x \in B_x$, $y \in B_y$, $\pdbo(B_x) \subset F_x$ and $\pdbo(B_y)
 \subset F_y$.  Thus $B_x$ and $B_y$ are disjoint.

 Now assume that $\pdbo(x)= \pdbo(y)$.  Thus there is a
 non-trivial $[\rho] \in \db$ such that $[\rho]x =y$.  Since by Theorem \ref{th:refined_Teich_base} $\db$ acts properly discontinuously
 there is an open set $V$ containing $x$ such that $[\rho]V \cap
 V$ is empty; $[\rho]V$ is open and contains $y$.  This completes
 the proof that $T_0(\riem^B)$ is Hausdorff.

 To see that $T_0(\riem^B)$ is separable, let $\mathfrak{A}$ be a countable
 dense subset of $\ttildeop(\riem)$.  Define $\mathfrak{B} = \{ p \in T_0(\riem^B)\,:\, \Pi(p) \in \mathfrak{A} \}$.  Since $\db$ is countable,
 $\mathfrak{B}$ is countable.  To see that $\mathfrak{B}$ is dense, observe that
 if $U$ is open in $T_0(\riem^B)$ then, since $\db$ acts properly discontinuously
 by Theorem \ref{th:refined_Teich_base}, there is a $V \subseteq U$ on which $\Pi$ is a homeomorphism onto its image.  So there is a $q \in \mathfrak{A} \cap \Pi(V)$, and
 thus for a local inverse $\Pi^{-1}$ on $\Pi(V)$ we can set $p=\Pi^{-1}(q) \in V \cap \mathfrak{B} \subseteq U \cap \mathfrak{B}$.  This completes the proof.
 \end{proof}
 \begin{remark}  It can also be shown that $T_0(\riem^B)$ is second countable.
 To see this, let $\mathcal{F}'$
 be a countable base for $\ttildeop(\riem)$. Such a base exists
 by Remark \ref{re:tildeTP_second_countable}.  Let $\mathcal{B}'=
 \{B \in \mathcal{B} \,:\, \pdbo(B) \in \mathcal{F}'\}$.  It is elementary
 to verify that $\mathcal{B}'$ is a base.  The fact that $\mathcal{B}'$ is countable follows from the facts that $\mathcal{F}'$ is countable and
 $\db$ is countable. Indeed, for each
 element $F$ of $\mathcal{F}'$ we can choose an element $B_F$ of $\mathcal{B}'$.
 Each $B$ in $\mathcal{B}'$ is $[\rho] B_F$ for some $F \in \mathcal{F}'$
 and $\rho \in \db$.
 \end{remark}

Using this base, we now define the charts on $T_0(\riem^B)$ that will give it a complex Hilbert space structure.
For any $x \in T_0(\riem^B)$, let $B$ be in the base $\mathcal{B}$; therefore
  $F=\pdb(B)$ is in the base $\mathcal{F}$ of $\ttildeop(\riem)$ (see Definition \ref{de:base_rigged_Teich}).  From Definition \ref{de:riggedcharts}
  there is the chart $G^{-1} : F \rightarrow \mathbb{C}^d \otimes (\Oqco)^n$,
  where $d = 3g-3+n$ is the dimension of $T(\riem)$ and $n$ is the number
  of boundary curves of $\riem^B$.
\begin{definition}[Charts for $T_0(\riem^B)$]
\label{de:TB_charts}
Given $x \in B \subset T_0(\riem^B)$ as above, we define the chart
$$
S : B \longrightarrow \mathbb{C}^d \otimes (\Oqco)^n
$$
by $S=G^{-1} \circ \pdb_0$.
\end{definition}
Note that to get a true chart into a Hilbert space we need to compose $S$ with maps $\chi : \Oqco \to A_1^2(\mathbb{D}) \oplus \mathbb{C}$ (see (\ref{eq_chidefinition}) and Theorem \ref{th:Oqco_open_in_Oqc}) as in the proof of  Theorem \ref{th:top in tp}.

 \begin{theorem} \label{th:Hilbert_manifold_on_TB} The refined \teich space
  $T_0(\riem^B)$ with charts given in the above definition is a complex Hilbert manifold.  With this given complex structure,
  $\pdb_0$ is locally biholomorphic in the sense that for every
  point $x \in T_0(\riem^B)$ there is a neighborhood $U$ of $x$
  such that $\pdb_0$ restricted to $U$ is a biholomorphism onto its
  image.
 \end{theorem}
 \begin{proof}
  By Corollary \ref{co:Teichzero_Haus_2nd}, we need only to
  show that $T_0(\riem^B)$ is locally homeomorphic to a Hilbert
  space, and exhibit an atlas of charts with holomorphic transition
  functions. Since Definition \ref{de:TB_charts} defines a chart for any $x \in T_0(\riem^B)$, the set of such charts clearly covers
  $T_0(\riem^B)$. The maps $S$ are clearly homeomorphisms, since $G$'s
  are biholomorphisms by
  Theorem \ref{th:tilde_top_base} and $\pdb_0$'s are local homeomorphisms by the definition of the topology on $\ttildeop(\riem)$.

  Assume that two such charts $(S,B)$ and $(S',B')$
  have overlapping domains.  We show that $S' \circ S^{-1}$ is
  holomorphic on $B \cap B'$.  Let $x \in B \cap B'$.  Since
  $\mathcal{B}$ is a base, there is a $B_1 \in B \cap B'$
  containing $x$.  So $\pdb$ is one-to-one on $B_1$; note also
  that the determination of $\pdb^{-1}$ on $\pdb(B_1)$ agrees with
  those on $\pdb(B)$ and $\pdb(B')$.  So $S' \circ S^{-1} = (G')^{-1}
  \circ \pdb \circ \pdb^{-1} \circ G = (G')^{-1} \circ G^{-1}$
  which is holomorphic by Theorem \ref{th:tilde_top_base}.  The same
  proof applies to ${S \circ S}'^{-1}$.
 \end{proof}

 The construction of the Hilbert manifold structure on $T_0(\riem^B)$ made use of an arbitrary choice of a \textit{base rigging} $\tau \in \rigo(\riem^B)$, but in fact the resulting complex structure is independent of this choice.  We will show a slightly stronger result.
 If one considers a base Riemann surface together with a base rigging $(\riem_b^B,\tau_b)$ to define a base point, then
 the \textit{change of base point} to another such pair $(\riem^B_a,\tau_a)$ is a biholomorphism. We proceed by first examining the change of base point map for $\ttildeop(\riem)$.

 Fix two punctured Riemann surfaces $\riem_a$ and $\riem_b$ of the same topological type, and let $\alpha: \riem_a \to \riem_b$ be a quasiconformal map.
 The change of base point map $\alpha^*$ is defined by
 \begin{align}
 \label{eq:tp_base_point}
  \alpha^*: \ttildeop(\riem_b) & \longrightarrow  \ttildeop(\riem_a) \\
   {[\riem_b,g, \riem_1, \phi]} & \longmapsto [\riem_a, g \circ \alpha, \riem_1, \phi].  \nonumber
 \end{align}
 This is completely analogous to the usual change of base point biholomorphism for the \teich space $T(\riem)$ (see the paragraph following Theorem \ref{th:Schiffer}). From the general definition of the Schiffer variation map in (\ref{Schiffer_map2}), it is worth noting that the coordinates for $\ttildeop(\riem_0)$ as defined in (\ref{eq:Gdefinition}) actually have this change of base point biholomorphism built in.
 From this observation we easily obtain the following theorem.
 \begin{theorem} \label{th:change_base_Tp}
 The change of base point map in (\ref{eq:tp_base_point}) is a biholomorphism.
 \end{theorem}
 \begin{proof} The map $\alpha^*$ has inverse $(\alpha^*)^{-1} = (\alpha^{-1})^*$ and hence is a bijection.
 Consider the points $p = [\riem_b,g, \riem_1, \phi]$ and  $q = \alpha^*(p) = [\riem_a, g \circ \alpha, \riem_1, \phi]$. One can choose coordinates, as in equation (\ref{eq:Gdefinition}), for neighborhoods of $p$ and $q$  which use the same Schiffer variation on $\riem_1$, and thus the same map $\nu^{\epsilon}$. In terms of these local coordinates, the map $\alpha^*$ is the identity map and so is certainly holomorphic. The same argument shows that $(\alpha^{-1})^*$ is holomorphic and hence $\alpha^*$ is biholomorphic.
 \end{proof}
 The next task is to relate the preceding change of base point map to the one between bordered surfaces.
 Let $\riem_b^B$ and $\riem_a^B$ be bordered Riemann surfaces of type $(g,n)$ and fix riggings $\tau_b \in \rigo(\riem_b^B)$
 and $\tau_a \in \rigo(\riem_a^B)$. Then there exists $\rho \in \qco(\riem_a^B,\riem_b^B)$ such that $\rho \circ \tau_a=\tau_b$. In fact one can prove a stronger statement \cite[Corollary 4.7 and Lemma 4.17]{RS05}: Given any quasiconformal map $\rho': \riem_a^B \to \riem_b^B$, there exists $\rho \in \qco(\riem_a^B,\riem_b^B)$ such that $\rho \circ \tau_a=\tau_b$ and $\rho$ is homotopic (not rel boundary) to $\rho'$. The map $\rho'$ is obtained by deforming $\rho$ in a neighborhood of the boundary curves so as to have the required boundary values.

 For such a $\rho$, define the change of base point map
 \begin{align} \label{eq:change_of_base_TB}
   \rho^*:T^B_0(\riem^B_b)  & \longrightarrow  T^B_0(\riem_a^B) \\
   [\riem_b^B,f,\riem^B_1] & \longmapsto  [\riem_a^B,f \circ \rho,\riem^B_1] \nonumber
  \end{align}
 which is just the usual change of base point map restricted to the refined \teich space, together with the added condition of compatibility with the fixed base riggings.

Let $\riem_b$ and $\riem_a$ be the punctured surfaces obtained from $\riem_b^B$ and $\riem_a^B$ by sewing on caps via $\tau_b$ and $\tau_a$ respectively.
Given $\rho$ as above we have its quasiconformal extension $\tilde{\rho} : \riem_a \to \riem_b$ defined by
$$
\tilde{\rho} =
\begin{cases} \rho & \text{on } \overline{\riem_a^B} \\
\text{id} & \text{on } \mathbb{D}
\end{cases}
$$
as in (\ref{eq:marking_ext}).  Let $\tilde{\rho}^*$ be the change of base point biholomorphism as in (\ref{eq:tp_base_point}).

 \begin{lemma} \label{le:change_of_base_commutes}
  Let $(\riem_b^B,\tau_b)$, $(\riem_a^B,\tau_a)$, $\rho$, $\rho^*$, $\tilde{\rho}$ and $\tilde{\rho}^*$ be as above.
  Then the diagram
  $$
   \xymatrix{
   T^B_0(\riem_b^B) \ar[r]^{\rho_*} \ar[d]_{\Pi_0} & T^B_0(\riem_a^B) \ar[d]^{\Pi_0} \\
   \ttildeop(\riem_b)  \ar[r]^{\tilde{\rho}^*} &
   \ttildeop(\riem_a)
   }
  $$
  commutes.
 \end{lemma}
 \begin{proof}
  Let $[\riem_b^B,f,\riem_1^B] \in T^B_0(\riem_b^B)$.  We have that
  \begin{align*}
   \Pi_0 \circ \rho^* ([\riem_b^B,f,\riem_1^B]) &= \Pi_0 ([\riem_a^B, f\circ \rho,\riem_1^B])  \\
   &= [\riem_a^B \#_{\tau_a} \mathbb{D}, \widetilde{f \circ \rho}, \riem_1^B \#_{f \circ \rho \circ \tau_a} \mathbb{D},\widetilde{f \circ \rho \circ \tau_a}] \\
   & =  [\riem_a^B \#_{\tau_a} \mathbb{D}, \widetilde{f \circ \rho}, \riem_1^B \#_{f \circ \tau_b} \mathbb{D},\widetilde{f \circ \tau_b}] \\
   &= [\riem_a, \widetilde{f \circ \rho}, \riem_1,\widetilde{f \circ \tau_b}]
  \end{align*}
  since $\rho \circ \tau_a=\tau_b$.
  On the other hand
  \begin{equation*}
   \tilde{\rho}^* \circ \Pi_0 ([\riem_b^B,f,\riem_1^B])  =  \tilde{\rho}^*([\riem_b^B \#_{\tau_b} \mathbb{D}, \tilde{f},
   \riem_1^B \#_{f \circ \tau_b} \mathbb{D}, \widetilde{f \circ \tau_b}]) = [\riem_b,\tilde{f} \circ \tilde{\rho},\riem_1,
   \widetilde{f \circ \tau_b}].
  \end{equation*}
  The claim follows from the fact that $\widetilde{f \circ \rho} = \tilde{f} \circ \tilde{\rho}$ (Remark \ref{re:tilde_commutes}).
 \end{proof}

 Theorem \ref{th:Hilbert_manifold_on_TB}, Theorem \ref{th:change_base_Tp}, and Lemma \ref{le:change_of_base_commutes} immediately imply the following theorem.
 \begin{theorem} \label{th:change_of_base_TB}
  Let $(\riem_b^B,\tau_b)$ and $(\riem_a^B,\tau_a)$ be a pair of rigged bordered Riemann surfaces, with $\tau_b \in \rigo(\riem_b^B)$
  and $\tau_a \in \rigo(\riem_a^B)$.  Let $\rho \in \qco(\riem_1^B,\riem^B_b)$ satisfy $\rho \circ \tau_a=\tau_b$.  Then
  the change of base point map $\rho^*$ given by equation (\ref{eq:change_of_base_TB}) is a biholomorphism.
 \end{theorem}

 \begin{corollary} \label{th:independence}
  The complex Hilbert manifold structure on $T_0(\riem^B)$ is independent of the choice of rigging $\tau \in \qso(\riem^B)$.
 \end{corollary}
 \begin{proof}
  Apply Theorem \ref{th:change_of_base_TB} with $\riem_b^B=\riem_a^B=\riem^B$.
 \end{proof}
 \begin{theorem} \label{th:Teich_inclusion_holo}
  The inclusion map from $T_0(\riem^B)$ to $T(\riem^B)$ is holomorphic.
 \end{theorem}
 \begin{proof}
  Since $\pdb$ has local holomorphic inverses the inclusion map from $T_0(\riem^B)$ to
  $T(\riem^B)$ can be locally written as $\pdb^{-1} \circ \iota
  \circ \pdbo$ where $\iota:\ttildeop(\riem) \rightarrow \ttildep(\riem)$ is
  inclusion.  The theorem follows from the facts that $\pdb^{-1}$ and $\pdbo$ are holomorphic and
  $\iota$ is holomorphic by Theorem \ref{th:top in tp}.
 \end{proof}
\end{subsection}
\begin{subsection}{Rigged moduli space is a Hilbert manifold}
\label{se:Friedan_Shenker_refined}
 In this section we show that the rigged moduli space of conformal field theory originating with Friedan and Shenker \cite{FriedanShenker}, with riggings chosen as in this paper, have Hilbert
 manifold structures.

 First we define the moduli spaces.  There are two models, which we will
 refer to as the border and the puncture model. These models are defined as follows:
 \begin{definition} Fix integers $g$ and $n$, $2g-2+n >0$.
\begin{enumerate}
\item  The border model of the refined rigged moduli space is
  \[ \mathcal{M}_0^B(g,n) = \{ (\riem^B,\psi) \,:\, \riem^B \text{ bordered of type } \ (g,n), \ \
     \psi \in \rigo(\riem^B)  \} / \sim  \]
  where $(\riem^B_1,\psi) \sim (\riem^B_2,\phi)$ if and only if there is a
  biholomorphism $\sigma:\riem^B_1 \rightarrow \riem^B_2$ such that
  $\phi = \sigma \circ \psi$.

\item   The puncture model of the rigged moduli space is
  \[ \mathcal{M}_0^P(g,n) = \{ (\riem,\psi) \,:\, \riem \text{ punctured of type }  (g,n), \ \  \psi \in \Oqco(\riem)  \} /\sim \]
   where $(\riem_1,\psi) \sim (\riem_1,\phi)$ if and only if
   there is a biholomorphism $\sigma:\riem_1 \rightarrow \riem_2$ such that
  $\phi = \sigma \circ \psi$.
\end{enumerate}
 \end{definition}
The puncture and border models (but with different classes of riggings) were used by \cite{Vafa} and \cite{Segal} respectively, in the study of conformal field theory. It was understood from their inception that these rigged moduli spaces are in bijective correspondence, as can be seen by cutting and sewing caps. However, one needs to careful about the exact classes of riggings used to make this statement precise. Replacing ``bijection'' with ``biholomorphism'' in this statement of course requires the careful construction of a complex structure on at least one of these spaces. It was shown in \cite{RS05} that these two moduli spaces are
 quotient spaces of $T(\riem^B)$ by a fixed-point-free properly
 discontinuous group, and thus inherit a complex
 Banach manifold structure from $T(\riem^B)$. Similarly, we will demonstrate that the refined
 rigged moduli spaces inherits a complex Hilbert manifold structure from
 $T_0(\riem^B)$. We first need to show that the action of $\pmodi(\riem^B)$ defined by (\ref{eq:MCG_action}) is fixed
 point free and properly discontinuous.

  \begin{theorem} \label{th:pmodi_prop_disc_biholo}
  The modular group $\pmodi(\riem^B)$ acts properly
  discontinuously and fixed-point-freely on $T_0(\riem^B)$.  The
  action of each
  element of $\pmodi(\riem^B)$ is a biholomorphism of $T_0(\riem^B)$.
 \end{theorem}
 \begin{proof}
 Recall that $\db(\riem^B)$ preserves $T_0(\riem^B)$ by Lemma \ref{le:MCG_reduced_action}.
 By \cite[Lemma 5.2]{RS05}, $\db(\riem^B)$ acts properly
 discontinuously and fixed-point freely on $T(\riem^B)$.  Thus
 $\db(\riem^B)$ acts fixed-point freely on $T_0(\riem^B)$.  Now let $x \in T_0(\riem^B)$. There is a neighborhood $U$ of $x$ in $T(\riem^B)$ such that $[\rho]U \cap U$
 is empty for all $[\rho] \in \db(\riem^B)$.  Clearly $V = U \cap T_0(\riem^B)$ has
 the same property, and is open in $T_0(\riem^B)$ by Theorem \ref{th:Teich_inclusion_holo}.

 Each element $[\rho] \in \pmodi(\riem^B)$ is a biholomorphism of $T_0(\riem^B)$, by observing that $\rho \circ \tau = \tau$ and applying
 Theorem \ref{th:change_of_base_TB}.
 \end{proof}

 We now show that the rigged moduli spaces are Hilbert manifolds.
 Let $\riem^B$ be a fixed bordered Riemann surface of type $(g,n)$
 and let $\tau \in \rig(\riem^B)$ be a fixed rigging.  Define the
 mapping
 \begin{align*}
   P: T(\riem^B) & \longrightarrow  \mathcal{M}^B(g,n) \\
   {[\riem^B,f,\riem^B_1]} & \longmapsto  (\riem_1^B,f \circ \tau)
 \end{align*}
 where $f \circ \tau = (f \circ \tau_1,\ldots, f \circ \tau_n)$.
 Note that this map depends on the choice of $\riem^B$ and $\tau$.
 If we choose $\tau \in \rigo(\riem^B)$, we have the map
 \[  P_0 = \left. P \right|_{T_0(\riem^B)}.  \]
 It follows immediately from Proposition \ref{pr:mixed_composition_riggings}
 that $P_0$ maps into $\mathcal{M}^B_0(g,n)$.
 \begin{theorem} \label{th:P0_surjection_etc} Given any $p,q \in T_0(\riem^B)$,
  $P_0(p)=P_0(q)$ if and only if $q = [\rho] p$ for some $[\rho] \in \pmodi(\riem^B)$.  Moreover, $P_0$ is a surjection onto $\mathcal{M}_0^B(g,n)$.
 \end{theorem}
 \begin{proof}
  All of these claims hold in the non-refined setting by \cite[Theorem 5.2]{RS05}.  Thus the first claim follows immediately.  It was already observed that $\pi_0$ maps
  into $\mathcal{M}^B_0(g,n)$.  To show that $\pi_0$ is surjective,
  observe that by \cite[Theorem 5.2]{RS05}, for any $[\riem^B_1,\psi] \in
  \mathcal{M}^B_0(g,n)$ there is a $[\riem^B,f_1,\riem^B_*] \in T(\riem^B)$
  such that $[\riem^B_*,f_1 \circ \tau]=[\riem^B_1, \psi]$.  By composing
  with a biholomorphism we can assume that $\riem^B_*=\riem^B_1$ and $f_1
  \circ \tau = \psi$.  Thus $f_1 = \psi \circ \tau^{-1}$.  Since
  for $i=1,\ldots,n$ we have $\psi_i \circ \tau_i^{-1} \in \qso(\partial_i \riem^B,\partial_i \riem^B_1)$ by Proposition \ref{pr:composition_preserves_qso_surfaces}, $f_1 \in \qco(\riem^B,\riem^B_1)$.
  Thus $[\riem^B,f_1,\riem^B_1] \in T_0(\riem^B)$ and $P_0([\riem^B,f_1,\riem^B_1])=[\riem^B_1,\psi]$, which completes the proof.
 \end{proof}
 This shows that $T_0(\riem^B)/\pmodi(\riem^B)$ and $\mathcal{M}_0^B(g,n)$
 are bijective.  They are also biholomorphic.
 \begin{corollary}  The rigged moduli space $\mathcal{M}^B(g,n)$ is a Hilbert manifold
  and the map $P_0$ is holomorphic and possesses local holomorphic inverses.
  The Hilbert manifold structure is independent of the choice of base
  surface $\riem^B$ and rigging $\tau$.
 \end{corollary}
 \begin{proof}
  This follows immediately from Theorem \ref{th:P0_surjection_etc}, the fact that $\pmodi(\riem^B)$ acts
  fixed-point freely and properly discontinuously by biholomorphisms
  (Theorem \ref{th:pmodi_prop_disc_biholo}), and the fact that the
  complex structure on $T_0(\riem^B)$ is independent of the choice of
  base rigging.
 \end{proof}

 It was shown in \cite{RS05} that the border and puncture models of
 the rigged moduli space are in one-to-one correspondence, and that
 the puncture model can be obtained as a natural quotient of $\ttildeop(\riem)$.  Those results pass immediately to the refined
 setting, with only very minor changes to the proofs (much as above).
 We will simply summarize the results here.
 Let $\riem$ be a punctured Riemann surface of type $(g,n)$.
 Denote by $\operatorname{PModP}(\riem)$ the modular group of quasiconformal maps
 $f:\riem \rightarrow \riem$ modulo the quasiconformal maps
 homotopic to the identity rel boundary.  Elements $[\rho]$
 of $\operatorname{PModP}(\riem)$ act on $\tilde{T}_0(\riem)$
 via $[\rho][\riem,f_1,\riem_1,\psi]=[\riem,f_1 \circ \rho,
 \riem_1,\psi]$.
 Define the projection map
 \begin{align*}
  Q:\ttildeop(\riem) & \longrightarrow  \mathcal{M}_0^P(g,n) \\
  {[\riem,f,\riem_1,\psi]} & \longmapsto  [\riem_1,\psi].
 \end{align*}
 Finally, define the map
 \begin{align*}
  \mathcal{I}:\mathcal{M}^P(g,n) &\longrightarrow \mathcal{M}^B(g,n) \\
  {[\riem,\phi]} & \longmapsto  [\riem \backslash \overline{\phi_1(\mathbb{D})
  \cup \cdots \cup \phi_n(\mathbb{D})}, \left.\phi \right|_{S^1}].
 \end{align*}

 \begin{theorem}  The moduli spaces $\mathcal{M}^P(g,n)$ and $\mathcal{M}^B(g,n)$ are in one-to-one correspondence under the
 bijection $\mathcal{I}$.  Thus $\mathcal{M}^P(g,n)$ can be endowed
 with a unique Hilbert manifold structure so that $\mathcal{I}$
 is a biholomorphism.  The map $Q$ satisfies
 \begin{enumerate}
  \item $Q(p)=Q(q)$ if and only if
 there is a $[\rho] \in \operatorname{PModP}(\riem)$ such that $[\rho]p=[q]$
  \item $Q$ is surjective,
  \item $Q$ is holomorphic, and possesses a local holomorphic inverse
 in a neighborhood of every point.
 \end{enumerate}
 \end{theorem}

\end{subsection}
\end{section}

\end{document}